\documentclass[a4paper,11pt]{amsart}
\usepackage{quiver}
\usepackage{pdflscape}

\usepackage{amsthm,amsmath,amssymb,stmaryrd,enumerate,euscript,mathrsfs,a4,array, microtype, stmaryrd, url}
\usepackage{multirow}
\usepackage{arydshln}
\usepackage{enumitem}
\usepackage{pictexwd,dcpic,graphicx,cite}
\usepackage[export]{adjustbox}
\usepackage{subcaption}
\usepackage[latin1]{inputenc}
\usepackage{indentfirst}
\usepackage{calligra}
\usepackage{graphicx}
\usepackage{amssymb}
\usepackage{fancyhdr}
\usepackage{amsmath}
\usepackage{latexsym}
\usepackage{amsthm}
\usepackage[all]{xy} 
\usepackage{amscd}
\usepackage{hyperref}
\usepackage{tikz-cd}
\usepackage{tikz}
\usetikzlibrary{arrows.meta}
\usetikzlibrary{nfold}
\usepackage{thmtools}
\usepackage{cleveref}
\usepackage[a4paper,top=3cm,bottom=3cm,inner=2.5cm,outer=2.5cm]{geometry}

\usepackage{tikz}
\usetikzlibrary{positioning,decorations.markings,arrows.meta,calc,fit,quotes,cd,math,arrows,backgrounds,shapes.geometric}
\usepackage{amsmath,mathtools}
\usepackage{xparse}

\makeatletter
\newenvironment{smallermatrix}[1][c]
{\null\,\vcenter\bgroup
  \Let@\restore@math@cr\default@tag
  \baselineskip0pt \lineskip0.4pt \lineskiplimit0pt
  \ialign\bgroup\if#1l\else\hfil\fi$\m@th\scriptstyle##$\if#1r\else\hfil\fi&&\thickspace\hfil
  $\m@th\scriptstyle##$\hfil\crcr
}{%
  \crcr\egroup\egroup\,%
}
\makeatother

\NewDocumentCommand{\ts}{O{c} e{^?_}}{
  \begin{smallermatrix}[#1]
  \mathstrut\IfValueT{#2}{#2} \\
  \mathstrut\IfValueT{#3}{#3} \\
  \mathstrut\IfValueT{#4}{#4}
  \end{smallermatrix}%
}

\makeatletter
\def\hlinewd#1{%
\noalign{\ifnum0=`}\fi\hrule \@height #1 \futurelet
\reserved@a\@xhline}
\makeatother

\graphicspath{ {images/} }

\newtheorem{theorem}{Theorem}[section]
\newtheorem*{theorem*}{Theorem}
\newtheorem{prop}[theorem]{Proposition}

\newtheorem{lemma}[theorem]{Lemma}
\newtheorem{cor}[theorem]{Corollary}

\newtheorem{conjecture}[theorem]{Conjecture}

\theoremstyle{definition} 
\newtheorem{defn}[theorem]{Definition}
\newtheorem{defn/propn}[theorem]{Definition/Proposition}

\theoremstyle{remark}
\newtheorem{rmk}[theorem]{Remark}
\newtheorem{term}[theorem]{Terminology}
\newtheorem{eg}[theorem]{Example}
\newtheorem{egs}[theorem]{Examples}

\newtheorem*{remark*}{Remark}
\newtheorem*{remarks*}{Remarks}

\newtheorem*{notation*}{Notation}
\newtheorem*{convention*}{Convention}

\crefname{appendix}{appendix}{Appendices}
\Crefname{appendix}{Appendix}{Appendices}

\newcommand{\cat}{\mathcal} 

\newcommand{\catA}{\cat{A}}

\newcommand{\ca}{\cat{A}}
\newcommand{\cb}{\cat{B}}
\newcommand{\cc}{\cat{C}}
\newcommand{\cd}{\cat{D}}
\newcommand{\ce}{\cat{E}}

\newcommand{\cx}{\cat{X}}

\newcommand{\ck}{\mathcal{K}}

\newcommand{\tcat}{\mathbf}

\newcommand{\Gray}{\tcat{Gray}}
\newcommand{\GCat}{\tcat{G}\text{-}\tcat{Cat}}

\newcommand{\GCatm}{\mathbb{L}\tcat{ax}}

\newcommand{\GCatmp}{\mathbb{P}\tcat{sd}}
\newcommand{\GCatmpshl}{\mathbb{L}\tcat{Sharp}}
\newcommand{\GCatmpshr}{\mathbb{R}\tcat{Sharp}}
\newcommand{\GCatmpsh}{\mathbb{S}\tcat{harp}}

\newcommand{\Coalg}[1]{#1\text{-}\tcat{Coalg}}
\newcommand{\Free}{\tcat{Free}}
\newcommand{\Ho}{\textnormal{Ho}}
\newcommand{\leftsharp}{\otimes_\sharp}

\newcommand{\midfourcat}{moderate 4-category}

\newcommand{\aeq}{\textrm{ae}} 
\newcommand{\aexi}{\Theta} 

\newcommand{\twocat}{2\text{-}\tcat{Cat}}
\newcommand{\RGph}{\tcat{R}\text{-}\tcat{Gph}}
\newcommand{\Gph}{\tcat{Gph}}
\newcommand{\Cat}{\tcat{Cat}}
\newcommand{\Set}{\mathbf{Set}}

\newcommand{\lax}{\mathbf{Lax}}
\newcommand{\psd}{\mathbf{Psd}}
\newcommand{\adje}{{AE}}

\newcommand{\rev}{\mathit{rev}}

\def\black{\color{blue}}
\def\black{\color{black}}

\makeatletter
\protected\def\xvcenter{%
  \hbox\bgroup$\everyvbox{\everyvbox{}\aftergroup\m@th\aftergroup$\aftergroup\egroup}%
  \vcenter
}

\DeclareRobustCommand{\midscript}[1]{
  \mathchoice{\mid@script\scriptstyle{#1}}
    {\mid@script\scriptstyle{#1}}
    {\mid@script\scriptscriptstyle{#1}}
    {\mid@script\scriptscriptstyle{#1}}
}
\newcommand{\mid@script}[2]{
  \vcenter{\hbox{$\m@th#1#2$}}
}

\makeatletter

\makeindex
\title{Semi-strict 4-categories as enriched categories}
\author{John Bourke}
\address{Department of Mathematics and Statistics, Masaryk University, Kotl\'a\v rsk\'a 2, Brno 61137, Czech Republic}
\email{bourkej@math.muni.cz}

\author{Gabriele Lobbia}
\address{Departmental of Mathematics, University of Milan, Via Saldini, 50, 20133 Milano, Italy}
\email{gabriele.lobbia@unimi.it}

\begin{document}
\maketitle

\begin{abstract}
We introduce several flavours of semi-strict $4$-category as enriched categories.   This is done using the framework of enrichment over skew monoidal categories of Gray-categories and their associated multicategories, as developed in our previous work.  We give the axioms for semi-strict $4$-categories in elementary terms and describe a wide range of examples.  In particular, we answer the question of what exactly Gray-categories and the higher transformations between them form --- namely, a left-sharp $4$-category $\GCat^4$.  In addition, we describe examples arising from braided monoidal $2$-categories, monoidal Gray-categories (herein defined) and Gray-categories with structure.  We use the concepts introduced here and the relationships between them to unify various concepts in the literature, such as Crans' $4$-teisi.  

In the second part of the paper, we investigate homotopical aspects.  In particular, we classify the equivalence $n$-cells in the left-sharp $4$-category $\GCat^4$ for each $n$.   Using this, we show that the left-sharp skew monoidal structure on the category of Gray-categories induces a symmetric monoidal closed structure on the homotopy category of Gray-categories.
\end{abstract}

%
%

\tableofcontents

\section{Introduction}

The Gray-tensor product $\otimes_p$ of $2$-categories has many useful features.  One of these is that categories enriched in it, called Gray-categories, provide a framework for $3$-category theory which is theoretically powerful, sufficiently general and yet sufficiently simple to be amenable to calculations.  Let us flesh out the key points which make the theory appealing.
\begin{enumerate}
\item \emph{(Powerful theory.)} The Gray-tensor product forms a symmetric monoidal closed structure $(\twocat,\otimes_p,1)$ on the category of $2$-categories.  This means that the full power of classical enriched category theory \cite{KellyG:bascec} --- including enriched adjunctions, weighted limits and so on --- applies to Gray-categories as a special case.
\item \emph{(Sufficient generality.)} Each weak $3$-category (aka tricategory) is equivalent to a Gray-category, though not necessarily to a strict $3$-category \cite{Gordon1995Coherence}.  For this reason, Gray-categories are often called \emph{semi-strict} $3$-categories and provide a sufficiently general framework.
\item \emph{(Ease of calculation.)} Gray-categories are almost strict, involving only a single additional $3$-dimensional datum which accounts for the failure of the middle $4$-interchange for $2$-cells.  Since their homs are strict $2$-categories, one can use traditional 2-categorical pasting diagrams for calculations.  Such calculations can also be carried out using the powerful $3$-dimensional graphical calculus \cite{Hummon, Barrett}.
\item \emph{(Natural examples.)} Thanks to (1) above, there is a Gray-category of $2$-categories, $2$-functors, \emph{pseudonatural transformations} and modifications on which many further examples are based.  There is also a close connection between pseudonaturality and braidings, which has the effect that doubly degenerate Gray-categories capture braided strict monoidal categories.  More recently it has been shown that $3$-dimensional defect TQFT's lead naturally to Gray-categories with duals \cite{Carqueville}.
\end{enumerate}

The need for a $4$-dimensional analogue of Gray-categories grew out of research on braided monoidal $2$-categories in the 1990s.  With motivations coming from theoretical physics, braided monoidal 2-categories were first defined by Kapranov and Voevodsky  \cite{Kapranov}.  The definition was then refined by several authors \cite{Baez1996,Day}, with the final axioms added by Sjoerd Crans \cite{Crans1998Generalized}.  It was recognised by Baez and Neuchl \cite{Baez1996} that braided monoidal $2$-categories should be \emph{semi-strict 4-categories} with one object and one arrow, for some as-yet-undefined notion of semi-strict 4-category.

Accordingly, Crans set about developing such a theory.   In particular, he described a monoidal structure on the category of Gray-categories \cite{Crans1999A-tensor} and defined his notion of semi-strict $4$-category, called a $4$-tas \cite{Crans2000On}, to be a category enriched in (the pseudo-variant) of this monoidal structure.  However, there were a couple of unappealing features here.  Perhaps a minor one was that braided monoidal $2$-categories were more general than his $4$-teisi.  A more fundamental issue is that the monoidal structure was not closed on either side, which meant that the category of Gray-categories was not enriched over itself and, in particular, not a $4$-tas.  It was furthermore recognised that no monoidal biclosed structure on Gray-categories suitable for a theory of semi-strict $4$-categories could exist \cite{Bourke2015A-cocategorical}.

Our solution to these problems, described in the present paper, is based upon the emergence, in the early 2010's, of the new concept of a \emph{skew monoidal category}.  Discovered independently by researchers in quantum groups \cite{Szlachanyi}  and theoretical computer science  \cite{Altenkirch}, they relax the classical notion of a monoidal category by allowing non-invertible coherence constraints
$$\alpha\colon (a \otimes b) \otimes c \to a \otimes (b \otimes c),\, l\colon i \otimes a \to a,\, r\colon a \to a \otimes i$$
and have found diverse applications in the intervening years (for instance, see \cite{Bourke2016Skew, Lack2018Operadic, Shoikhet2018, Garner2021Cartesian}).  One of these, described in our earlier paper \cite{BourkeLobbia}, was the construction of several skew monoidal structures on the category $\GCat$ of Gray-categories.  Here we will focus on two of these, which we call the \emph{left-sharp} and \emph{flexible} skew monoidal structures.  Both are closed, overcoming the earlier obstruction, meaning that $\GCat$ is self-enriched --- in particular, it forms a particular kind of \emph{semi-strict} 4-category $\GCat^4$: namely, a left-sharp $4$-category as introduced here.  In fact, one of our main contributions will be to use enrichment over these bases to introduce several interconnected flavours of semi-strict $4$-category, depicted in the diagram below

\[\begin{tikzcd}[ampersand replacement=\&]
	\&\&\& {\textnormal{Left-sharp}} \&\&\& \\
	{\textnormal{4-teisi}} \&\& {\textnormal{Sharp}} \&\& {\textnormal{Moderate}} \&\& {\textnormal{Loose}} \\
	\&\&\& {\textnormal{Right-sharp}}
	\arrow[hook, from=1-4, to=2-5]
	\arrow[hook, from=2-1, to=2-3]
	\arrow[hook, from=2-3, to=1-4]
	\arrow[hook, from=2-3, to=3-4]
	\arrow[hook, from=2-5, to=2-7]
	\arrow[hook, from=3-4, to=2-5]
\end{tikzcd}\]

and to use them to draw together the existing literature on the topic.  We believe that these include the main notions of semi-strict $4$-category that merit further investigation.  Let us mention a few features and highlights.
\begin{itemize}
\item The most general kind of $4$-category we introduce are the loose $4$-categories.  These are defined as categories enriched in the multicategory of loose multimorphisms underlying the flexible skew monoidal structure.  Since this multicategory is symmetric biclosed, loose $4$-categories admit lots of convenient constructions --- opposites, enriched representable functors and so on. 
\end{itemize}
The other kinds of semi-strict $4$-category we introduce are then obtained by imposing natural further axioms on the loose $4$-categories.
\begin{itemize}
\item Moderate $4$-categories are loose $4$-categories in which whiskering by $1$-cells is strictly functorial.  We give a complete list of their axioms in elementary terms in the appendix.
\item Braided monoidal $2$-categories, in the sense of Crans, turn out to be doubly-degenerate moderate $4$-categories.  Taking this as an inspiration, we also introduce monoidal Gray-categories as one-object moderate $4$-categories.
\item Adding further (dual) axioms leads to the notions of left-sharp and right-sharp $4$-categories, whose intersection we call the sharp $4$-categories.  From these, we obtain Crans' $4$-teisi by imposing a single further invertibility constraint.
\item A particular focus will be on the left-sharp $4$-categories.  These are key since our leading example of a semi-strict $4$-category, namely $\GCat^4$, is itself left-sharp.  They are also well-behaved since they precisely categories enriched in the left-sharp skew monoidal structure whilst right-sharp ones are categories enriched in its reverse, the right-sharp skew monoidal structure.  
\end{itemize}

After introducing these various flavours of $4$-category, we turn to homotopical aspects.  A key further feature of the Gray-tensor product is that it equips the category $\twocat$ of $2$-categories with the structure of a monoidal model structure, and so induces a symmetric monoidal closed structure on the homotopy category $\Ho(\twocat)$ of $2$-categories.  In \Cref{thm:smchomotopy} we show that the left-sharp skew monoidal structure induces a genuine symmetric monoidal closed structure on the homotopy category $\Ho(\GCat)$ of Gray-categories.

Let us now give a more detailed overview of the paper.  In \Cref{prelims1}, we describe background material on $2$-categories and Gray-categories.  \Cref{sect:base} describes the necessary background on multicategories, skew monoidal categories and the flexible and left-sharp skew monoidal structures from our earlier paper \cite{BourkeLobbia}.  Enriching in these, in \Cref{sect:flavours}, we introduce the different kinds of semi-strict $4$-category, develop their basic theory and describe a wide range of examples. 

In \Cref{hoprelim} we turn to homotopy theory, beginning with the necessary background concerning Lack's model structure on Gray-categories \cite{Lack2011A-quillen}.  \Cref{sect:replacing} can be viewed as a first step in the direction of coherence theorems for $4$-categories.  It is devoted to understanding the extent to which pseudomaps between Gray-categories can be replaced by strict ones, with the core result being \Cref{thm:replacement}.  In \Cref{thm:char-equiv-in-GCat_s} we use this to characterise the equivalence $n$-cells in the $4$-categories $\GCat^4$.  In \Cref{axiommc} we use the above to prove that the left-sharp skew monoidal structure interacts well with the model structure for Gray-categories \cite{Lack2011A-quillen} and in \Cref{thm:smchomotopy} prove that it induces a symmetric monoidal closed structure on the homotopy category of Gray-categories).   

Finally, in \Cref{app:axioms-4-cats}, we describe the axioms for a moderate $4$-category in elementary terms.

\black
\subsection{Related work}
Shortly after the publication of \cite{BourkeLobbia}, Miranda \cite{MirandaThesis} in his thesis independently described several enriching structures for Gray-categories and, furthermore, investigated enriched categories over these as semistrict 4-categories.  Whilst none of these concepts match those introduced here, his detailed comparison with the literature was immensely valuable to us, since slight refinements of two of his results enabled us to draw precise comparisons between moderate $4$-categories and braided monoidal $2$-categories and between sharp $4$-categories and $4$-teisi.  See Section \ref{sect:flavours} for more on these results and \Cref{mirandacomparison} for a more detailed comparison with Miranda's work.

Whilst our final manuscript was in preparation, Fukuda and Shu \cite{Fukuda} defined the concept of a Gray-type 4-category, which are a slightly restricted form of the moderate $4$-categories introduced here.  For a further comparison, see \Cref{fukudacomparison}.\black

%
%
%
%
%

\subsection*{Acknowledgements}

The authors thank Nils Carqueville, Nick Gurski and Lorenz Weber for useful feedback on a preliminary draft.

\section{Preliminaries on 2-categories and Gray-categories}\label{prelims1}

In the present section, we recall a few basic concepts that we will use about $2$-categories and Gray-categories, and take the opportunity to fix some notation.

\subsection{2-categories}\label{sect:prelim2}
Let $\Cat$ denote the cartesian closed category of small categories.  A $2$-category, by definition, is a $\Cat$-enriched category.  

As standard concepts of enriched category theory, we have the notions of enriched functor and enriched natural transformation between $2$-categories, which we call $2$-functors and $2$-natural transformations.  These form the $1$-cells and $2$-cells of the $2$-category $\twocat$ of small $2$-categories.  Equivalences and adjoint equivalences in this $2$-category are called $2$-equivalences and adjoint $2$-equivalence, and are the standard enriched notions.  There is also the more general notion of biequivalence, which we recall in Section \ref{sect:2model}.  Beyond the above usage, we will treat $\twocat$ as a mere category and henceforth use the notation $\twocat$ for the category of small $2$-categories.

Whilst $\twocat$ is itself cartesian closed, it also admits a second symmetric monoidal closed structure, called the \emph{Gray tensor product}.  We write $\Gray := (\twocat,\otimes_p,1)$ for this monoidal structure where $\otimes_p$ denotes the Gray-tensor product, as described in Section 3.2 of \cite{Gurskibook} or Section 4 of \cite{Gordon1995Coherence}.  An important point is that the internal hom for this monoidal structure is the $2$-category $\psd^2(\ca,\cb)$ of 2-functors, \emph{pseudonatural transformations} and modifications.

\subsection{Equivalences in $2$-categories}

In the various flavours of higher-category that we will consider, we speak of equivalence $n$-cells for $n \geq 1$.  In a $2$-category $\ca$, an equivalence $2$-cell is, by definition, an invertible $2$-cell. An equivalence $1$-cell is the standard notion of equivalence mentioned above: that is, a $1$-cell $f \colon A \to B$ for which there exists a $1$-cell $g \colon B \to A$ and equivalence $2$-cells $gf \simeq 1_A$ and $fg \simeq 1_B$.

We will also speak of adjoint equivalence $n$-cells.  For $n=2$,  these are $2$-cells equipped with an inverse.  Adjoint equivalence 1-cells are the usual adjoint equivalences in a 2-category, which we sometimes just call adjoint equivalences.
These will play a particularly important role and appear in the various flavours of semi-strict $4$-category that we introduce in \Cref{sect:flavours}.

\black
\begin{notation*}
We now set some notation for adjoint equivalences in a $2$-category $\ca$.
\begin{itemize}
\item We use the notation $\aexi:=(\epsilon,f\dashv g,\eta) \colon x \to y$ to denote an adjoint equivalence in $\ca$ pointing in the direction of its right adjoint $g \colon x \to y$.
\item  Given adjoint equivalences $\aexi=(\epsilon,f\dashv g,\eta) \colon x \to y$ and $\aexi'=(\epsilon',f'\dashv g',\eta') \colon y \to z$ the composite adjoint equivalence is given by $\aexi' \cdot \aexi := (\epsilon\cdot f\epsilon' g,\,ff'\dashv g'g,\,g'\eta f'\cdot\eta')\colon x\to z$.  
\item For $x \in \ck$, the identity adjoint equivalence is $1^\aeq_x := 1_x \dashv 1_x\colon x \to x$.  
\item With this structure, we obtain the category $\adje(\ca)$ of adjoint equivalences in $\ca$.
\item Given $\aexi=(\epsilon,f\dashv g,\eta)\colon X\to Y$ we can form the \emph{reverse} adjoint equivalence $\aexi^\rev := (\eta^{-1},g\dashv f,\epsilon^{-1})\colon Y\to X$.  This describes a contravariant involution $(-)^{\rev} \colon \adje(\ca)^{op} \to \adje(\ca)$.
\item A 2-functor $F\colon \ca \to \cb$ induces a functor $\adje(F) \colon \adje(\ca) \to \adje(\cb)$ sending $\aexi=(\epsilon,f\dashv g,\eta)$ to $F\aexi:= (F\epsilon,Ff\dashv Fg,F\eta)$.  This satisfies $(F\aexi)^{\rev} = F(\aexi^{\rev})$.
\end{itemize} 
\end{notation*}

An important fact is that each equivalence $n$-cell in a $2$-category underlies an adjoint equivalence $n$-cell.  For the case $n=2$, this is trivial and for the case $n=1$, it is a basic $2$-categorical fact.
\black

%
%

\subsection{Gray-categories}\label{sect:graycats}

A Gray-category $\ca$ is a category enriched in $\Gray = (\twocat,\otimes_p,1)$.  We presume some familiarity with Gray-categories, but take the opportunity to fix our notation.  In particular, a Gray-category $\ca$ amounts to:
\begin{enumerate}[label=\roman*]
\item A category $\ca_1$ of $0$ and $1$-cells with composition of $f\colon x \to y$ and $g\colon y \to z$ denoted by $gf\colon x \to z$, and with identities denoted by $1_x\colon x \to x$.
\item\label{ax:strict-hom} For $0$-cells $x,y$, a $2$-category $\ca(x,y)$ whose objects are the $1$-cells of $\ca_1$, and whose $n$-cells are referred to as the $(n+1)$-cells of $\ca$.
\item Given $1$-cells $f\colon w \to x$ and $g\colon y \to z$, there are \emph{whiskering} $2$-functors $- \circ f\colon \ca(x,y) \to \ca(w,y)$ and $g \circ -\colon \ca(x,y) \to \ca(x,z)$ whose action on objects is composition in $\ca_1$.  These whiskering operations are functorial in $\ca_1$ and compatible, in the sense that $(g \circ -) \circ f = g \circ (- \circ f)$.
\item Given a pair of $2$-cells 
\[\begin{tikzcd}
	x & y & z
	\arrow[""{name=0, anchor=center, inner sep=0}, "f", curve={height=-12pt}, from=1-1, to=1-2]
	\arrow[""{name=1, anchor=center, inner sep=0}, "{f'}"', curve={height=12pt}, from=1-1, to=1-2]
	\arrow[""{name=2, anchor=center, inner sep=0}, "g", curve={height=-12pt}, from=1-2, to=1-3]
	\arrow[""{name=3, anchor=center, inner sep=0}, "{g'}"', curve={height=12pt}, from=1-2, to=1-3]
	\arrow["\phi"{xshift=0.1cm}, between={0.2}{0.8}, Rightarrow, nfold, from=0, to=1]
	\arrow["\psi"{xshift=0.1cm}, between={0.2}{0.8}, Rightarrow, nfold, from=2, to=3]
\end{tikzcd}\]\black
	there is an invertible $3$-cell
\[\begin{tikzcd}[ampersand replacement=\&]
	gf \& {g'f} \\
	{gf'} \& {g'f'}
	\arrow[""{name=0, anchor=center, inner sep=0}, "{\psi\circ f}", from=1-1, to=1-2]
	\arrow["{g\circ\phi}"', from=1-1, to=2-1]
	\arrow["{g'\circ\phi}", from=1-2, to=2-2]
	\arrow[""{name=1, anchor=center, inner sep=0}, "{\psi\circ f'}"', from=2-1, to=2-2]
	\arrow["{\psi:\phi}", shift right=2, between={0.4}{0.6}, Rightarrow, nfold, from=0, to=1]
\end{tikzcd}\]
\black
	called the interchange map.
	The interchange maps satisfy several equations, for which we refer the reader to \cite[Definition~1.1]{GambLob:psm}.\footnote{\black Notice that the direction of the interchange isomorphism here is the opposite with respect to both \cite{BourkeLobbia} and \cite[Definition~1.1]{GambLob:psm}.  This choice of convention matches the viewpoint of pseudo-natural transformations as lax natural transformations (rather than oplax ones) and fits better with our treatment of pseudo-transformations of Gray-categories as certain lax transformations.  Since in \cite{BourkeLobbia} we used the dashed arrows notation throughout for interchange cells, the two papers can be read consistently --- the only exception to this is in a single axiom for a pseudomap in Appendix A.1.(vii), whose interchange map (according to the present convention) should be replaced by its inverse.}
\end{enumerate}



\begin{notation*}
Our diagram style follows that of Gurski's book \cite{Gurskibook}.  In diagrams, we will abuse notation for whiskering $2/3$-cells by a $1$-cell and only write $1\circ-$ and $-\circ1$, as in the diagram below. 
\begin{center}
\begin{tikzcd}[ampersand replacement=\&]
	gf \& {g'f} \\
	{gf'} \& {g'f'}
	\arrow[""{name=0, anchor=center, inner sep=0}, "{\psi\circ1}", from=1-1, to=1-2]
	\arrow["{1\circ\phi}"', from=1-1, to=2-1]
	\arrow["{1\circ\phi}", from=1-2, to=2-2]
	\arrow[""{name=1, anchor=center, inner sep=0}, "{\psi\circ1}"', from=2-1, to=2-2]
	\arrow[between={0.3}{0.7}, Rightarrow, nfold, dashed, from=0, to=1]
\end{tikzcd}
\hspace{0.5cm}
\begin{tikzcd}[ampersand replacement=\&]
	gf \& {gf'} \\
	{g'f} \& {g'f'}
	\arrow[""{name=0, anchor=center, inner sep=0}, "{1\circ\phi}", from=1-1, to=1-2]
	\arrow["{\psi\circ1}"', from=1-1, to=2-1]
	\arrow["{\psi\circ1}", from=1-2, to=2-2]
	\arrow[""{name=1, anchor=center, inner sep=0}, "{1\circ\phi}"', from=2-1, to=2-2]
	\arrow[between={0.3}{0.7}, Rightarrow, nfold, dashed, from=0, to=1]
\end{tikzcd}
\end{center}
Finally, in some larger diagrams, where the context will make it clear, we will avoid labelling the interchange and its inverse and instead indicate it with a dashed double arrow again as in the above two diagrams. 
\end{notation*}

By standard enriched category theory, we have the $2$-category of Gray-categories,  Gray-functors and Gray-enriched-natural transformation.  Equivalences and adjoint equivalences in this $2$-category are called Gray-enriched equivalences and Gray-enriched adjoint equivalences, respectively.  There is also the more general notion of triequivalence, which we recall in Section \ref{sect:Graymodel}.

We write $\GCat$ for the category of small Gray-categories and Gray-functors.  
An important role will also be played by the more general pseudomaps of Gray-categories, which we recall in \Cref{sect:pseudomap}.

\subsection{Equivalences in Gray-categories}\label{sect:equiv-gray}

For $n \in \{2,3\}$, an equivalence $n$-cell in a Gray-category $\ca$ is, by definition, an equivalence $(n-1)$-cell in the appropriate hom $2$-category $\ca(x,y)$.  An equivalence $1$-cell $f \colon X \to Y$ is a $1$-cell for which there exists $g \colon Y \to X$ and equivalence $2$-cells $gf \simeq 1_X$ and $gf \simeq 1_Y$.  

Similarly, for $n \in \{2,3\}$, an adjoint equivalence $n$-cell in $\ca$ is an adjoint equivalence $(n-1)$-cell in the appropriate hom $2$-category $\ca(x,y)$.  An adjoint equivalence $1$-cell is an adjoint biequivalence in the sense of Section 3 of \cite{Lack2011A-quillen}.

\black As in the $2$-categorical setting, each equivalence $n$-cell in a Gray-category underlies an adjoint equivalence $n$-cell.  For the cases $n=2,3$, this follows from the corresponding fact for $2$-categories.  That each equivalence $1$-cell can be made into an adjoint equivalence $1$-cell is proven in \cite[Proposition 3.1]{Lack2011A-quillen}.\black

\subsection{Pseudomaps of Gray-categories}\label{sect:pseudomap}
%
%

In addition to Gray-functors, we will consider the more general notion of pseudo-Gray-functor \cite{GohlaB:mapp} which we will call a \emph{pseudomap}.  We denote a pseudomap with a wobbly arrow $\ca \rightsquigarrow \cb$ to distinguish it from the strict maps $\ca \to \cb$, i.e. the Gray-functors.

Like a Gray-functor, a pseudomap $F\colon \ca \rightsquigarrow \cb$ involves an object assignment $x \mapsto Fx$ together with $2$-functors $F_{x,y} \colon \ca(x,y) \to \cb(Fx,Fy)$ for all $x,y \in \ca$.  In addition, it comes equipped with a cocycle $F^2$, that is,
	\begin{itemize}
	\item  for any pair of 1-cells $f_1\colon x\to y$ and $f_2\colon y\to z$ in $\ca$, an invertible 2-cell 
	$$F^2_{f_2,f_1}\colon Ff_2\circ Ff_1\to F(f_2\circ f_1)$$
	 in $\cb$ satisfying the axioms listed in Appendix A.1 of \cite{BourkeLobbia}. 
	\end{itemize}  

Pseudomaps admit a natural composition and form the morphisms of the category $\GCat_p$ (see \cite[Section 3.2]{BourkeLobbia} or \cite{GohlaB:mapp} for further details).

\begin{rmk}\label{normalpseudo}
We point out that pseudomaps are not very weak at all: they preserve identity $1$-cells strictly and their cocycles are strictly associative.  Indeed, they generalise \emph{normal pseudofunctors} in $2$-category theory in the sense that a pseudomap between Gray-categories in which all $3$-cells are identities is precisely a normal pseudofunctor between the underlying $2$-categories.
\end{rmk}

\begin{lemma}\label{pseudomapEquivs}
A pseudomap $F \colon \ca \rightsquigarrow \cb$ preserves equivalence $n$-cells for each $n$.
\end{lemma}
\begin{proof}
Pseudomaps are locally $2$-functors.  As such they preserve equivalence $n$-cells for $n=2,3$ since $2$-functors preserve equivalence $n$-cells for $n=1,2$.  
Consider an equivalence $1$-cell $f \colon x \to y \in \ca$.  Then there exists $g \colon y \to x$ and equivalence 2-cells $g \circ f \simeq 1_x$ and $f \circ g \simeq 1_y$.  By the above, these are sent to equivalence $2$ -cells $F(g \circ f) \simeq F(1_x)$ and $F(f \circ g) \simeq F(1_y)$.  Now pseudomaps of Gray-categories preserve $1$-cell composition up to an invertible $2$-cell and strictly preserve identity $1$-cells; thus we have a composite equivalence $2$-cell $Fg \circ Ff \cong F(g \circ f) \simeq F(1_x) = 1_{Fx}$ and likewise a composite equivalence $2$-cell $Ff \circ Fg \simeq 1_{Fy}$, proving that $Ff$ is an equivalence $1$-cell.
\end{proof}

\subsection{Lax transformations, pseudo-transformations and higher cells}

Between pseudomaps $F, G \in \GCat_p(\ca,\cb)$, there are lax and pseudo-transformations and further higher cells.  We recall the data involved in these concepts now, and refer the reader to \cite{BourkeLobbia} Appendix A for the axioms they must satisfy.

\begin{itemize}
\item A \emph{lax transformation} $\alpha\colon F\to G$ consists of:
	\begin{itemize}
	\item For any object $x\in\ca$, a 1-cell $\alpha_x\colon Fx\to Gx$ in $\cb$.
	\item For any 1-cell $f\colon x\to y$ in $\ca$, a 2-cell  in $\cb$ of the form
\[\begin{tikzcd}
	{Fx} & {Gx} \\
	{Fy} & {Gy.}
	\arrow[""{name=0, anchor=center, inner sep=0}, "{\alpha_x}", from=1-1, to=1-2]
	\arrow[""{name=1, anchor=center, inner sep=0}, "{\alpha_Y}"', from=2-1, to=2-2]
	\arrow["Ff"', from=1-1, to=2-1]
	\arrow["{Gf}", from=1-2, to=2-2]
	\arrow["{\alpha_f}"{xshift=0.1cm}, shorten <=10pt, shorten >=10pt, Rightarrow, from=0, to=1]
\end{tikzcd}\]

	\item For any 2-cell $\phi\colon f\to f'$ in $\ca$, a 3-cell $\alpha_\phi$ in $\cb$ of the form
\[\begin{tikzcd}
	{Gf\alpha_x} & {\alpha_y Ff} \\
	{Gf\alpha_x} & {\alpha_y Ff.}
	\arrow[""{name=0, anchor=center, inner sep=0}, "{\alpha_f}", from=1-1, to=1-2]
	\arrow[""{name=1, anchor=center, inner sep=0}, "{\alpha_{f'}}"', from=2-1, to=2-2]
	\arrow["G\phi\circ1"', from=1-1, to=2-1]
	\arrow["{1\circ F\phi}", from=1-2, to=2-2]
	\arrow["{\alpha_\phi}"{xshift=0.1cm}, shorten <=10pt, shorten >=10pt, Rightarrow, from=0, to=1]
\end{tikzcd}\]
	\item For any pair of 1-cells $f_1\colon x\to y$ and $f_2\colon y\to z$ in $\ca$, an invertible 3-cell $\alpha^2_{f_2,f_1}$ in $\cb$ of the form
\[\begin{tikzcd}
	{Gf_2Gf_1\alpha_x} & {Gf_2\alpha_y Ff_1} & {\alpha_z Ff_2Ff_1} \\
	{G(f_2f_1) \alpha_x} && {\alpha_zF(f_2f_1).}
	\arrow["{1\circ\alpha_{f_1}}", from=1-1, to=1-2]
	\arrow["{G^2_{f_2,f_1}\circ1}"', from=1-1, to=2-1]
	\arrow["{\alpha_{f_2}\circ1}", from=1-2, to=1-3]
	\arrow["{1\circ F^2_{f_2,f_1}}", from=1-3, to=2-3]
	\arrow[""{name=0, anchor=center, inner sep=0}, "{\alpha_{f_2f_1}}"', from=2-1, to=2-3]
	\arrow["{\alpha^2_{f_2,f_1}}"{xshift=0.1cm}, shorten <=7pt, shorten >=7pt, Rightarrow, from=1-2, to=0]
\end{tikzcd}\]
	\end{itemize}
This data has to satisfy the seven axioms listed in \cite[Appendix A.2]{BourkeLobbia}.   

A \emph{pseudo-transformation} is a lax transformation such that each $2$-cell $\alpha_f$ is equipped with the structure of an adjoint equivalence $\aexi^\alpha_f := \alpha^o_f \dashv \alpha_f\colon Gf\alpha_x \to \alpha_yFf$\black.  Furthermore the adjoint equivalence is required to be the identity if $f$ is an identity, and the $3$-cell $\alpha_\phi$ is required to be invertible.  
%
\item A \emph{modification} $\Lambda\colon\alpha\to\beta$ consists of:
	\begin{itemize}
	\item For any object $x\in\ca$, a 2-cell $\Lambda_x\colon \alpha_x\to \beta_x$ in $\cb$.
	\item For any 1-cell $f\colon x\to y$ in $\ca$, a 3-cell $\Lambda_f$ in $\cb$ of the form
\[\begin{tikzcd}
	{Gf\alpha_x} & {Gf\beta_x} \\
	{\alpha_y Ff} & {\beta_y Ff.}
	\arrow["{\alpha_f}"', from=1-1, to=2-1]
	\arrow["{\beta_{f}}", from=1-2, to=2-2]
	\arrow[""{name=0, anchor=center, inner sep=0}, "{1\circ\Lambda_x}", from=1-1, to=1-2]
	\arrow[""{name=1, anchor=center, inner sep=0}, "{\Lambda_y\circ1}"', from=2-1, to=2-2]
	\arrow["{\Lambda_f}"{xshift=0.1cm}, shorten <=8pt, shorten >=8pt, Rightarrow, from=0, to=1]
\end{tikzcd}\]
	\end{itemize}
This data has to satisfy the three axioms listed in  Appendix A.3 of \cite{BourkeLobbia}.  It is a \emph{pseudo-modification} if the $3$-cell $\Lambda_f$ is invertible.

\item A \emph{perturbation} $\Theta\colon\Lambda\to\Gamma$ consists of, for any object $x\in\ca$, a 3-cell $\Theta_x\colon \Lambda_x\to\Gamma_x$ in $\cb$ subject to the single axiom in Axiom A.4 of \cite{BourkeLobbia}.
\end{itemize}

In the present paper we are concerned only with the pseudo-variations of the above.  We let  $\psd_p(\ca,\cb)$ denote the Gray-category of pseudomaps, pseudo-transformations, pseudo-modifications and perturbations.  Its full sub-Gray-category of strict maps is denoted by $\psd_s(\ca,\cb)$.

\black
\subsection{More on pseudo-transformations: oplax and semi-strict aspects}\label{sect:semistrict}
We now look at a few finer details about pseudo-transformations.  These will not play any role in defining semi-strict $4$-categories in \Cref{sect:flavours} but only in the later homotopical parts of the paper, and the reader is recommended to skip them at a first reading.

Firstly, we remark that the presentation of pseudo-transformations as particular lax transformation is biased.   Indeed a pseudo-transformation $\alpha \colon F \to G$ also gives rise to an \emph{oplax transformation} $\alpha^o \colon F \to G$, whose components at $0$-cells $\alpha_x$ are as before, but whose components at $1$-cells are now the left adjoints $\alpha^o_f$ of the adjoint equivalences $\alpha^o_f \dashv \alpha_f$.

The component of $\alpha^o$ at composable $1$-cells is the mate
\begin{equation}\label{eq:mate}
\begin{tikzcd}[ampersand replacement=\&]
	\& {\alpha_Z Ff_2Ff_1} \&\& {} \&\& \\
	\& {Gf_2\alpha_Y Ff_1} \& {} \\
	{(\alpha^{o})^2_{f_2,f_1}:=} \& {G{f_2}G{f_1}\alpha_X} \&\& {Gf_2\alpha_Y Ff_1} \&\& {\alpha_Z Ff_2Ff_1} \\
	\& {G({f_2}{f_1}) \alpha_X} \&\& {} \& {} \& {\alpha_ZF(f_2f_1)} \\
	\&\&\&\& {} \& {G({f_2}{f_1}) \alpha_X}
	\arrow["{{\alpha^o_{f_2}}\circ1}"', from=1-2, to=2-2]
	\arrow[curve={height=-12pt}, equals, nfold, from=1-2, to=3-6]
	\arrow["{\eta_{f_2}\circ1}"'{pos=0.6}, between={0.5}{0.8}, Rightarrow, nfold, from=1-4, to=3-4]
	\arrow["{1\circ\alpha^o_{f_1}}"', from=2-2, to=3-2]
	\arrow[curve={height=-12pt}, equals, nfold, from=2-2, to=3-4]
	\arrow[""{name=0, anchor=center, inner sep=0}, "{1\circ\alpha_{f_1}}"{description}, from=3-2, to=3-4]
	\arrow["{G^2_{f_2,f_1}\circ1}"', from=3-2, to=4-2]
	\arrow["{\alpha_{f_2}\circ1}"{description}, from=3-4, to=3-6]
	\arrow["{{\alpha}^2_{f_2,f_1}}", between={0.3}{0.7}, Rightarrow, nfold, from=3-4, to=4-4]
	\arrow["{1\circ F^2_{f_2,f_1}}", from=3-6, to=4-6]
	\arrow["{\alpha_{f_2f_1}}"{description}, from=4-2, to=4-6]
	\arrow[curve={height=12pt}, equals, nfold, from=4-2, to=5-6]
	\arrow["{\epsilon_{f_2f_1}}"', between={0.1}{0.7}, Rightarrow, nfold, from=4-5, to=5-5]
	\arrow["{{\alpha^o_{f_2f_1}}}", from=4-6, to=5-6]
	\arrow["{1\circ\eta_{f_1}}"'{pos=0.4}, between={0.3}{0.7}, Rightarrow, nfold, from=2-3, to=0]
\end{tikzcd}
\end{equation}
through the adjoint equivalences and $\alpha^o_{\phi}$ is obtained in a similar fashion --- see Proposition 3.1 of \cite{BourkeLobbia2}.  The same result shows that pseudo-transformations can be seen as interacting pairs of lax and oplax transformations, with similar results for higher transfors.  In particular, this symmetric, or unbiased, description of the $n$-cells of $\psd_p(\ca,\cb)$ has the consequence that $\psd_p(\ca,\cb)$ is the internal hom of a \emph{symmetric} biclosed multicategory (see Corollary 6.13 of \cite{BourkeLobbia2}.)  Biclosedness will play an important role in the later homotopical parts of the paper, beginning in \Cref{hoprelim}. 

An important but delicate point concerns the concept of a \emph{semi-strict pseudo-transformation}, which we now turn to addressing.

\begin{prop}\label{semistrictequivalence}
Let $F,G \colon \ca \rightrightarrows \cb$ be strict maps of Gray-categories and consider a pseudo-transformation $\alpha \colon F \to G$.  The following are equivalent.
\begin{enumerate}
\item The $3$-cells $\alpha^2_{f_2,f_1}$ are identities and we have the equality of adjoint equivalences $\aexi^\alpha_{f_2f_1} = (\aexi^{\alpha}_{f_2} \circ 1)\cdot (1 \circ \aexi^{\alpha}_{f_1})$.
\item Both the $3$-cells $\alpha^2_{f_2,f_1}$ and their mates $(\alpha^{o})^2_{f_2,f_1}$ in \eqref{eq:mate} are identities.
\item The $3$-cells $(\alpha^{o})^2_{f_2,f_1}$ are identities and we have the equality of adjoint equivalences $\aexi^\alpha_{f_2f_1} = (\aexi^{\alpha}_{f_2} \circ 1)\cdot (1 \circ \aexi^{\alpha}_{f_1})$.
\end{enumerate}
\end{prop}
\begin{proof}
Suppose (1).  Since $F$ and $G$ are strict, the two central vertical $1$-cells in \eqref{eq:mate} are identities and so, by assumption, is the $3$-cell filling the central square.  By the equality of adjoint equivalences, the mate \eqref{eq:mate} is an identity by the triangle equations for an adjunction.  Thus (2) holds.  Conversely supposing (2), we must show that the adjoint equivalences in (1) are equal.  Since  $\alpha^2_{f_2,f_1}$ is an identity their right adjoints coincide, whilst since $(\alpha^{o})^2_{f_2,f_1}$ is an identity, their left adjoints coincide (again noting that the two central vertical $1$-cells are identities).  It remains to show that the units and counits coincide.  Since \eqref{eq:mate} is the identity, $(1 \circ \eta_{f_1}). (\eta_{f_2}\circ 1)$ satisfies one of the triangle equation for the adjunction $\aexi^\alpha_{f_2f_1} = (\epsilon_{f_2f_1}, \alpha^o_{f_2f_1} \dashv   \alpha_{f_2f_1}, \eta_{f_2f_1})$.  By uniqueness of $\eta_{f_2f_1}$ with this property, it follows that $\eta_{f_2f_1} = (1 \circ \eta_{f_1}). (\eta_{f_2}\circ 1)$.  Given the equality of units, the dual argument shows that the counits are equal, proving (1).  The equivalence of (2) and (3) is dual.
\end{proof}

\begin{defn}\label{semistrictpseudotransformation}
Let $F,G \colon \ca \rightrightarrows \cb$ be strict maps of Gray-categories.  A pseudo-transformation $\alpha \colon F \to G$ is \emph{semi-strict} if it satisfies the equivalent conditions of \Cref{semistrictequivalence}.
\end{defn}

\begin{rmk}
Our definition above strengthens \cite[Definition 8.1.3]{MirandaSemistrict} which defines a pseudo-transformation to be semi-strict merely when the $3$-cells $\alpha^2_{f_2,f_1}$ are identities.  The extra condition on the equality of adjoint equivalences will be needed in \Cref{semistrictasstrict} where we show that semi-strict transformations correspond to certain Gray-functors.  To see that it is not redundant, let $\ca = \{0 \to 1 \to 2\}$ and $\cb = \twocat$, viewed as a Gray-category.  Consider an equivalence $u: \cc \to \cd$ of ordinary categories with distinct adjoint equivalence inverses $f \dashv u$ and $f' \dashv u$.  Now define $F \colon \ca \to \twocat$  by $F0=F1=\mathbf 1$, the terminal category, and $F2 = \Cat$, with $F(0 \to 2) = F(1 \to 2) = \cc \colon \mathbf 1 \to \Cat$ the 2-functor selecting $\cc$ and define $G \colon \ca \to \twocat$ identically, but substituting $\cc$ for $\cd$.   We then define a pseudo-transformation $\alpha \colon F \to G$ with each component $\alpha_x \colon Fx \to Gx$ the identity for $x = 0,1,2$.  Then $\aexi^\alpha_{0 \to 1}$ must be the identity adjoint equivalence, whilst we set $\aexi^\alpha_{1 \to 2} = f \dashv u$ and $\aexi^\alpha_{0 \to 2} = f' \dashv u$.  Finally, setting $\aexi^\alpha_{0 \to 1,1 \to 2} = 1_{u}$ produces a pseudo-transformation which is semi-strict in the sense of \cite{MirandaSemistrict} but not in the stronger sense defined above since $f' \neq f$.
\end{rmk}
\black

\section{Skew monoidal categories and multicategories of Gray-categories}\label{sect:base}

In this section, we recall several skew monoidal categories and multicategories of Gray-categories from \cite{BourkeLobbia}.  In Section \ref{sect:flavours}, we will enrich in these structures to define our notions of semi-strict $4$-category.  Since we will primarily take a multicategorical approach, we start by briefly recalling a few basics about multicategories.

\subsection{Background on multicategories}
Multicategories were introduced by Lambek in \cite{Lambek} and generalise categories by allowing morphisms with many inputs.  In particular, a multicategory $\mathbb V$ has objects and sets $\mathbb V_n(\overline{a};b)$ of $n$-ary multimaps from n-tuples of objects $\overline{a} = (a_1,\ldots,a_n)$ to objects $b$.  In addition, it has unary identities and substitution operations that allows the composition of multimaps.  A further good introductory text is \cite{LeinsterBook}.

\subsubsection{Closed multicategories}
Given a multicategory $\mathbb V$, it is said to be left closed (often just called closed) if given $b,c \in \mathbb V$, there exists an internal hom object $[b,c]$ and binary map $e_{b,c} \colon [b,c],b \to c$ such that the induced function $$e_{b,c} \circ_1 - \colon \mathbb V_n(\overline{a};[b,c]) \to \mathbb V_{n+1}(\overline{a},b;c)$$ obtained by substituting in the first variable is a bijection.  Dually, there are right closed multicategories: these are defined by objects $[b,c]$ and binary maps $e_{b,c} \colon b,[b,c] \to c$ such that the induced function $e_{b,c} \circ_2 - \colon \mathbb V_n(\overline{a};[b,c]) \to \mathbb V_{n+1}(b,\overline{a};c)$ obtained by substituting in the second position is a bijection.  The multicategory $\mathbb V$ is said to be biclosed if it is left and right closed.

\subsubsection{Symmetric multicategories} The multicategory $\mathbb V$ is said to be symmetric if it equipped with actions $\sigma.- \colon \mathbb V_n(\overline{a};b) \to \mathbb V_n(a_{\sigma 1},\ldots,a_{\sigma n};b)$ of elements $\sigma$ of the symmetric group $S_n$ compatible with the substitution operations.  

A useful fact is that a symmetric multicategory is left closed if and only if it is right closed: for instance, if the binary maps $e_{b,c} \colon [b,c],b \to c$ equip it with left-closed structure, then the binary maps $(1,2).e_{b,c} \colon b,[b,c] \to c$ obtained by applying the transposition $(1,2) \in S_2$ equip it with a right-closed structure.  In this paper, we are interested in two main examples, both of which are left closed and one of which is symmetric and hence biclosed.

\subsubsection{Categories enriched in a multicategory} A category $\ca$ enriched in the multicategory $\mathbb V$ has objects $x,y,z \ldots$, hom-objects $\ca(x,y) \in \mathbb V$ together with composition and identity multimaps $$\circ_{x,y,z} \colon \ca(y,z),\ca(x,y) \to \ca(x,z) \hspace{0.5cm} \textnormal{and} \hspace{0.5cm} 1_x \colon \diamond \to \ca(x,x)$$ satisfying the evident associativity and unit equations.  A $\mathbb V$-enriched functor $F \colon \ca \to \cb$ involves an assignment on objects together with unary morphisms $F_{x,y} \colon \ca(x,y) \to \cb(Fx,Fy)$ in $\mathbb V$ compatible with the composition and identity multimaps.  (See Section 2 of \cite{LeinsterGeneralised} for a general framework for enriched categories including the above.)

The following lemma gathers the few basic facts about multicategory-enriched categories that we will need into one place.

\begin{lemma}\label{basicenrichment}
Let $\mathbb V$ be a left closed multicategory.
\begin{enumerate}
\item Then there is a $\mathbb V$-enriched category, $\mathcal V$, with objects those of $\mathbb V$ and hom-objects $[a,b]$ given by the internal hom in $\mathbb V$.
\item If $\mathcal A$ is a $\mathbb V$-category, and $a \in \mathcal A$, there is a $\mathbb V$-enriched functor $\mathcal A(a,-) \colon \mathcal A \to \mathcal V$.
\item If $\mathbb V$ is symmetric, then each $\mathbb V$-category $\mathcal A$ has an opposite $\mathbb V$-category $\mathcal A^{op}$.  
\item Under the same assumptions, given an object $a \in \mathcal A$, we then also have a $\mathbb V$-functor $\mathcal A(-,a) \colon \mathcal A^{op} \to \mathcal V$.
\end{enumerate}
\end{lemma}
\begin{proof}
The arguments are standard and we only outline the constructions.
\begin{enumerate}
\item The binary map $m_{a,b,c} \colon [b,c],[a,b] \to [a,c]$ for $\mathcal V$ is the adjoint of the composite multimap $e_{b,c} \circ_2 e_{a,b} \colon [b,c],[a,b],a \to c$.
The identity nullary map $1_a \colon \diamond \to [a,a]$ is the adjoint of the unary identity $1_{a} \colon a \to a$ in $\mathbb V$.  
\item The $\mathbb V$-enriched functor $\mathcal A(a,-)$ maps $b$ to  $\mathcal A(a,b)$ and its action $\mathcal A(b,c) \to [\ca(a,b),\ca(a,c)]$ on homs is the adjoint of $m_{a,b,c} \colon \mathcal A(b,c),A(a,b) \to A(a,c)$.
\item $\ca^{op}$ has the same objects as $\ca$ but $\ca^{op}(a.b) := \ca(b,a)$.  Its composition $\ca^{op}(b,c),\ca^{op}(a,b) \to \ca^{op}(a,c)$ is given $(1,2) . \circ_{c,b,a} \colon \ca(c,b),\ca(b,a) \to \ca(c,a)$ whilst its unit map is unchanged.
\item Take $\ca(-,a) = \ca^{op}(a,-)$ using (2).
\end{enumerate}
\end{proof}

\subsubsection{Representable multicategories}\label{representable}
Many fundamental tensor products, such as the tensor product of vector spaces or of abelian groups, are constructed as classifying objects for certain kinds of multilinear maps.  Building on this idea, Hermida \cite[Theorem 9.8]{Hermida} showed that monoidal categories correspond to certain multicategories called \emph{representable multicategories}.  

The notion of representable multicategory depends upon the concept of a universal multimap, which is a multimap $\theta_{\overline{a}} \colon \overline{a} \to m(\overline{a})$ in $\mathbb V$ for which the induced function $\mathbb V_1(m(\overline{a});b) \to \mathbb V_n(\overline{a};b)$ is a bijection for all $b$.  (If $\theta_{\overline{a}} \colon \overline{a} \to m(\overline{a})$ is of arity $n$, we follow \cite{BourkeLack:skew} in calling it an $n$-ary map classifier.)

The multicategory $\mathbb V$ is then said to be \emph{representable} if it admits universal multimaps and they are closed under multicategorical substitution.  Given a representable multicategory $\mathbb V$, the corresponding monoidal category \cite[Definition 9.6]{Hermida} is described as follows.
\begin{itemize}
\item Its underlying category is the category $\mathbb V_1$ of unary morphisms in $\mathbb V$;
\item The tensor product $a \otimes b$ is the codomain of the binary map classifier $\theta_{a,b} \colon (a,b) \to a \otimes b$ whilst the unit $i$ arises from the nullary map classifier $u \colon \diamond \to i$;
\item The associator $\alpha_{a,b,c} \colon (a \otimes b) \otimes c \cong a \otimes (b \otimes c)$ is the unique unary map for which $\alpha_{a,b,c} \circ_1 (\theta_{a \otimes b,c} \circ_1 \theta_{a,b}) = \theta_{a, b \otimes c} \circ_2 \theta_{b,c}$.
\item The unit isomorphisms $i \otimes a \cong a \cong a \otimes i$ are induced by the fact that each of these three objects classifies unary maps with domain $a$.
\end{itemize}
If $\mathbb V$ is furthermore closed/symmetric, the associated monoidal category is closed/symmetric, with the same internal hom $[a,b]$ and with braiding $s_{a,b} \colon a \otimes b \to b \otimes a$ the unique map for which $s_{a,b} \circ \theta_{a,b} = (1,2). \theta_{b,a}$.  We will use this fact in \Cref{thm:smchomotopy} to prove that the homotopy category of Gray-categories is symmetric monoidal closed.


\subsection{Background on skew monoidal categories}
Skew monoidal categories were introduced independently in \cite{Altenkirch} and \cite{Szlachanyi}, with motivations coming respectively from the study of relative monads and bialgebroids over rings.  Here we recall the definitions and associated multicategorical aspects. 
\subsubsection{Skew monoidal categories}

A left skew monoidal category is a category $\cc$ together with a functor $\otimes\colon\cc\times \cc \to \cc$, a unit object $i \in \cc$, and natural families $\alpha \colon(a \otimes b) \otimes c \to a \otimes (b \otimes c)$, $l \colon i \otimes a \to a$ and $r \colon a \to a \otimes i$ satisfying five axioms described in \cite{Szlachanyi}.  There are also right skew monoidal categories in which $\alpha,l$ and $r$ have their orientiations reversed.   In this paper, our primary focus will be on left skew monoidal categories.  Monoidal categories are precisely those skew monoidal categories for which $\alpha$, $l$ and $r$ are invertible.  

The left skew monoidal $(\cc,\otimes,i)$ is said to be:
\begin{itemize}
\item \emph{left normal} if $l$ is invertible;
\item  \emph{right normal} if $r$ is invertible;
\item \emph{normal} if it is both left and right normal;
\item \emph{closed} (or left closed) if each endofunctor $- \otimes a \colon \cc \to \cc$ has a right adjoint;
\item \emph{symmetric} if it is equipped with a natural isomorphism $s_{a,b,c} \colon(a \otimes b) \otimes c \to (a \otimes c) \otimes b$ satisfying the five equations in \cite{BourkeLack:symmetric}.
\end{itemize}

A category $\ca$ enriched over the skew monoidal $\cc$ involves a class of objects $ob\ca$ together with hom-objects $\ca(x,y) \in \cc$ for $x,y \in ob\ca$ plus composition and unit maps $\ca(y,z) \otimes \ca(x,y) \to \ca(x,z)$ and $i \to \ca(x,x)$ in $\cc$ satisfying the standard associativity and unit equations.

\subsubsection{Skew multicategories associated to skew monoidal categories}

Just as monoidal categories correspond to certain multicategories (the representable ones), skew monoidal categories correspond to certain \emph{skew multicategories} (the left representable ones).  

Skew multicategories, introduced in \cite{BourkeLack:skew} differ from classical multicategories in that they have two kinds of multimap (tight and loose).  We will not recall the definition here because they are not really needed in this paper.  Rather, the reader only needs to know that if $\mathbb C$ is a skew multicategory, then it has an underlying ordinary multicategory $\mathbb C^l$ containing its loose multimaps.  In particular, each skew monoidal category $(\cc,\otimes,i)$ has an associated multicategory $\mathbb C^l$.  

Properties of a skew monoidal category correspond to properties of the associated multicategory, with the key ones summarised below.
\begin{itemize}
\item If $(\cc,\otimes,i)$ is closed then the multicategory $\mathbb C^l$ is left closed by \cite[Theorem 6.4 ]{BourkeLack:skew}.
\item If $(\cc,\otimes,i)$ is symmetric, then $\mathbb C^l$ is a symmetric multicategory by \cite[Theorem 5.9 ]{BourkeLack:symmetric}.
\item Combining the above, if $(\cc,\otimes,i)$ is symmetric and closed, then $\mathbb C^l$ is symmetric and left closed, and so bi-closed.
\end{itemize}

There are two skew monoidal structures of interest in the present paper, which we call the flexible skew monoidal structure $(\GCat,\otimes_p,\mathbf{1})$ and the left-sharp skew monoidal structure $(\GCat,\leftsharp,\mathbf{1})$.  Both of these were introduced in \cite{BourkeLobbia}.  In the following sections, we describe the key aspects of these skew monoidal categories and their associated multicategories, in which we will enrich.

\subsection{The flexible skew monoidal structure and its associated multicategory $\GCatmp^l$} In  \cite[Theorem 6.3]{BourkeLobbia}, we constructed a skew monoidal structure $(\GCat,\otimes_p,\mathbf{1})$ that we here call the \emph{flexible skew monoidal structure}.  We recall some of its properties now.
\begin{itemize}
\item The unit is the terminal Gray-category $\mathbf{1}$.
\item It is closed, with internal hom the Gray-category $\psd_p(\ca,\cb)$ of pseudomaps and pseudo-transformations.\begin{footnote}{Note that in \cite{BourkeLobbia} and \cite{BourkeLobbia2} we write $\psd(\ca,\cb)$ for what we here call $\psd_p(\ca,\cb)$.  Here we adopt the notation $\psd_p(\ca,\cb)$ to distinguish it from the sub Gray-category $\psd_s(\ca,\cb)$ of strict maps, which will also be important.}\end{footnote}
\item It is \emph{right normal}.  (This corresponds to the fact that $\psd_p(\mathbf{1},\cb) \cong \cb$.)
\item It is a \emph{symmetric} skew monoidal category, as established in Theorem 6.12 of \cite{BourkeLobbia2}. 
\end{itemize}

The associated multicategory of loose multimaps is called $\GCatmp^l$.  This was described in \cite{BourkeLobbia} up to dimension 4, and we will recall its low-dimensional structure.  
\begin{itemize}
\item The objects of $\GCatmp^l$ are Gray-categories.
\item Nullary maps $x \colon \diamond \to \ca$ are objects $x \in \ca$.
\item Unary maps of $\GCatmp^l$ are pseudomaps $F \colon \ca \rightsquigarrow \cb$.  In particular, the underlying category of unary maps in $\GCatmp^l$ is the category $\GCat_p$ of Gray-categories and pseudomaps.
\item A binary map $F \colon \ca,\cb \rightsquigarrow \cc$ consists of:
\begin{itemize}
\item at $a\in\ca$ a pseudomap $F^a\colon \ca\rightarrow \cb$, and at $b\in\cb$ a pseudomap $F_b\colon\catA\rightarrow\cc$ satisfying $F^a(b)=F_b(a)=:F(a,b)$;
\item at $A \colon a \to a^{\prime}$ and $B \colon b \to b^{\prime}$, a $2$-cell 
\[\begin{tikzcd}
	{F(a,b)} & {F(a',b)} \\
	{F(a,b')} & {F(a',b')}
	\arrow[""{name=0, anchor=center, inner sep=0}, "{F_b(A)}", from=1-1, to=1-2]
	\arrow["{F^a(B)}"', from=1-1, to=2-1]
	\arrow["{F^{a'}(B)}", from=1-2, to=2-2]
	\arrow[""{name=1, anchor=center, inner sep=0}, "{F_{b'}(A)}"', from=2-1, to=2-2]
	\arrow["{A_B}"', between={0.2}{0.8}, Rightarrow, from=0, to=1]
\end{tikzcd}\]
\black equipped with the structure of an adjoint equivalence $\aexi^A_B := A^o_B \dashv A_B$;\black
\item additional invertible $3$-cells $F^A_{B_2,B_1}, F_B^{A_2,A_1}$, $\alpha_B$ and $A_\beta$ satisfying equations described in Section 6.3 and Appendix D.1 of \cite{BourkeLobbia}.
\end{itemize}
\item its ternary and $4$-ary maps are as described in Section 6.3 of \cite{BourkeLobbia}.
\end{itemize}
The reader will not need to know the details of the multicategory $\GCatmp^l$ for the reading of this paper.   The following result will be of considerable importance, particularly in the later homotopical parts of the paper, so we outline the proof for the readers' convenience.

\begin{cor}[Corollary 6.13 of \cite{BourkeLobbia2}]\label{prop:biclosed}
The multicategory $\GCatmp^l$ is symmetric biclosed with internal hom $\psd_p(\ca,\cb)$.  In particular, we have isomorphisms $\lambda_{\ca,B,C} \colon \psd_p(\ca,\psd_p(\cb,\cc)) \rightsquigarrow \psd_p(\cb,\psd_p(\ca,\cc)) \in \GCat_p$ natural in each variable.
\end{cor}
\begin{proof}
Since $(\GCat,\otimes_p,\mathbf{1})$ is symmetric skew monoidal closed by Theorem 6.12 of \cite{BourkeLobbia2}, its associated multicategory $\GCatmp^l$ is symmetric biclosed.  In any symmetric biclosed multicategory, we have a unary isomorphism relating the iterated internal homs as above.  Indeed the left and right hand sides above respectively represent $4$-ary maps $\cd, \ca,\cb \rightsquigarrow \cc$ and $\cd, \cb,\ca \rightsquigarrow \cc$ and in a symmetric multicategory these are in natural bijection.
\end{proof}

\subsection{The left-sharp skew monoidal structure and its associated multicategory $\GCatmpshl$}\label{sect:sharpmulti}
In  \cite[Theorem 7.6]{BourkeLobbia}, we constructed a left skew monoidal structure $(\GCat,\leftsharp,\mathbf{1})$ that we here call the \emph{left-sharp skew monoidal structure}.

 We recall some of its properties now.
\begin{itemize}
\item The unit is the terminal Gray-category $\mathbf{1}$.
\item It is closed, with internal hom the Gray-category $\psd_s(\ca,\cb)$ of Gray-functors and pseudo-transformations.
\item It is both left and right normal.
\end{itemize}

Its associated multicategory $\GCatmpshl$ is, by definition, the wide sub-multicategory $\GCatmpshl \hookrightarrow \GCatmp^l$ of left-sharp multimaps.  In particular:
\begin{itemize}
\item All nullary maps are left-sharp.
\item The left-sharp unary maps are the Gray-functors $F \colon \ca \to \cb$.
\item A binary map $F \colon \ca,\cb \rightsquigarrow \cc$ is left-sharp when the associated pseudomaps $F^a\colon \cb\rightarrow \cb$ and $F_b\colon\catA\rightarrow\cc$ are strict, the invertible $3$-cells $F^{A_2,A_1}_{B}$ are identities \black and we have the equality of adjoint equivalences $\aexi^{A_2A_1}_B = (1 \circ \aexi^{A_1}_B) \cdot (\aexi^{A_2}_B \circ 1)$ depicted below.

\black
\[\begin{tikzcd}
	{\aexi^{A_2A_1}_B =} & {F^{a''}_{B}F^{A_2}_{b}F^{A_1}_{b}} & {F_{b'}^{A_2}F^{a'}_{B}F^{A_1}_{b}} & {F_{b'}^{A_2}F_{b'}^{A_1}F^{a}_{B}}
	\arrow[""{name=0, anchor=center, inner sep=0}, "{(A_2)_B \circ 1}"', curve={height=12pt}, from=1-2, to=1-3]
	\arrow[""{name=1, anchor=center, inner sep=0}, "{(A_2)^o_B \circ 1}"', curve={height=12pt}, from=1-3, to=1-2]
	\arrow[""{name=2, anchor=center, inner sep=0}, "{1 \circ (A_1)_B }"', curve={height=12pt}, from=1-3, to=1-4]
	\arrow[""{name=3, anchor=center, inner sep=0}, "{1 \circ (A_1)^o_B }"', curve={height=12pt}, from=1-4, to=1-3]
	\arrow["\bot"{description}, draw=none, from=0, to=1]
	\arrow["\bot"{description}, draw=none, from=3, to=2]
\end{tikzcd}\]
(\emph{Note that here we are using the corrected definition of \cite[Section 7]{BourkeLobbiaCorrected}, in which the equation concerning the equality of adjoint equivalences was added.})
\end{itemize}

As the multicategory associated to the left-sharp skew monoidal structure, $\GCatmpshl$ is closed with internal hom $\psd_s(\ca,\cb)$.  The following result relates left-sharp and loose multimaps via their associated internal homs.

\begin{lemma}\label{sharploose}
A loose multimap $F \colon \overline{\ca},\cb \rightsquigarrow \cc$ is left-sharp if and only if its adjoint $\overline{F} \colon \overline{\ca} \rightsquigarrow \psd_p(\cb,\cc)$ factors through $J_{\cb,\cc} \colon \psd_s(\cb,\cc) \hookrightarrow \psd_p(\cb,\cc)$ as a left-sharp multimap $\overline{\ca} \to \psd_s(\cb,\cc)$.
\end{lemma}
\begin{proof}
Consider a strict map $F \colon \ca \to \psd_s(\cb,\cc)$.  By the $n=1$ case in the proof of \cite[Theorem 7.5]{BourkeLobbiaCorrected}, the left-sharp multimap $\overline{F} \colon \ca,\cb \to \cc$ obtained through the closedness bijection $\GCatmpshl_1(\ca,\psd_p(\cb,\cc)) \cong \GCatmpshl_2(\ca,\cb;\cc)$ is the loose multimap obtained by applying the closedness bijection $\GCatmp^{l}_1(\ca,\psd_p(\cb,\cc)) \cong \GCatmp^{l}_2(\ca,\cb;\cc)$ to the composite $J \circ F \colon \ca \to \psd_s(\cb,\cc) \hookrightarrow \psd_p(\cb,\cc)$.  

Now taking the identity on $\psd_s(\cb,\cc)$ in the above, this says precisely that the universal binary map $e^{\sharp} \colon \psd_s(\cb,\cc),\cb \to \cc$ for $\GCatmpshl$ is obtained as the restriction $e^{l} \circ_1 J \colon \psd_s(\cb,\cc),\cb \rightsquigarrow \cc$ where $e^{l} \colon \psd_p(\cb,\cc),\cb \rightsquigarrow \cc$ is the universal loose binary map for the closed structure on $\GCatmp^l$.

As such, for all $n$ we obtain a commutative square

\[\begin{tikzcd}
	{\GCatmpshl_{n}(\overline{\ca};\psd_s(\cb,\cc))} & {\GCatmpshl_{n}(\overline{\ca},\cb;\cc)} \\
	{\GCatmp^{l}_{n}(\overline{\ca};\psd_p(\cb,\cc))} & {\GCatmp^{l}_{n}(\overline{\ca},\cb;\cc)}
	\arrow["{e^{\sharp} \circ_1-}", from=1-1, to=1-2]
	\arrow["{J \circ_1-}"', from=1-1, to=2-1]
	\arrow[hook, from=1-2, to=2-2]
	\arrow["{e^l \circ_1-}"', from=2-1, to=2-2]
\end{tikzcd}\]

Now, by closedness, the two horizontals are bijections.  And since any commutative square in which two parallel morphisms are isomorphisms is automatically a pullback, the above square is a pullback, proving the claim.
\end{proof}

\begin{term}
In this paper, we use the term \emph{left-sharp} for what in \cite{BourkeLobbiaCorrected} we called \emph{sharp}.  Therefore, we use \emph{left-sharp skew monoidal structure} and \emph{left sharp multimaps} instead of \emph{sharp skew monoidal structure} and \emph{sharp multimaps}.  The reason for this change is that \emph{left-sharpness} is a one-sided condition, as is evident from the description of the left-sharp binary maps above.  In \Cref{sect:flavours}, we will also consider the dual \emph{right-sharpness} and we reserve \emph{sharp} for the intersection of \emph{left-sharp} and \emph{right-sharp}.  In addition, we use the notation $\GCatmpshl$ for the multicategory which in \cite{BourkeLobbiaCorrected} we called $\GCatmp^{\sharp}$. 
\end{term}

\section{Flavours of semi-strict 4-category}\label{sect:flavours}

In the present section, we describe several new notions of semi-strict $4$-category and connect them with the existing literature.  We give a wide range of examples and develop the basic results.  

\subsection{Loose $4$-categories}
We begin by defining the most general kind of $4$-category that we will consider, the loose $4$-categories.

\begin{defn}
A \emph{loose $4$-category} is a category enriched over the multicategory $\GCatmp^l$.  
\end{defn}

Accordingly, a loose $4$-category $\ca$ has objects $x,y,z,\ldots$, hom-Gray-categories $\ca(x,y)$, loose multimaps $$\circ_{x,y,z} \colon \ca(y,z),\ca(x,y) \rightsquigarrow \ca(x,z) \hspace{0.5cm} \textnormal{and} \hspace{0.5cm} 1_x \colon \diamond\rightsquigarrow \ca(x,x)$$ satisfying the usual associativity and unit laws.  It follows that each loose $4$-category $\ca$ has an underlying category $\ca_1$ with composition denoted by $f \circ g$ or just $fg$.  

\begin{eg}[\emph{The loose $4$-category of Gray-categories and pseudomaps}]
Since $\GCatmp^l$ is closed, by Lemma \ref{basicenrichment}, there is a loose $4$-category $\GCat_p^4$ whose objects are Gray-categories and with homs $\psd_p(\ca,\cb)$: that is, its $1$-cells are pseudomaps, $2$-cells are pseudo-transformations, $3$-cells are pseudo-modifications and $4$-cells are perturbations.  
\end{eg}

\begin{eg}[\emph{The opposite loose $4$-category}]
Let $\ca$ be a loose $4$-category.  Since  $\GCatmp^l$ is symmetric, by \Cref{basicenrichment},we can form the opposite loose $4$-category $\ca^{op}$ which has the same objects as $\ca$ but $\ca^{op}(x,y) = \ca(y,x)$.
\end{eg}

\begin{defn}
Let $\ca$ and $\cb$ be loose $4$-categories.  A loose $4$-functor $F \colon \ca \rightsquigarrow \cb$ is a $\GCatmp^l$-enriched functor.
\end{defn}

As such, a loose $4$-functor $F \colon \ca \rightsquigarrow \cb$ involves pseudomaps $F_{x,y} \colon \ca(x,y) \rightsquigarrow \cb(Fx,Fy)$ which commute strictly with the composition and identity multimaps.  We say that $F$ is a \emph{strict} $4$-functor if each $F_{x,y}$ is a Gray-functor, in which case we denote it by a straight arrow  $F \colon \ca \to \cb$.

\begin{eg}[\emph{Representables}]\label{eg:rep}
If $\ca$ is a loose $4$-category, then by \Cref{basicenrichment}, we can form the \emph{representable} loose $4$-functors $\ca(x,-) \colon \ca \rightsquigarrow \GCat_p^4$ and $\ca(-,x) \colon \ca^{op} \rightsquigarrow \GCat_p^4$.
\end{eg}




\begin{notation*}Before going further, let us set out some basic notation for a loose $4$-category $\ca$. 
\begin{itemize}
	 \item Our basic system for $0$-cells, $1$-cells and $2$-cells will follow the notation below.
	\[\begin{tikzcd}
		w && x && y && z
		\arrow[""{name=0, anchor=center, inner sep=0}, "c", curve={height=-12pt}, from=1-1, to=1-3]
		\arrow[""{name=1, anchor=center, inner sep=0}, "{c'}"', curve={height=12pt}, from=1-1, to=1-3]
		\arrow[""{name=2, anchor=center, inner sep=0}, "b", curve={height=-12pt}, from=1-3, to=1-5]
		\arrow[""{name=3, anchor=center, inner sep=0}, "{b'}"', curve={height=12pt}, from=1-3, to=1-5]
		\arrow[""{name=4, anchor=center, inner sep=0}, "a", curve={height=-12pt}, from=1-5, to=1-7]
		\arrow[""{name=5, anchor=center, inner sep=0}, "{a'}"', curve={height=12pt}, from=1-5, to=1-7]
		\arrow["C", between={0.2}{0.8}, Rightarrow, from=0, to=1]
		\arrow["B", between={0.2}{0.8}, Rightarrow, from=2, to=3]
		\arrow["A", between={0.2}{0.8}, Rightarrow, from=4, to=5]
	\end{tikzcd}\]
	\begin{itemize}
	\item Objects will be denoted with $w,x,y,z$.
	\item $A$'s will always be on the right of $B$'s and similarly for $C$'s and so on. 
	\item  For composable 2-cells we will use increasing subscripts, e.g. $A_1\colon a\to a'$ and $A_2\colon a'\to a''$. 
	\item Parallel $n$-cells will be denoted using $(-)'$, e.g. $a,a'\colon y\to z$ and $A,A'\colon a\to a'$. 
	\end{itemize}
	
	\item We denote 3-cells by $\gamma\colon C\to C'$, $\beta\colon B\to B'$, and $\alpha\colon A\to A'$, and 4-cells by $\Theta\colon\beta\to\beta'$ and $\Lambda\colon\alpha\to\alpha'$.
	\item The hom-Gray-categories $\ca(x,y)$ involve whiskering operations by $1$-cells --- that is, $2$-cells in $\ca$ --- which will be denoted by  $-.B$ or $A.-$.  In particular, for composable $2$-cells $B_1\colon b\to b'$ and $B_2\colon b'\to b''$ (along a 1-cell $b'$) 2-cells we write $B_2.B_1\colon b\to b''$.  Given vertically composable 3-cells $\beta_1\colon B\to B'$ and $\beta_2\colon B'\to B''$ (i.e. 2-cells in $\ca(x,y)$), we denote their composite by $\beta_2\ast\beta_1\colon B\to B''$.  No notation for vertical composition of $4$-cells will be needed since such composites will be represented using pasting diagrams.  Finally, the interchanger in $\ca(x,y)$ will be called the \emph{$3$-interchanger}, since it describes a $4$-cell $\alpha \colon \beta$ associated to a pair of \emph{$3$-cells} $\alpha$ and $\beta$.  This distinguishes it from the \emph{2-interchanger} $A_B$ introduced in \Cref{unpacked} below.
%
%

\item The loose bimap $-\circ- \colon \ca(y,z),\ca(x,y) \rightsquigarrow \ca(x,z)$ involves pseudomaps $a \circ - \colon \ca(x,y) \rightsquigarrow \ca(x,z)$ and $- \circ b \colon \ca(y,z) \rightsquigarrow \ca(x,z)$ describing whiskering by $1$-cells, whose action on $n$-cells we denote by $\circ$ or juxtaposition.    For instance, we will write $a\circ b=ab\colon x\to z$ for a composite $1$-cell and $a \circ B = aB \colon ab\to ab'$ for a whiskered $2$-cell.
\end{itemize}
\end{notation*}

\subsection{The stricter flavours}
In this section, we introduce the stricter flavours of $4$-category that we will consider.  We do so by imposing additional conditions on loose $4$-categories.

\begin{defn}
A \emph{moderate $4$-category} $\ca$ is a loose $4$-category such that each multimap $\circ_{x,y,z} \colon \ca(y,z),\ca(x,y) \rightsquigarrow \ca(x,z)$ is a strict Gray-functor in each variable.
\end{defn}

\begin{rmk}Moderate $4$-categories are equally categories enriched over the sub-multicategory of $\GCatmp^l$ containing loose multimaps which are strict in each variable.  This sub-multicategory is symmetric, but not closed on either side.
\end{rmk}

\begin{rmk}[\emph{The unpacked definition}]\label{unpacked}
Since moderate $4$-categories are a bit more manageable than loose ones, and since they include most of our examples, we spell out the data and axioms for a moderate 4-category in elementary terms.  
A \midfourcat\ $\ca$ consists of:
\begin{itemize}
	\item a collection $\ca_0$ of objects (also called $0$-cells);
	\item for any $x,y\in\ca_0$, a Gray-category $\ca(x,y)$ of (1/2/3/4)-cells;
	\item for any $x\in\ca_0$, a $1$-cell $1_x\in\ca(x,x)$;
	\item for any $x,y,z\in\ca_0$ a loose bimap $-\circ-\colon\ca(y,z),\ca(x,y)\rightsquigarrow \ca(x,z)$ which is strict in each variable, i.e. the data of
\begin{enumerate}
	\item For each 1-cell $a \in \ca(y,z)$, a Gray-functor $a \circ - \colon \ca(x,y) \to \ca(x,z)$;
	\item For each 1-cell $b \in \ca(x,y)$, a Gray-functor $- \circ b \colon \ca(y,z) \to \ca(x,z)$;
	\item For 2-cells $A\colon a\to a' \in \ca(y,z)$ and $B\colon b\to b' \in \ca(x,y)$, a $3$-cell 
	\[\begin{tikzcd}
		ab && {a'b} \\
		{ab'} && {a'b'}
		\arrow[""{name=0, anchor=center, inner sep=0}, "Ab", from=1-1, to=1-3]
		\arrow["aB"', from=1-1, to=2-1]
		\arrow["{a'B}", from=1-3, to=2-3]
		\arrow[""{name=1, anchor=center, inner sep=0}, "{Ab'}"', from=2-1, to=2-3]
		\arrow["{A_B}", between={0.2}{0.8}, Rightarrow, from=0, to=1]
	\end{tikzcd}\]
	called the \emph{2-interchanger} equipped with the structure of an adjoint equivalence $\aexi^A_B:=\,A_B^o\dashv A_B$;
	\item At $(A_2,A_1:B)$, 
		an invertible $4$-cell
\[\begin{tikzcd}[ampersand replacement=\&]
	ab \& {a'b} \& {a''b} \&\& ab \& {a'b} \& {a''b} \\
	{ab'} \& {a'b'} \& {a''b'} \&\& {ab'} \& {a'b'} \& {a''b'}
	\arrow["{A_1b}", from=1-1, to=1-2]
	\arrow["aB"', from=1-1, to=2-1]
	\arrow["{A_2b}", from=1-2, to=1-3]
	\arrow["{(A_2.A_1)_B}", Rightarrow, nfold, from=1-2, to=2-2]
	\arrow[""{name=0, anchor=center, inner sep=0}, "{a'B}", from=1-3, to=2-3]
	\arrow[""{name=1, anchor=center, inner sep=0}, "{A_1b}", from=1-5, to=1-6]
	\arrow[""{name=2, anchor=center, inner sep=0}, "aB"', from=1-5, to=2-5]
	\arrow[""{name=3, anchor=center, inner sep=0}, "{A_2b}", from=1-6, to=1-7]
	\arrow["{a'B}"{description}, from=1-6, to=2-6]
	\arrow["{a'B}", from=1-7, to=2-7]
	\arrow["{A_1b'}"', from=2-1, to=2-2]
	\arrow["{A_2b'}"', from=2-2, to=2-3]
	\arrow[""{name=4, anchor=center, inner sep=0}, "{A_1b'}"', from=2-5, to=2-6]
	\arrow[""{name=5, anchor=center, inner sep=0}, "{A_2b'}"', from=2-6, to=2-7]
	\arrow["{(A_2,A_1)_B}", between={0.4}{0.6}, Rightarrow, scaling nfold=3, from=0, to=2]
	\arrow["{(A_1)_B}"{description}, between={0.2}{0.8}, Rightarrow, nfold, from=1, to=4]
	\arrow["{(A_2)_B}"{description}, between={0.2}{0.8}, Rightarrow, nfold, from=3, to=5]
\end{tikzcd}\]
		
		\item At $(A:B_1,B_2)$, an invertible $4$-cell
		\[\begin{tikzcd}[ampersand replacement=\&]
			ab \&\& {a'b} \&\&\& ab \&\& {a'b} \\
			{ab'} \&\& {a'b'} \&\&\& {ab'} \&\& {a'b'} \\
			{ab''} \&\& {a'b''} \&\&\& {ab''} \&\& {a'b''}
			\arrow[""{name=0, anchor=center, inner sep=0}, "Ab", from=1-1, to=1-3]
			\arrow["{aB_1}"', from=1-1, to=2-1]
			\arrow["{a'B_1}", from=1-3, to=2-3]
			\arrow[""{name=1, anchor=center, inner sep=0}, "Ab", from=1-6, to=1-8]
			\arrow["{aB_1}"', from=1-6, to=2-6]
			\arrow["{a'B_1}", from=1-8, to=2-8]
			\arrow[""{name=2, anchor=center, inner sep=0}, "{Ab'}"{description}, from=2-1, to=2-3]
			\arrow["{aB_2}"', from=2-1, to=3-1]
			\arrow["{A_{(B_2,B_1)}}"{yshift=0.1cm}, between={0.3}{0.7}, Rightarrow, scaling nfold=3, from=2-3, to=2-6]
			\arrow["{a'B_2}", from=2-3, to=3-3]
			\arrow["{aB_2}"', from=2-6, to=3-6]
			\arrow["{a'B_2}", from=2-8, to=3-8]
			\arrow[""{name=3, anchor=center, inner sep=0}, "{Ab''}"', from=3-1, to=3-3]
			\arrow[""{name=4, anchor=center, inner sep=0}, "{Ab''}"', from=3-6, to=3-8]
			\arrow["{A_{B_1}}", between={0.3}{0.7}, Rightarrow, from=0, to=2]
			\arrow["{A_{B_2.B_1}}", shift right=4, between={0.4}{0.6}, Rightarrow, from=1, to=4]
			\arrow["{A_{B_2}}", between={0.3}{0.7}, Rightarrow, from=2, to=3]
		\end{tikzcd}\]
		
		\item At $(\alpha:B)$, an invertible $4$-cell
\[\begin{tikzcd}[ampersand replacement=\&]
	ab \&\& {a'b} \&\&\& ab \&\& {a'b} \\
	\\
	{ab'} \&\& {a'b'} \&\&\& {ab'} \&\& {a'b'}
	\arrow[""{name=0, anchor=center, inner sep=0}, "Ab", curve={height=-12pt}, from=1-1, to=1-3]
	\arrow[""{name=1, anchor=center, inner sep=0}, "{A'b}"', curve={height=12pt}, from=1-1, to=1-3]
	\arrow["aB"', from=1-1, to=3-1]
	\arrow[""{name=2, anchor=center, inner sep=0}, "{a'B}", from=1-3, to=3-3]
	\arrow[""{name=3, anchor=center, inner sep=0}, "Ab", curve={height=-12pt}, from=1-6, to=1-8]
	\arrow[""{name=4, anchor=center, inner sep=0}, "aB"', from=1-6, to=3-6]
	\arrow["{a'B}", from=1-8, to=3-8]
	\arrow[""{name=5, anchor=center, inner sep=0}, "{A'b'}"', curve={height=12pt}, from=3-1, to=3-3]
	\arrow[""{name=6, anchor=center, inner sep=0}, "{A'b'}"', curve={height=12pt}, from=3-6, to=3-8]
	\arrow[""{name=7, anchor=center, inner sep=0}, "{Ab'}", curve={height=-12pt}, from=3-6, to=3-8]
	\arrow["{\alpha b}", between={0.2}{0.8}, Rightarrow, nfold, from=0, to=1]
	\arrow["{A'_B}", between={0.3}{0.7}, Rightarrow, nfold, from=1, to=5]
	\arrow["{\alpha_B}", between={0.3}{0.7}, Rightarrow, scaling nfold=3, from=2, to=4]
	\arrow["{A_B}", between={0.3}{0.7}, Rightarrow, nfold, from=3, to=7]
	\arrow["{\alpha b'}", between={0.2}{0.8}, Rightarrow, nfold, from=7, to=6]
\end{tikzcd}\]
		
		\item At $(A:\beta)$, an invertible $4$-cell
\[\begin{tikzcd}[ampersand replacement=\&]
	ab \&\& {a'b} \&\&\& ab \&\& {a'b} \\
	\\
	{ab'} \&\& {a'b'} \&\&\& {ab'} \&\& {a'b'}
	\arrow[""{name=0, anchor=center, inner sep=0}, "Ab", from=1-1, to=1-3]
	\arrow[""{name=1, anchor=center, inner sep=0}, "{aB'}"', curve={height=18pt}, from=1-1, to=3-1]
	\arrow[""{name=2, anchor=center, inner sep=0}, "aB", curve={height=-18pt}, from=1-1, to=3-1]
	\arrow[""{name=3, anchor=center, inner sep=0}, "{a'B}"{description}, curve={height=-18pt}, from=1-3, to=3-3]
	\arrow[""{name=4, anchor=center, inner sep=0}, "Ab", from=1-6, to=1-8]
	\arrow[""{name=5, anchor=center, inner sep=0}, "{aB'}"{description}, curve={height=18pt}, from=1-6, to=3-6]
	\arrow[""{name=6, anchor=center, inner sep=0}, "{a'B}", curve={height=-18pt}, from=1-8, to=3-8]
	\arrow[""{name=7, anchor=center, inner sep=0}, "{a'B'}"', curve={height=18pt}, from=1-8, to=3-8]
	\arrow[""{name=8, anchor=center, inner sep=0}, "{Ab'}"', from=3-1, to=3-3]
	\arrow[""{name=9, anchor=center, inner sep=0}, "{Ab'}"', from=3-6, to=3-8]
	\arrow["{a\beta}"', between={0.3}{0.7}, Rightarrow, nfold, from=2, to=1]
	\arrow["{A_B}", shift left=5, between={0.3}{0.7}, Rightarrow, nfold, from=0, to=8]
	\arrow["{A_{\beta}}", between={0.3}{0.7}, Rightarrow, scaling nfold=3, from=3, to=5]
	\arrow["{A_{B'}}"', shift right=5, between={0.3}{0.7}, Rightarrow, nfold, from=4, to=9]
	\arrow["{a'\beta}"', between={0.3}{0.7}, Rightarrow, nfold, from=6, to=7]
\end{tikzcd}\]
		
	\end{enumerate}
\end{itemize}
This data needs to satisfy the axioms in \Cref{app:axioms-4-cats}. 
\end{rmk}

We are now in a position to define the remaining flavours of semi-strict $4$-category.

\begin{defn}\label{variations}
Let $\ca$ be a moderate $4$-category.  We say that $\ca$ is:
\begin{itemize}
\item a \emph{left-sharp} $4$-category  if each $4$-cell $(A_2,A_1)_B$ is the identity \black and we have the equality of adjoint equivalences $\aexi^{A_2A_1}_B = (1 \circ \aexi^{A_1}_B) \cdot (\aexi^{A_2}_B \circ 1)$;\black
\item  a \emph{right-sharp} $4$-category if each $4$-cell $A_{(B_2,B_1)}$ is the identity \black and we have the equality of adjoint equivalences $\aexi^{A}_{B_2B_1} = (\aexi^{A}_{B_2} \circ 1)\cdot (1 \circ \aexi^{A}_{B_1})$;\black
\item a \emph{sharp} $4$-category if it is left-sharp and right-sharp;
\item a \emph{4-tas} \cite{Crans2000On} if it is sharp and each adjoint equivalence $\aexi^A_B$ is an isomorphism.
\item a \emph{strict} $4$-category if is a $4$-tas such that each of $A_B$, $\alpha_B$ and $A_\beta$ is the identity.
\end{itemize}
\end{defn}

\begin{rmk}[\emph{A word on the axioms}]
Apart from the relatively trivial ones --- concerning identities and asserting that whiskering by $1$-cells is functorial and associative --- there are $15 = (6 \times 2) + 3$ further axioms for a moderate $4$-category.  These are the $15$ axioms depicted using 2-categorical pasting diagrams in \Cref{app:axioms-4-cats}.  

Of these, there are $6$ pairs of dual axioms, so ignoring duals, there are 9 substantial axioms.  None of these are particularly complex, with only Axiom \eqref{ax:4-cat_A_1,A_2:B_1,B_2} containing a composite of length greater than $3$ (it has length $5$).  

The situation simplifies further if we restrict our attention to some the stricter flavours described in \Cref{variations}.  For instance, in a left-sharp $4$-category, the invertible $4$-cells $(A_1,A_1)_B$ and $A_{(B_2,B_1)}$ are replaced by identity $4$-cells (thus are replaced by a pair of equations).  In this case, the two dual axioms \eqref{ax:4-cat_A_1,A_2,A_3:B} and \eqref{ax:4-cat_A:B_1,B_2,B_3} are  then redundant, whilst many further axioms simplify considerably: for instance, Axiom \eqref{ax:4-cat_A_1,A_2:B_1,B_2} reduces from involving a composite of length $5$ to a ``composite" of length $1$.
\end{rmk}


The following diagram depicts the relationships between the non-strict flavours of semi-strict $4$-categories.

\[\begin{tikzcd}[ampersand replacement=\&]
	\&\&\& {\textnormal{Left-sharp}} \&\&\& \\
	{\textnormal{4-teisi}} \&\& {\textnormal{Sharp}} \&\& {\textnormal{Moderate}} \&\& {\textnormal{Loose}} \\
	\&\&\& {\textnormal{Right-sharp}}
	\arrow[hook, from=1-4, to=2-5]
	\arrow[hook, from=2-1, to=2-3]
	\arrow[hook, from=2-3, to=1-4]
	\arrow[hook, from=2-3, to=3-4]
	\arrow[hook, from=2-5, to=2-7]
	\arrow[hook, from=3-4, to=2-5]
\end{tikzcd}\]

We now investigate the different flavours in more detail.

\subsubsection{Moderate $4$-categories}
We begin by describing several examples of moderate $4$-categories.
\begin{eg}[\emph{Gray-categories}]
A Gray-category is precisely a moderate $4$-category in which all $4$-cells are identities.  
On the other hand loose $4$-categories in which all $4$-cells are identities are more general than Gray-categories, since their whiskering operations  $a \circ - \colon \ca(x,y) \rightsquigarrow \ca(x,z)$ and $- \circ b \colon \ca(y,z) \rightsquigarrow \ca(x,z)$ by $1$-cells $a$ and $b$ are normal pseudo-functors in the sense of $2$-category theory (see \Cref{normalpseudo}) rather than strict $2$-functors.
\end{eg}

\begin{eg}[\emph{Braided monoidal 2-categories}]\label{braided}

A definition of (semi-strict) braided monoidal 2-category was first given by Kapranov and Voevodsky \cite{Kapranov}, improved by Baez and Neuchl \cite{Baez1996}, Day and Street \cite{Day} and refined to the form we will consider by Crans \cite{Crans1998Generalized}, who added further degeneracy conditions.  Gurski \cite{Gurski2011} proved that each braided monoidal bicategory is equivalent to a braided monoidal 2-category.

We now show that moderate $4$-categories with a single $0$-cell and $1$-cell are precisely the same as braided monoidal $2$-categories, in Crans' sense \cite{Crans1998Generalized}.  \black This correspondence was inspired by \cite[Example 3.1.22]{MirandaThesis} and perfects that result --- see \Cref{mirandacomparison} for further details.\black

To get started, we start by observing that the data for such a moderate $4$-category $\ca$ starts with a single object $\bullet$ together with a $1$-object Gray-category $\cb = \ca(\bullet,\bullet)$.  

The $1$-object Gray-category equally specifies a \emph{Gray-monoid} $(\cc,.,I)$, where $\cc = \cb(1_\bullet,1_\bullet)$ is the $2$-category, the object $I$ is the identity $2$-cell on $1_\bullet$, and the tensor product $2$-functor $. \colon \cc \otimes_p \cc \to \cc$ with domain the Gray-tensor product encodes vertical composition $A.B$ of $2$-cells in $\ca$.

The remaining data for the moderate $4$-category $\ca$ are the seven pieces of data (1-7) for the composition bimap $\circ \colon \ca(\bullet,\bullet),  \ca(\bullet,\bullet) \rightsquigarrow \ca(\bullet,\bullet)$ in \Cref{unpacked}.  (1-2) provide whiskering Gray-functors by $1$-cells in $\ca$.  However, since all $1$-cells are identities, the axioms for \emph{unitality of the composition bimap} in a moderate $4$-category (\Cref{bimapunit}) ensure that these are the identity Gray-functors, and so redundant.

As such we are left with the five pieces of data $A_B$, $A_\beta$, $\alpha_B$, $A_{(B_2,B_1)}, (A_1,A_2)_B$.  Most fundamental is the adjoint equivalence $3$-cell

\[\begin{tikzcd}
		ab && {a'b} \\
		{ab'} && {a'b'}
		\arrow[""{name=0, anchor=center, inner sep=0}, "Ab", from=1-1, to=1-3]
		\arrow["aB"', from=1-1, to=2-1]
		\arrow["{a'B}", from=1-3, to=2-3]
		\arrow[""{name=1, anchor=center, inner sep=0}, "{Ab'}"', from=2-1, to=2-3]
		\arrow["{A_B}", between={0.2}{0.8}, Rightarrow, from=0, to=1]
	\end{tikzcd}\]
which specifies an adjoint equivalence $1$-cell $$A_B \colon B.A \to A.B$$ in $\cc$.  The remaining four pieces of data take the form of invertible $2$-cells

\[\begin{tikzcd}[ampersand replacement=\&]
	{B.A} \& {A.B} \&\& {B.A} \& {A.B} \\
	{B.A'} \& {A'.B} \&\& {B'.A} \& {A.B'}
	\arrow["{A_B}", from=1-1, to=1-2]
	\arrow[""{name=0, anchor=center, inner sep=0}, "{1.\alpha}"', from=1-1, to=2-1]
	\arrow[""{name=1, anchor=center, inner sep=0}, "{\alpha.1}", from=1-2, to=2-2]
	\arrow["{A_B}", from=1-4, to=1-5]
	\arrow[""{name=2, anchor=center, inner sep=0}, "{\beta.1}"', from=1-4, to=2-4]
	\arrow[""{name=3, anchor=center, inner sep=0}, "{1.\beta}", from=1-5, to=2-5]
	\arrow["{A'_B}"', from=2-1, to=2-2]
	\arrow["{A'_B}"', from=2-4, to=2-5]
	\arrow["{\alpha_B}", between={0.2}{0.8}, Rightarrow, nfold, from=0, to=1]
	\arrow["{A_\beta}"', between={0.2}{0.8}, Rightarrow, nfold, from=3, to=2]
\end{tikzcd}\]

and 

\[\begin{tikzcd}[ampersand replacement=\&]
	{B.A_2.A_1} \&\& {A_2.A_1.B} \&\& {B_2.A.B_1} \& \\
	\& {A_2.B.A_1} \&\& {B_2.B_1.A} \&\& {A.B_2.B_1}
	\arrow[""{name=0, anchor=center, inner sep=0}, "{(A_2.A_1)_B}", from=1-1, to=1-3]
	\arrow["{(A_2)_B.1}"'{pos=0.3}, from=1-1, to=2-2]
	\arrow["{A_{B_2}.1}", from=1-5, to=2-6]
	\arrow["{1.(A_1)_B}"', from=2-2, to=1-3]
	\arrow["{1.A_{B_1}}", from=2-4, to=1-5]
	\arrow[""{name=1, anchor=center, inner sep=0}, "{A_{B_2.B_1}}"', from=2-4, to=2-6]
	\arrow["{(A_2,A_1)_B}", shift right=2, between={0.2}{1}, Rightarrow, nfold, from=0, to=2-2]
	\arrow["{A_{(B_2,B_1)}}"', shift left=2, between={0}{0.8}, Rightarrow, nfold, from=1-5, to=1]
\end{tikzcd}\]

in the $2$-category $\cc$.  

On the other hand a braided monoidal $2$-category structure on the Gray-monoid $(\cc,.,I)$ begins with a pseudo-natural adjointt equivalence $R$\begin{footnote}{In fact Crans \cite{Crans1998Generalized} requires $R$ to be a pseudonatural equivalence, without explicitly mentioning adjointness.  Later authors' \cite{Gurski2011, SchommerPries} explicitly include the adjointness here as it makes the definition of braided monoidal 2-category \emph{algebraic} and we shall do the same.}\end{footnote} as below
\[\begin{tikzcd}[ampersand replacement=\&]
	{\cc \otimes_p \cc} \&\& \cc \\
	\& {\cc \otimes_p \cc}
	\arrow[""{name=0, anchor=center, inner sep=0}, "{.}", from=1-1, to=1-3]
	\arrow["s"', from=1-1, to=2-2]
	\arrow["{.}"', from=2-2, to=1-3]
	\arrow["R"', between={0.2}{1}, Rightarrow, nfold, from=0, to=2-2]
\end{tikzcd}\]
where $s$ is the symmetry isomorphism for the Gray-tensor product.

On objects, this amounts to adjoint equivalences $R_{A,B} \colon A.B \to B.A$ and, accordingly,  we set $R_{A,B} := B_A \colon A.B \to B.A$ to begin the comparison. The additional data for the pseudonatural adjoint equivalence $R$ is specified by the invertible $3$-cells $\alpha_B$ and $A_\beta$ as above.

The remaining data for a braided monoidal $2$-category is specified by two invertible modifications $\tilde{R}_{(-|-,-)}$ and $\tilde{R}_{(-,-|-)}$ whose components $\tilde{R}_{(B|A_2,A_1)}$ and  $\tilde{R}_{(B_2, B_1|A)}$ are exactly as our two final invertible $2$-cells above with the exception that $\tilde{R}_{(B_2, B_1|A)}$ is oriented in the opposite direction to our $A_{(B_2,B_1)})$.   As such, we set $\tilde{R}_{(B|A_2,A_1)} := {(A_2,A_1)_B}$ and $\tilde{R}_{(B_2, B_1|A)}:= (A_{(B_2,B_1)})^{-1}$.

In summary, we have a bijection between the data for a moderate $4$-category $\ca$ with a single $0$-cell and $1$-cell (with unitality of $\circ$ already imposed) and that of a braided monoidal $2$-category $\cc$.  Finally we compare the axioms.

To begin with, Crans requires that $R$ is pseudonatural.  Pseudonaturality of $R$ in each variable gives three axioms each --- our bimap axioms' \eqref{ax:4-cat_A:Theta}, \eqref{ax:4-cat_A:beta_1,beta_2}, \eqref{ax:bin_D_1a:B} in the first and the dual ones \eqref{ax:4-cat_Lambda:B}, \eqref{ax:4-cat_alpha_1,alpha_2:B}, \eqref{ax:bin_D_A:1b}  in the second --- plus an additional naturality equation at the interchanger $2$-cell in $\cc \otimes_p \cc$, amounting to our Axiom \eqref{ax:4-cat_alpha:beta}.

Crans further requires that $\tilde{R}_{(-|-,-)}$ and $\tilde{R}_{(-,-|-)}$ are modifications (in each variable) corresponding to six conditions, one for each position: namely our bimap axioms \eqref{ax:4-cat_alpha,A_2:B}, \eqref{ax:4-cat_A_1,alpha:B}, \eqref{ax:4-cat_A_1,A_2:beta}, \eqref{ax:4-cat_alpha:B_1,B_2}, \eqref{ax:4-cat_A:beta,B_2} and \eqref{ax:4-cat_A:B_1,beta}.

Crans has four further (non-degeneracy) axioms which are the most complex.  The first three, drawn in Diagrams 2.4, 2.5 and 2.6 of \cite{Crans1998Generalized} correspond to our bimap axioms \eqref{ax:4-cat_A_1,A_2,A_3:B}, \eqref{ax:4-cat_A:B_1,B_2,B_3} and \eqref{ax:4-cat_A_1,A_2:B_1,B_2} whilst the fourth one, Diagram 2.7 of \cite{Crans1998Generalized}, corresponds to our incubator axiom \eqref{ax:4-cat_A_B_C}.

Crans also imposes 8 degeneracy axioms.  The first two state that the pseudonatural transformations $R_{I,-}$ and $R_{-,I}$ are the identity.  Since a pseudonatural transformation involves two pieces of data, this yields four equations.  These are four of our degeneracy equations for a bimap --- namely \eqref{ax:bin_D_1a:B}, \eqref{ax:bin_D_1a:bet}, \eqref{ax:bin_D_A:1b} and \eqref{ax:bin_D_alph:1b} --- of which the first and third were covered already.  Crans' other six degeneracy axioms state that the modifications $\tilde{R}_{(-|-,-)}$ and $\tilde{R}_{(-,-|-)}$ have identity component when $I$ is plugged into a variable.  These correspond to our final bimap degeneracy axioms \eqref{ax:bin_D_A1A2:B}, 
\eqref{ax:bin_D_A:B1B2} (which are two equations each), \eqref{ax:bin_D_A1A2:1b} and \eqref{ax:bin_D_1a:B1B2}.  

This covers all of Crans' axioms from \cite{Crans1998Generalized}.  On the side of moderate $4$-categories, we still have the associativity axioms of \Cref{bimapassociativity} which concern associativity of whiskering by $1$-cells, but since all $1$-cells are identities and we have assumed that the composition bimap is unital, these are redundant.  This completes the correspondence.

\end{eg}

\black
\begin{rmk}[Comparison with the Gray-4-categories of Fukuda and Hu]\label{fukudacomparison}
Fukuda and Hu \cite{Fukuda} recently introduced their notion of a Gray-4-category, motivated by the notions of 3-crossed module and Gray-3-groups.   Their \emph{Gray 4-categories} are precisely moderate 4-categories in which the adjoint equivalence $3$-cell $\aexi^A_B = A^o_B \dashv A_B$ is an isomorphism.\begin{footnote}{We note that their degeneracy axiom 2.46 does not seem to parse.  In order to make the comparison precise, our assumption is that their 2.46 is intended to capture our axioms \eqref{ax:bin_D_1a:B}, \eqref{ax:bin_D_1a:bet}, \eqref{ax:bin_D_A:1b} and \eqref{ax:bin_D_alph:1b}.  In addition, they have different conventions concerning the orientation of several of the invertible structural 4-cells, which lead to axioms that appear slightly different, but are equivalent. }
\end{footnote}
 In Lemma 3.4 of \cite{Fukuda}, they state that doubly degenerate Gray-4-categories coincide with braided monoidal 2-categories.  In fact, they correspond to those braided monoidal 2-categories for which $A_B$ is an isomorphism.
\black
\end{rmk}
\black

To our knowledge, monoidal Gray-categories have not previously been defined within the literature.  However, having seen above that braided monoidal $2$-categories are moderate $4$-categories with a single $0$-cell and $1$-cell, it is natural to make the following definition.

\begin{defn}
A monoidal Gray-category is a moderate $4$-category with one object.
\end{defn}

We now spell out what the concept means in more elementary terms.  A monoidal Gray-category $(\ca,\otimes,i)$ amounts to:
\begin{itemize} 
\item a Gray-category $\ca$;
\item a binary operation $\otimes$ on objects (denoted by juxtaposition $a \otimes b := ab$) and a unit object $i$ which are associative and unital;
\item extensions of these object assignments to Gray-functors $a-\colon \ca \to \ca$ and $-b \colon \ca \to \ca$ for each $a,b \in \ca$;
\item a $2$-cell $A_B$
	\[\begin{tikzcd}
		ab && {a'b} \\
		{ab'} && {a'b'}
		\arrow[""{name=0, anchor=center, inner sep=0}, "Ab", from=1-1, to=1-3]
		\arrow["aB"', from=1-1, to=2-1]
		\arrow["{a'B}", from=1-3, to=2-3]
		\arrow[""{name=1, anchor=center, inner sep=0}, "{Ab'}"', from=2-1, to=2-3]
		\arrow["{A_B}", between={0.2}{0.8}, Rightarrow, from=0, to=1]
	\end{tikzcd}\]
\black equipped with the structure of an adjoint equivalence $\aexi^A_B = A^o_B \dashv A_B$ as in Item 3 of \Cref{unpacked} (now a $2$-cell rather than a $3$-cell) plus the remaining data in Items' 4-7 \Cref{unpacked} (again with cell dimension decreased by $1$), all satisfying the same axioms.
\end{itemize}

\begin{eg}
Since braided monoidal $2$-categories are the moderate $4$-categories with a single object and $1$-cell, braided monoidal $2$-categories are precisely the monoidal Gray-categories with a single object.
\end{eg}

In \Cref{monoidalGray} and \Cref{centre} we will give examples of monoidal Gray-categories and braided monoidal $2$-categories arising from left-sharp $4$-categories.

\subsubsection{Left-sharp $4$-categories}

These are particularly important since Gray-categories, Gray-functors, pseudo-transformations and their higher transfors form a left-sharp $4$-category.  They are also well-behaved thanks to the following proposition.
\begin{prop}
Left-sharp $4$-categories are precisely:
\begin{enumerate}
\item categories enriched over the multicategory $\GCatmpshl$ of left-sharp multimaps;
\item categories enriched over the left-sharp skew monoidal structure $(\GCat,\leftsharp,\mathbf{1})$.
\end{enumerate}
\end{prop}
\begin{proof}
For (1), by definition a loose $4$-category $\ca$ is a category enriched in the multicategory $\GCatmp^l$.  Since the inclusion multifunctor $\GCatmpshl \hookrightarrow \GCatmp^l$ is full on nullary arrows, to say that $\ca$ is enriched in $\GCatmp^l$ is equally to say that the composition bimaps $-\circ -\colon \ca(y,z),\ca(x,y) \to \ca(x,z) \in \GCatmp^l$ belong to $\GCatmpshl$: i.e. are left-sharp.  By the description of left-sharp bimaps in \Cref{sect:sharpmulti}, this is precisely to say that $\ca$ is a left-sharp $4$-category.

For (2), since the multicategory $\GCatmpshl$ is represented by the left-sharp skew monoidal structure $(\GCat,\leftsharp,\mathbf{1})$, categories enriched over the multicategory $\GCatmpshl$ coincide with categories enriched over the representing skew monoidal structure $(\GCat,\leftsharp,\mathbf{1})$.
\end{proof}

\begin{term}
One justification for the terminology \emph{left-sharp} vs \emph{right-sharp} is that left-sharp $4$-categories are categories enriched in the left skew monoidal structure $(\GCat,\leftsharp,\mathbf{1})$ whereas right-sharp $4$-categories are categories enriched in a right skew monoidal structure (see \Cref{sect:rightsharp} below).
\end{term}

\begin{eg}[\emph{The left-sharp $4$-category $\GCat^4$ of Gray-categories and Gray-functors}]
An example of fundamental importance is the left-sharp $4$-category $\GCat^4$ since it answers the question of what Gray-categories and their higher transfors form.

$\GCat^4$ arises as follows.  Since the multicategory  $\GCatmpshl$ is left-closed, by Lemma \ref{basicenrichment}, it gives rise to a $\GCatmpshl$-enriched category (that is, a left-sharp $4$-category) $\GCat^4$ whose objects are Gray-categories and with homs $\psd_s(\ca,\cb)$: that is, its $1$-cells are strict Gray-functors, $2$-cells are pseudo-transformations, $3$-cells are pseudo-modifications and $4$-cells are perturbations.  

The inclusion $\GCat^4 \hookrightarrow \GCat_p^4$ is a strict $4$-functor (which follows formally using that $\GCatmpshl \hookrightarrow \GCatmp$ is a multifunctor).
\end{eg}

\begin{rmk}
$\GCat^4$ is left-sharp but it is \emph{not} right-sharp!
\end{rmk}

\begin{eg}[\emph{Gray-categories with structure}]
Gray-categories with additional structure can sometimes be assembled into left-sharp $4$-categories, living over $\GCat^4$.  One simple class of examples is as follows.  Since $\Gray$ is a symmetric monoidal closed category, we can consider Gray-weighted limits, which include simple examples like products but also more exotic ones, such as algebra objects for pseudomonads, as considered in \cite{Lack2000ACoherent}.  Given a class $\Phi$ of weights, there is a left-sharp $4$-category $\Phi$-\textbf{Lim} of Gray-categories with $\Phi$-limits, $\Phi$-limit preserving functors, pseudo-transformations, pseudo-modifications and perturbations.   This comes equipped with a strict forgetful $4$-functor $U \colon  \Phi\textnormal{-}\textbf{Lim} \to \GCat^4$.
\end{eg}

\begin{eg}[\emph{Monoidal Gray-categories of endofunctors}]\label{monoidalGray}
For each Gray-category $\ca$, the Gray-category $\psd_s(\ca,\ca)$ of Gray-functors and pseudo-transformations is an example of a monoidal Gray-category.  Indeed, since $\GCat^4$ is a left-sharp (and so moderate) 4-category, restricting it to a moderate $4$-category with just one object $\ca$, produces the example at hand since $\GCat^4(\ca,\ca) := \psd_s(\ca,\ca)$.  

Note that  $\psd_s(\ca,\ca)$ is a \emph{left-sharp} monoidal Gray-category since $\GCat^4$ is a left-sharp $4$-category.  
\end{eg}

\begin{eg}[\emph{The centre of a Gray-monoid}]\label{centre}
Let $(\cc, .,i)$ be a Gray-monoid, viewed as a $1$-object Gray-category $\Sigma \cc$.  By the preceding example, $\psd_s(\Sigma \cc,\Sigma \cc)$ is a monoidal Gray-category, whence $\psd_s(\Sigma \cc,\Sigma \cc)(Id,Id)$ is a braided monoidal $2$-category.  Indeed, the braided monoidal $2$-category $\mathcal{Z}(\cc) = \psd_s(\Sigma \cc,\Sigma \cc)(Id,Id)$ is the \emph{centre} of $\cc$ as described in Theorem 3.1 of \cite{Crans1998Generalized} in elementary terms. This observation was inspired by Remark 3.1.23 of Miranda \cite{MirandaThesis} and again perfects that result --- see \Cref{mirandacomparison} for further details.\black

\end{eg}

Based upon the good homotopical properties of the left-sharp skew monoidal structure $(\GCat,\leftsharp,\mathbf{1})$ established in \Cref{sect:homotopycat}, we expect that the following is true.

\begin{conjecture}
Each weak $4$-category is equivalent to a left-sharp $4$-category.
\end{conjecture}

Some first evidence for the conjecture is given in \Cref{strictification}.  \black A stronger conjecture, also plausible, would be that each weak is equivalent to a left-sharp $4$-category, or even a $4$-tas. \black  Note that even if this were true, left-sharp $4$-categories are still important, since the left-sharp $\GCat^4$ is where the representables of a sharp $4$-category (or $4$-tas) take their values.


\subsubsection{Right-sharp $4$-categories}\label{sect:rightsharp}

We mention these only briefly, since they are dual to the left-sharp case.  A few relevant facts are the following.
\begin{enumerate}
\item A loose $4$-category $\ca$ is left-sharp if and only if its opposite $\ca^{op}$ is right-sharp.  So, for instance, $(\GCat^4)^{op}$ is right-sharp.
\item Consistent with our notation, right-sharp $4$-categories are categories enriched in a \emph{right skew monoidal category}.  Namely, if $(\cc,\otimes,I)$ is a left skew monoidal category, then we can form the \emph{reverse} right skew monoidal category $(\cc,\otimes,i)^{rev}$. This has the same underlying category $\cc$ and unit $i$ but tensor product given by $a \otimes_{rev} b := b \otimes a$ (see Section 2 of \cite{Szlachanyi}).  Now a right left-sharp $4$-category is precisely a category enriched in $(\GCat,\leftsharp,\mathbf{1})^{rev}$, which we call the \emph{right-sharp skew monoidal structure}.
\item Right-sharp $4$-categories are likewise categories enriched in the multicategory $\GCatmpshr$ of \emph{right-sharp multimaps}: by definition a loose multimap $F \colon \ca_1,\ldots,\ca_n \rightsquigarrow \cb$ is right-sharp if $rev_n.F \colon \ca_n,\ldots,\ca_1 \rightsquigarrow \cb$ is left-sharp, where $rev_n \in S_n$ is the permutation which reverse the order of elements. (Since $\GCatmp^l$ is a symmetric multicategory, it follows from the symmetry axioms that right-sharp multimaps are closed under substitution and hence that  $\GCatmpshr \hookrightarrow  \GCatmp^{l}$ is indeed a sub-multicategory.). In elementary terms, a loose binary map $F \colon \ca,\cb \rightsquigarrow \cc$ is right-sharp if it satisfies the dual axioms to those for a left-sharp one: that is, if the invertible $3$-cells $F^A_{B_2,B_1}$ are identities and we have the equality of adjoint equivalences $\aexi^{A}_{B_2B_1} = (\aexi^{A}_{B_2} \circ 1)\cdot (1 \circ \aexi^{A}_{B_1})$.
\end{enumerate}

\subsubsection{Sharp $4$-categories}

Apart from Gray-categories (which can be viewed as sharp $4$-categories whose $4$-cells are identities), none of the significant examples we consider are sharp.  Ultimately, this comes from the fact that our guiding example, $\GCat^4$ is left-sharp but not sharp.  

Nonetheless, we do expect examples to emerge.  In particular, the first-named author is involved in work with Carqueville and Weber to show that $4$-dimensional defect TQFTs give rise to sharp $4$-categories, extending the construction of a Gray-category from a $3$-dimensional defect TQFT described in \cite{Carqueville}.


 Sharp $4$-categories are categories enriched in the multicategory  $\GCatmpsh$ of sharp multimaps, which is the intersection $\GCatmpshl \cap \GCatmpshr$ of the sub-multicategories $\GCatmpshl, \GCatmpshr \hookrightarrow \GCatmp$ of left-sharp and right-sharp multimaps.  The multicategory of sharp multimaps is not closed on either side nor is it symmetric.

\subsubsection{$4$-teisi}

In \cite[Section 2.2]{Crans2000On}, Crans gave an elementary definition of a structure called a $4$-tas (plural $4$-teisi).  Though his efficient but terse $1$-dimensional algebraic notation is quite different at first appearance, $4$-teisi are \emph{exactly} sharp $4$-categories for which the adjoint equivalence $3$-cells $\aexi^A_B$ are isomorphisms. Here we summarise the correspondence.
\begin{itemize}
\item Crans' basic setup involves a $4$-dimensional reflexive globular set, with three kinds of functions --- vertical composition, whiskering and \emph{horizontal composition} --- on cells, each of which is denoted by $- \#_n -$ for some $i$.\begin{footnote}{We note that the whiskering function $\#_n$ for whiskering $i$-cells by $(n+1)$-cells should have $i \leq 4$, not $i < 4$ as claimed in \cite{Crans2000On}.  This is because whiskering of $4$-cells needs to be defined, for instance in order that whiskering by $1$-cells is a well-defined Gray-functor as claimed in Axiom (iii).}\end{footnote}
\item Vertical composition and whiskering have the usual meaning.  The \emph{horizontal composition} operations are the ``dimension raising ones": that is, they specify exactly our cells $A \#_0 B = A_B$, $A \#_0 \beta = A_\beta$, $\alpha \#_0 B = \alpha_B$ and the local interchanger $\alpha \#_1 \beta$ for the hom Gray-categories.
\item To begin, his axioms (vi.a-c) specify that the cells $A_B$, $A_\beta$, and $\alpha_B$ are invertible and have the correct source and target, matching those described in \Cref{unpacked}.
\item Returning to the start, his first axioms (i-iv) express that the structure so defined is a reflexive $4$-d globular set whose homs are Gray-categories and whose whiskering operations are Gray-functors.  
\item Axioms (vii), (viii) and (ix) are precisely the (non-degeneracy) bimap axioms (\Cref{bimapaxioms}) for a sharp $4$-category.  (Which axioms corresponds to which is clear from the labelling.). There are a couple of comments to make here.
\begin{itemize}
\item Firstly, the equalities in Axiom (viii.a-b) express that the $4$-cells $A_{(B_1,B_2)}$ and $(A_1,A_2)_B$ are identities, as required in a sharp $4$-category.  Note also that the two Axioms' \eqref{ax:4-cat_A_1,A_2,A_3:B} and \eqref{ax:4-cat_A:B_1,B_2,B_3} are also redundant for a sharp $4$-category.
\item Crans' Axiom (vii.a) does not parse as written, since, on each side of the equation, it whiskers a $4$-cell by a $3$-cell on the wrong side (using $\#_2$).  Fixing this, the correct form is
$$[(\delta \#_0 \gamma') \#_2(\,(g' \#_0 \Gamma) \#_1 (\delta \#_0 f)\,)] \#_3 (\delta \#_0 \varphi) =  \newline (\delta \#_0 \varphi')\#_3 [(\,(\delta \#_0 f') \#_1(g \#_0 \Gamma)\,) \#_2 (\delta \#_0 \gamma)]$$
 which is our Axiom \eqref{ax:4-cat_A:Theta}.
\end{itemize}
\item Axiom (xi) captures the degeneracy axioms for the composition bimap (\Cref{bimapdegeneracy}).
\item Axiom (v) expresses unitality of the composition bimap (\Cref{bimapunit}).  
\item Axiom (x) captures associativity of the composition bimap (\Cref{bimapassociativity}).
\end{itemize}

Crans conjectured that each weak $4$-category is equivalent to a $4$-tas, but, to the best of our knowledge, no progress on this conjecture has since been made.

\subsubsection{Lax $4$-categories}\label{laxsetting}
Let us briefly mention that in addition to the flexible skew monoidal structure $(\GCat,\otimes_p,1)$, in \cite[Theorem 4.8]{BourkeLobbia} we described a further skew monoidal structure $(\GCat,\otimes_l,1)$ whose internal hom $\lax(\ca,\cb)$ captures \emph{lax transformations}.  Its associated skew multicategory $\GCatm$ was described explicitly in \emph{loc.cit.} up to dimension $4$.  Playing the same game as before leads to lax analogues of the various flavours of semi-strict $4$-categories introduced in \Cref{variations}.  To be precise, we define a \emph{loose lax $4$-category} to be a category $\ca$ enriched in $\GCatm^l$.  It is said to be \emph{moderate} if the lax binary map $\circ_{x,y,z} \colon \ca(y,z), \ca(x,y) \rightsquigarrow \ca(x,z)$ is strict in each variable.  In elementary terms, moderate lax $4$-categories have the same structure as moderate $4$-categories, except we just have a $3$-cell $A_B$ (not an adjoint equivalence $3$-cell) and the $4$-cells $\alpha_B$ and $A_\beta$ need not be invertible.  The axioms are the same except that any equalities of adjoint equivalence $3$-cells are replaced merely by equalities of $3$-cells.  Making the evident further restrictions, we obtain notions of \emph{(sharp/left-sharp/right-sharp)}-lax-4-categories.  

\subsection{Comparison with Miranda's semi-strict $4$-categories}\label{mirandacomparison}

In his thesis \cite{MirandaThesis}, Miranda discussed three enriching frameworks for semi-strict $4$-categories.
\begin{enumerate}
\item a closed, non-monoidal, structure $(\GCat,[-,?],1)$ on $\GCat$ whose hom contains \emph{semi-strictly decomposable} pseudo-transformations (later published in \cite{MirandaSemistrict});
\item a monoidal left closed structure $(\GCat_F,[-,?],1)$ on a different base (Section 3.1.2 of \cite{MirandaThesis});
\item a monoidal (non-closed) structure $(\GCat,\otimes)$ which he called the Crans' monoidal structure (see Section 3.1.1 of \cite{MirandaThesis}) although it is a variant of the monoidal structure considered by Crans.
\end{enumerate}

Categories enriched in $(\GCat,[-,?],1)$ amount to left-sharp $4$-categories with additional non-algebraic conditions about semi-strict decomposability (see Remark 3.1.20 of \emph{loc.cit.}).  

A more restrictive class are categories enriched in $(\GCat_F,[-,?],1)$, which are called \emph{weak interchange $4$-categories}.  Since $\GCat_F$ consists of the $1$-free Gray-categories (see \Cref{sect:freegraycats}) which are coalgebraic structures rather than algebraic ones, categories enriched in $(\GCat_F,[-,?],1)$ involve a coalgebraic aspect.  

These issues prevent either of these kinds of enriched category being monadic over $4$-dimensional globular sets, and in particular they are not definable by a $4$-dimensional globular operad \cite{Batanin1998Monoidal, LeinsterBook}, which we regard as an important requirement for any algebraic notion of $4$-category.  This is by contrast with the various flavours of semi-strict $4$-category introduced here, as can be shown using the theory of presentations for globular operads of Griffiths \cite{Griffiths}.
On the other hand, just as every bicategory is biequivalent not just to any $2$-category but to a cofibrant one (one whose underlying category is free), each weak $4$-category should ultimately be equivalent to one satisfying additional freeness conditions such as those in the definition of a weak interchange $4$-category.

The third structure considered by Miranda is called $(\GCat,\otimes)$.  Since only the tensor product $\otimes$ is described in \cite{MirandaThesis}, a proper comparison is not possible.  The tensor product $\ca \otimes \cb$ almost forms a binary map classifier for the multicategory $\GCatmpsh$ of sharp multimaps.  However, since it is based upon his definition of semi-strict pseudo-transformation (which is weaker than that considered here \Cref{semistrictpseudotransformation}), the resulting enriched categories would necessarily be slightly weaker than sharp $4$-categories since they would not require the two equalities of adjoint equivalences for a sharp 4-category.

As already noted, several of the results in this section refine and were inspired by those of Miranda.  In particular, in \cite[Example 3.1.22]{MirandaThesis} he describes how his doubly-degenerate $(\GCat,[-,?],1)$-enriched categories are braided monoidal $2$-categories satisfying additional freeness and decomposability conditions.  Likewise, in  \cite[Remark 3.1.23]{MirandaThesis} he describes a connection with Crans' centre construction, though it does not satisfy his freeness and decomposability conditions in general.  Working with moderate $4$-categories in  \Cref{braided} and with the left-sharp 4-category $\GCat^4$ in \Cref{centre} enabled us to perfect these results.\black
\black

\subsection{Equivalences in a loose $4$-category}

For $n \in \{2,3,4\}$, an equivalence $n$-cell in a loose $4$-category $\ca$ is, by definition, an equivalence $(n-1)$-cell in the appropriate hom Gray-category $\ca(x,y)$.  An equivalence $1$-cell $f \colon x \to y$ is a $1$-cell for which there exists $g \colon y \to x$ and equivalence $2$-cells $gf \simeq 1_x$ and $gf \simeq 1_y$.  

\begin{lemma}\label{looseequivalence}
A loose $4$-functor $F\colon\ca \rightsquigarrow \cb$ between loose $4$-categories preserves equivalence $n$-cells for each $n$. 
\end{lemma}
\begin{proof}
By definition, the hom-maps $F_{x,y} \colon \ca(x,y) \rightsquigarrow \cb(Fx,Fy)$ of $F$ are pseudomaps of Gray-categories.   Since by Lemma \ref{pseudomapEquivs}, pseudomaps preserve $(1/2/3)$-equivalences, $F$ itself preserves $(2/3/4)$-equivalences.  Since $F$ preserves composition of $1$-cells and identity $1$-cells strictly, and preserves equivalence $2$-cells, it trivially preserves equivalence $1$-cells.
\end{proof}

\begin{prop}
Let $\ca$ be a loose $4$-category.  
\begin{enumerate}
\item Then $\ca$ is a moderate $4$-category if and only the loose $4$-functors $\ca(x,-)\colon \ca \rightsquigarrow \GCat_p^4$ and $\ca(-,x)\colon \ca^{op} \rightsquigarrow \GCat_p^4$ of \Cref{eg:rep} factor through the inclusion $\GCat^4 \hookrightarrow \GCat_p^4$ as loose $4$-functors $\ca(x,-)\colon \ca \rightsquigarrow \GCat^4$ and  $\ca(-,x)\colon \ca^{op} \rightsquigarrow \GCat^4$.  
\item If $\ca$ is left-sharp, then $\ca(x,-)\colon \ca \rightsquigarrow \GCat^4$ is a strict $4$-functor.
\end{enumerate}
\end{prop}
\begin{proof}

For (1), let $f  \colon y \to z \in \ca$.  By definition, $\ca$ is moderate just when the induced pseudomaps $\ca(x,f) := f \circ - \colon \ca(x,y) \rightsquigarrow \ca(x,z)$ and $\ca(-,y):= - \circ f \colon \ca(z,x) \rightsquigarrow \ca(z,x)$ are strict --- that is,  $1$-cells of $\GCat^4$. Since $\GCat^4 \hookrightarrow \GCat_p^4$ is identity on objects and full in dimensions above $1$, this is equally to say that each of $\ca(x,-)$ and $\ca(-,x)$ factor through $\GCat^4$ as claimed.  

For (2),  that  $\ca(x,-)\colon \ca \to \GCat_p^4$ is a strict $4$-functor is just an application of \Cref{basicenrichment} to the left closed multicategory $\GCatmpshl$ since $\GCatmpshl$-enriched functors are strict $4$-functors (again because the unary maps in $\GCatmpshl$ are Gray-functors.)
\end{proof}

\begin{cor}\label{cor:RepEquiv}
For $\ca$ a moderate $4$-category, the loose $4$-functors $\ca(x,-) \colon \ca \rightsquigarrow \GCat^4_s$ and $\ca(-,x) \colon \ca^{op} \rightsquigarrow \GCat^4$ preserve equivalence $n$-cells for each $n \in \{1,\ldots,4\}$.
\end{cor}
\begin{proof}
This follows immediately from the preceding two results.
\end{proof}

In Section \ref{sect:char}, we characterise the equivalence $n$-cells for $n \geq 2$ in $\GCat^4$ and $\GCat_p^4$.  We also characterise the equivalence $1$-cells in $\GCat^4$ between cofibrant Gray-categories as the triequivalences.  This is essential in \Cref{sect:homotopycat} where we show that the left-sharp skew monoidal structure induces a symmetric monoidal closed structure on the homotopy category of Gray-categories.

\section{Homotopical preliminaries}\label{hoprelim}
In this section, we recall Lack's model structures for 2-categories \cite{Lack2002A-quillen} and Gray-categories \cite{Lack2011A-quillen}.  In the latter case, we describe a slight modification of Lack's path object which is closely connected to the internal hom $\psd_s(\ca,\cb)$ for the left-sharp skew monoidal structure.

\subsection{The Quillen model structure for 2-categories}\label{sect:2model}

In \cite{Lack2002A-quillen}, Lack described a Quillen model structure on the category $\twocat$ of $2$-categories and $2$-functors.  The weak equivalences are the biequivalences: that is, those $2$-functors $F \colon \ca \to \cb$ for which each hom functor $F \colon \ca(x,y) \to \cb(Fx,Fy)$ is an equivalence of categories and, moreover, for which $F$ is bi-essentially surjective on objects: this means that given $x \in \cb$, there exists $y \in \ca$ and an equivalence $1$-cell $Fx \simeq y$.  

A $2$-functor is a trivial fibration just when it is surjective on objects, full on $1$ and  fully faithful on $2$-cells.  The fibrations are the so-called \emph{equiv-fibrations}, which are those having the lifting property with respect to the two generating trivial cofibrations. 
To define these, for $n=0,1,2$ we let $D^2_n$ denote the free 2-category containing an n-cell and for $n=1,2$, we let $E^2_n$ denote the free $n$-adjoint equivalence.  For $n=1,2$, we let $j^2_n \colon D^2_n \to E^2_n$ denote the $2$-functor selecting either of the two generating equivalence $n$-cells and let $\rho_n \colon D^2_n \to D^2_{n+1}$ select the source $n$-cell of the free $(n+1)$-cell for $n=0,1$.  The two generating trivial cofibrations are then the composites $l^2_n = j^2_n \circ \rho^2_n \colon D^2_i \to D^2_{i+1} \to E^2_{i+1}$ for $i=0,1$.

Notable facts concerning the model structure are that all objects are fibrant and that the cofibrant objects are those $2$-categories whose underlying category is free on a graph.  Furthermore the Gray-tensor product $(\twocat,\otimes_p,1)$ equips $\twocat$ with the structure of a monoidal model category.

\subsection{The Quillen model structure for Gray-categories}\label{sect:Graymodel}
Building on the above, in \cite{Lack2011A-quillen} Lack described a Quillen model structure on $\GCat$ for Gray-categories.  This time the weak equivalences are the triequivalences: those Gray-functors $F \colon \ca \to \cb$ for which are locally biequivalences and tri-essentially surjective, meaning surjective up to an equivalence $1$-cell (or biequivalence) in $\cb$.  

This time $F \colon \ca \to \cb$ is a trivial fibration if it is surjective on objects, full on $1$- and $2$-cells and fully faithful on $3$-cells. 
As before, the fibrations can be captured by a lifting property against the three generating trivial cofibrations, which is the approach we shall take in practice.  To under the generating trivial cofibrations, for $n=0,1,2,3$, we let $D^3_n$ denote the free Gray-category containing an $n$-cell and for $n=1,2,3$, we let $E^3_n$ denote the free adjoint equivalence $1$-cell.  For $n=1,2,3$, we let $j^3_n \colon D^3_n \to E^3_n$ denote the $Gray$-functor selecting either of the two generating equivalence $n$-cells and
let $\rho_n \colon D^3_n \to D^3_{n+1}$ select the source $n$-cell of the free $(n+1)$-cell for $n=0,1,2$.  The three generating  trivial cofibrations are then the composites $l_n = j_n \circ \rho_n \colon D^3_i \to D^3_{i+1} \to E^3_{i+1}$ for $i=0,1,2$.

All objects are fibrant.   Moreover, the cofibrant objects are characterised in \cite[Corollary 9.4]{Lack2011A-quillen} as those Gray-categories whose underlying sesqui-category is free on a $2$-computad.  We will not need the full strength of this, but only to know that such a Gray-category, by definition, has underlying category free on a graph.

Finally, functorial cofibrant replacements exist (by Quillen's small object argument) since the model structure is combinatorial.  Therefore we fix  $C \colon \GCat \to \GCat$ to be a functorial cofibrant replacement functor and $T_{\ca} \colon C\ca \to \ca$ the components of the associated natural transformation.  

\subsection{The path object for the model structure for Gray-categories}\label{sect:pathobject}

We now describe a path object functor $P$ for the model structure on Gray-categories and prove some useful properties of it.  The path object $P\ca$ modifies the path object $P^L\ca$ of \cite[Proposition~4.1]{Lack2011A-quillen} only by changing the definition of 1-cells: instead of considering an equivalence we take an adjoint equivalence.  Our reason for making this change is that $P\ca$ will then be a full sub-Gray-category of $\psd_s(D^3_1,\ca)$.

Expanding out the definition of $\psd_s(D^3_1,\ca)$, we have that
\begin{itemize}
\item an object of $\psd_s(D^3_1,\ca)$ is a $1$-cell $A\colon a_0\to a_1$;
\item 1-cells $f\colon A\to B$ in $\psd_s(D^3_1,\ca)$, with $A\colon a_0\to a_1$ and $B\colon b_0\to b_1$, consists of a pair of 1-cells in $f_0\colon a_0\to b_0$ and $f_1\colon a_1\to b_1$  together with an equivalence $2$-cell $f$ as below
\[\begin{tikzcd}
	{a_0} & {b_0} \\
	{a_1} & {b_1}
	\arrow[""{name=0, anchor=center, inner sep=0}, "{f_0}", from=1-1, to=1-2]
	\arrow["A"', from=1-1, to=2-1]
	\arrow["B", from=1-2, to=2-2]
	\arrow[""{name=1, anchor=center, inner sep=0}, "{f_1}"', from=2-1, to=2-2]
	\arrow["f", between={0.4}{0.6}, Rightarrow, nfold, from=0, to=1]
\end{tikzcd}\]
equipped with the structure of an adjoint equivalence  $f^o \dashv f$.
\item a 2-cell $\alpha\colon f\to g$ consists of 2-cells $\alpha_0\colon f_0\to g_0$ and $\alpha_1\colon f_1\to g_1$ together with an invertible 3-cell as below
\[\begin{tikzcd}
	{Bf_0} & {Bg_0} \\
	{f_1A} & {g_1A}
	\arrow[""{name=0, anchor=center, inner sep=0}, "{1\circ\alpha_0}", from=1-1, to=1-2]
	\arrow["f"', from=1-1, to=2-1]
	\arrow["g", from=1-2, to=2-2]
	\arrow[""{name=1, anchor=center, inner sep=0}, "{\alpha_1\circ1}"', from=2-1, to=2-2]
	\arrow["\alpha", between={0.3}{0.7}, Rightarrow, nfold, from=0, to=1]
\end{tikzcd}\]

\item 3-cells $\phi\colon\alpha\to\beta$ consist of 3-cells $\phi_0 \colon\alpha_0 \to\beta_0 $ and $\phi_1 \colon\alpha_1 \to\beta_1$ satisfying the standard condition for a perturbation.\end{itemize}

\begin{defn}
$P\ca\hookrightarrow\psd_s(D^3_1,\ca)$ is the full sub-Gray-category whose objects are the equivalences $A \colon a_0 \to a_1$.
\end{defn}\black

We denote by $s,t\colon P\ca\to\ca$ the Gray-functors defined by 
\[\begin{tikzcd}[ampersand replacement=\&]
	P\ca \& {\psd_s(D^3_1,\ca)} \& \ca
	\arrow[hook, from=1-1, to=1-2]
	\arrow["s", shift left, from=1-2, to=1-3]
	\arrow["t"', shift right, from=1-2, to=1-3]
\end{tikzcd}\]
where $s,t\colon\psd_s(D^3_1,\ca)\to\ca$ are the maps picking the source and target. 

\begin{lemma}\label{lem:functorial}
The path object extends to a functor $P \colon \GCat_p \to \GCat_p$ such that the components $s,t \colon P\ca \to \ca$ are natural in $\ca$.  Furthermore, $P$ sends Gray-functors to Gray-functors.
\end{lemma}
\begin{proof}
Observe that $\psd_s(D^3_1,\ca) = \psd_p(D^3_1,\ca)$ since $D^3_1$ contains no non-identity composable $1$-cells.  Certainly we have a functor $\psd_p(D^3_1,-) \colon \GCat_p \to \GCat_p$ as the internal hom of the closed multicategory $\GCatmp$, for which the restriction maps $s,t \colon \psd_p(D^3_1,\ca) \rightrightarrows \ca$ are natural in $F \colon \ca \rightsquigarrow \cb$.  Now since by Lemma \ref{pseudomapEquivs} each pseudomap $F \colon \ca \rightsquigarrow \cb$ preserves $1$-cell equivalences, $\psd_p(D^3_1,F)$ restricts to a pseudomap $PF \colon P\ca \to P\cb$ commuting with the full sub Gray-category inclusions.  It remains to prove that $P$ sends Gray-functors to Gray-functors.  This follows since $\psd_s(D^3_1,-) \colon \GCat \to \GCat$ is a functor --- the internal hom functor for the left-sharp skew monoidal structure.
\end{proof}

\begin{prop}\label{prop:PA-path-obj}
$P\ca$ is a path object for the model structure on $\GCat$.
\end{prop}
\begin{proof}
Lack's path object, which we will call $P^L\ca$, is defined exactly as above with the exception that its $1$-cells involve equivalence $2$-cells rather than chosen adjoint equivalence $2$-cells, as in $PA$.  As such, we have a forgetful Gray-functor $U$ as depicted in the diagram below.  
\[\begin{tikzcd}
	\ca && {P\ca} \\
	&& {P^L\ca} \\
	&& {\ca^2}
	\arrow["W", from=1-1, to=1-3]
	\arrow["{W^L}"{description}, from=1-1, to=2-3]
	\arrow["\Delta"', from=1-1, to=3-3]
	\arrow["U", from=1-3, to=2-3]
	\arrow["{Q^L}", from=2-3, to=3-3]
\end{tikzcd}\]

Here $W^L$ and $Q^L$ are the path object structure for $P^L\ca$.  In particular, by \cite[Proposition~4.1]{Lack2011A-quillen} we know that $W^L$ is a weak equivalence and $Q^L$ a fibration as required of a path object.  We have the evident factorisation of $W^L$ through $P\ca$ as $W$ whilst $Q^L \circ U = (s,t) \colon P\ca \to \ca^2$.

Since fibrations are closed under composition and using the 2-out-of-3 rule, it suffices to prove that $U$ is a trivial fibration. This follows by the fact that any equivalence in a $2$-category can be promoted to an adjoint equivalence.  Hence $U$ is full on 1-cells and clearly it is identity on objects and fully faithful on 2- and 3-cells by construction.
\end{proof}

\begin{rmk}\label{pathprojection}
As for any path object, it follows that $s,t \colon P\ca \rightrightarrows \ca$ are weak equivalences (i.e. triequivalences).
\end{rmk}

\black \begin{prop}\label{prop:p-maps-into-PA}
There is a bijection between pseudo-transformations $\alpha \colon F \to G \in \psd_p(\cx,\ca)$ and pseudomaps $\overline{\alpha} \colon \cx \rightsquigarrow \psd_p(D^3_1,\ca)$ making the diagram below left commute.
\[\begin{tikzcd}
	&& {\psd_p(D^3_1,\ca)} &&& {P\ca} \\
	\cx && \ca & \cx && \ca
	\arrow["t", shift left, from=1-3, to=2-3]
	\arrow["s"', shift right, from=1-3, to=2-3]
	\arrow["t", shift left, from=1-6, to=2-6]
	\arrow["s"', shift right, from=1-6, to=2-6]
	\arrow["{\overline{\alpha}}", shift left, curve={height=-6pt}, squiggly, from=2-1, to=1-3]
	\arrow["F", shift left, squiggly, from=2-1, to=2-3]
	\arrow["G"', shift right, squiggly, from=2-1, to=2-3]
	\arrow["{\overline{\alpha}}", shift left, curve={height=-6pt}, squiggly, from=2-4, to=1-6]
	\arrow["F", shift left, squiggly, from=2-4, to=2-6]
	\arrow["G"', shift right, squiggly, from=2-4, to=2-6]
\end{tikzcd}\]
The pseudo-transformation $\alpha$ has each $1$-cell component $\alpha_x$ an equivalence if and only if the corresponding $\overline{\alpha}$ as above left factors further through $P\ca$ --- that is, giving a bijection between pointwise equivalences $\alpha$ and diagrams as above right.
\end{prop}
\begin{proof}
The pseudo-transformation $\alpha \colon F \to G$ specifies a pseudomap (equally a strict map) $\alpha \colon D^3_1 \to \psd_p(\cx,\ca)$ whose restriction along the two inclusions $D^3_0 \rightrightarrows D^3_1$ gives $F$ and $G$ respectively.  Applying the biclosedness isomorphism $\psd_p(D^3_1,\psd_p(\cx,\ca)) \cong \psd_p(\cx,\psd_p(D^3_1,\ca))$ of \Cref{prop:biclosed}
and using its naturality produces the $\overline{\alpha}$ proving the first claim.

On objects $\overline{\alpha}(x) = \alpha_x \colon Fx \to Gx$ and so $\overline{\alpha}$ factors through $P\ca$ if and only if each $\alpha_x$ is an equivalence, proving the second claim.
\end{proof}

\begin{rmk}\label{explicitbiclosed}
To avoid potential confusion, let us mention that while the pseudomap $\overline{\alpha} \colon \cx \rightsquigarrow \psd_p(D^3_1,\ca)$ has value $\alpha_x$ at $x \in \cx$, at a $1$-cell $f \colon x \to y \in \cx$ its value is not $\alpha_f$ but its left adjoint $\alpha^o_f$ as depicted below
\[\begin{tikzcd}
	Fx & Fy \\
	Gx & Gy
	\arrow[""{name=0, anchor=center, inner sep=0}, "Ff", from=1-1, to=1-2]
	\arrow["{\alpha_x}"', from=1-1, to=2-1]
	\arrow["{\alpha_y}", from=1-2, to=2-2]
	\arrow[""{name=1, anchor=center, inner sep=0}, "Gf"', from=2-1, to=2-2]
	\arrow["{\alpha_f^o}", between={0.3}{0.7}, Rightarrow, nfold, from=0, to=1]
\end{tikzcd}\]
equipped with the reverse adjoint equivalence $\aexi^{\overline{\alpha}}_f = (\aexi^{\alpha}_f)^{\rev}$.  Furthermore, the modification $\overline{\alpha}^2_{f_2,f_1} \colon \overline{\alpha}_{f_2}\overline{\alpha}_{f_1} \to \overline{\alpha}_{f_2f_1}$ has 2-cell components $F^2_{f_2,f_1}$, $G^2_{f_2,f_1}$ and $3$-cell component the mate $(\alpha^o)^2_{f_2,f_1}$ of \eqref{eq:mate}.  
This is essentially because the biclosedness isomorphism $\psd_p(D^3_1,\psd_p(\cx,\ca)) \cong \psd_p(\cx,\psd_p(D^3_1,\ca))$ interchanges the vertical and horizontal directions.  For full details on the biclosedness bijection, generalising the above description, we refer the reader to Remark 5.10 of \cite{BourkeLobbia2}.
\end{rmk}\black

\black \section{Strictifying pseudomaps and pseudo-transformations}\label{sect:replacing}

In the present section, we show that if $\ca$ is a cofibrant Gray-category (or more generally $1$-free), then each pseudomap $F \colon \ca \rightsquigarrow \cb$ can be replaced by an isomorphic strict one, and we explore the functoriality and naturality of this construction.  The main result here in this regard is \Cref{thm:replacement} which we then apply to obtain a similar result for pseudo-transformations in \Cref{thm:replacement2}.  
These strictification results can be viewed as a first towards replacing loose $4$-categories by equivalent left-sharp ones (see \Cref{strictification}).  From the point of view of the present paper, these results will be essential in establishing the characterisation of equivalences in \Cref{sect:char} and later homotopical results.\black

%

A key point in our analysis is the identification of the coalgebras for the pseudomap classifier comonad $Q$, as those Gray-categories which are $1$-free (see Section \ref{sect:freegraycats}).  Before doing so, in Section \ref{sect:freecats} we first describe the simpler situation of free categories on reflexive graphs.

\subsection{Free categories and reflexive graphs}\label{sect:freecats}

Let $\Gph$ denote the category of directed graphs and $\RGph$ the category of reflexive directed graphs.  Our main focus is on the latter.

Consider the adjunction 
\[\begin{tikzcd}
	\RGph && \Cat
	\arrow[""{name=0, anchor=center, inner sep=0}, "F", shift left=2, from=1-1, to=1-3]
	\arrow[""{name=1, anchor=center, inner sep=0}, "U", shift left=2, from=1-3, to=1-1]
	\arrow["\dashv"{anchor=center, rotate=-90}, draw=none, from=0, to=1]
\end{tikzcd}\]
between small categories and reflexive graphs.  The left adjoint sends a reflexive graph $\cx$ to the category $F\cx$ with the same objects as $\cx$ and whose morphisms are either identities $1_x \colon x \to x$ or composable sequences $\overline{f} = [f_1,\ldots,f_n]:x \to y$ of non-identity morphisms in $\cx$.  

Categories of the form $F\cx$ admit a standard intrinsic characterisation as \emph{free categories}, in the sense defined below --- see, for instance \cite{RiehlVerity, Campbell}.

\begin{defn}
Let $\ca$ be a category.  
\begin{itemize}
\item An arrow $f \colon x \to y$ is \emph{atomic} if it is not an identity and admits no non-trivial factorisations: that is, if $f = gh$, then $g$ or $h$ is an identity.
\item $\ca$ is \emph{free} if each non-identity arrow can be written uniquely as a composite of atomics and no identity admits a non-trivial factorisation.
\item A functor $F \colon \ca \to \cb$ between free categories is \emph{free} if it sends atomics to atomics or identities.  
\end{itemize}
\end{defn}

We write $\Free \hookrightarrow \Cat$ for the subcategory of free categories and free functors between them.

\begin{rmk}
Note that the property of being free is really a property rather than structure.  Indeed, if $U$ is a \emph{basis} of atomics for a free category (in the obvious sense) then the factorisation property into $U$'s ensures that $U$ consists of all the atomics. 
\end{rmk}

\begin{lemma}
The left adjoint $F \colon \RGph \to \Cat$ factors through $\Free$ as an equivalence of categories $\RGph \simeq \Free$.
\end{lemma}
\begin{proof}
$F \colon \RGph \to \Cat$ takes its image in the subcategory $\Free$ of $\Cat$ by construction.  The induced factorisation $\RGph \to \Free$ is (1) \emph{essentially surjective on objects} since each free category is free on its reflexive graph of atomics/identities, (2) \emph{full} since  free morphisms are precisely those arising from the corresponding reflexive graph maps and (3) \emph{faithful} since $F$ is, by a trivial calculation.
\end{proof}

We now turn to the induced comonad $FU$ on $\Cat$.  The following fact is discussed in \cite{Dawson,Campbell} but since the details are left to the reader, we describe the key aspects of the proof.

\begin{lemma}\label{comonadic}
The adjunction $F \dashv U$ is comonadic.
\end{lemma}
\begin{proof}
By Beck's comonadicity theorem, it suffices to show that $F$ reflects isomorphisms and preserves coreflexive equalisers, and the first point is easily checked.  For the latter, consider the commutative diagram of left adjoints

\[\begin{tikzcd}
	\Gph && \RGph \\
	&& \Cat
	\arrow["{F^r}", from=1-1, to=1-3]
	\arrow["{F^g}"', from=1-1, to=2-3]
	\arrow["F", from=1-3, to=2-3]
\end{tikzcd}\]

where $F^g$ is the free category on a graph and $F^r$ is the free reflexive graph construction, which has $F^r(X)_0 = X_0$ and $F^r(X)_1 = X_0 + X_1$.  Since coproducts commute with equalisers in $\Set$, therefore $F^r$ preserves them.

Now $F^r$ is essentially surjective on objects since each reflexive graph is free on the graph obtained by discarding its identities.  Therefore, it suffices to consider a coreflexive pair of the form $l,r \colon F^r X \rightrightarrows F^r Y$.  By coreflexivity, both $l$ and $r$ are split mono.  Although $F^r$ is not full, it is full on monos since they not only preserve, but also reflect, identities.  There $l = F^r l'$ and $r = F^r r'$.  We now form the equaliser $E \to X$ of $l',r' \colon X \rightrightarrows Y$ in $\Gph$ which is preserved by $F^r$.  Therefore, $F$ will preserve the equaliser $F^r E$ just when $F^g$ preserves the coreflexive equaliser in $\Gph$.  But $F^g$ preserves all equalisers (or connected limits more generally) since it is the free functor to the category of algebras for a familially representable monad --- see Example C.3.3 of \cite{LeinsterBook}.

\end{proof}

Combining the above results we obtain:

\begin{prop}\label{free}
There is an isomorphism of categories $\Free \cong \Coalg{FU}$ commuting with the forgetful functors to $\Cat$.
\end{prop}
\begin{proof}

Consider the following commutative diagram.  

\[\begin{tikzcd}
	\RGph && \Free && {\Coalg{FU}} \\
	&& \Cat
	\arrow["\simeq", from=1-1, to=1-3]
	\arrow["\simeq", curve={height=-24pt}, from=1-1, to=1-5]
	\arrow["F"', from=1-1, to=2-3]
	\arrow["K", from=1-3, to=1-5]
	\arrow[hook, from=1-3, to=2-3]
	\arrow["{U^{FU}}", from=1-5, to=2-3]
\end{tikzcd}\]

First note that $\Free \hookrightarrow \Cat$ is a left adjoint, the composite of $\Free \simeq \RGph$ and $F \colon \RGph \to \Cat$, and the induced comonad is again $FU$.  Hence, we obtain the canonical comparison $K$ in the right horizontal.  Explicitly, given a free category $\ca$, the induced functor for the coalgebra structure $\ca \to FU\ca$ acts by decomposing non-identities into their atomic components.  

Likewise the upper functor $\RGph \to \Coalg{FU}$ is the canonical comparison, which is an equivalence by Lemma \ref{comonadic}.  By $2$-out-of-$3$ for equivalences, $K$ is therefore an equivalence.  

In fact it is an isomorphism.  Indeed, like any forgetful functor from a category of coalgebras, $U^{FU}$ is a discrete isofibration --- that is, admits unique transport of structure along isomorphisms.  Furthermore, so does the inclusion $\Free \hookrightarrow \Cat$ since isomorphisms preserve atomic arrows.  Therefore, by $2$-out-of-$3$ for discrete isofibrations, $K$ is one too.  Since any discrete isofibration which is also an equivalence is an isomorphism of categories, this proves the claim.
\end{proof}

\subsection{The pseudomap classifier comonad $Q$ and $1$-free Gray-categories}\label{sect:freegraycats}
The main comonad $Q$ of interest to us resides on $\GCat$ and is obtained by lifting the comonad $FU$ on $\Cat$ in a simple way.  We now recall its construction, which first appeared in \cite{GohlaB:mapp}.

The starting observation is that the forgetful functor $(-)_1 \colon \GCat \to \Cat$ sending a Gray-category to its underlying category is a Grothendieck fibration, whose cartesian morphisms are those Gray-functors which are fully faithful on $2$-cells and $3$-cells.  Concretely,  the cartesian lifting $\cb^\star_F$ of $F \colon \cb \to \ca_1$ has the same underlying category as $\cb$ whilst $\cb^\star_F(x,y)(f,g) := \ca(Fx,Fy)(Ff,Fg)$.  Furthermore, the lifted Gray-functor $\cb^\star_F \to \ca$ then acts as $F$ on underlying categories and is the identity on $2$ and $3$-cells.

By \cite[Theorem 2.14]{GohlaB:mapp}, any comonad can be lifted along a Grothendieck fibration by taking the cartesian liftings of the counit and comultiplication.  In our setting, this produces the comonad $Q$ on $\GCat$, as below, and simply means that the counit and multiplication for $Q$ are the same as for $FU$ and have trivial action on $2$-cells and $3$-cells.

\[\begin{tikzcd}[ampersand replacement=\&]
	{QQ\ca} \&\& {Q\ca} \&\& \ca \\
	{FUFU\ca_1} \&\& {FU\ca_1} \&\& {\ca_1}
	\arrow["{\Delta_\ca}"', from=1-3, to=1-1]
	\arrow["{p_{\ca}}", from=1-3, to=1-5]
	\arrow["{(-)_1}"', between={0.1}{0.9}, maps to, from=1-3, to=2-3]
	\arrow["{F\eta_{U\ca_1}}", from=2-3, to=2-1]
	\arrow["{\epsilon_{\ca_1}}"', from=2-3, to=2-5]
\end{tikzcd}\]

In particular, the counit $p_{\ca} \colon Q\ca \to \ca$ is the identity on objects, sends a composable sequence $\overline{f} = [f_1,\ldots,f_n]:x \to y$ to its composite $f_n \circ \ldots \circ f_1$ and is the identity on $2$-cells and $3$-cells.  

\begin{rmk}
The reason that $Q$ is important for us is that, as established in \cite[Lemma 2.25]{GohlaB:mapp}, the category $\GCat_p$ is the Kleisli category for $Q$, with $J \colon\GCat \to \GCat_p$ the canonical inclusion to the Kleisli category.  
Accordingly, when referring to $Q$, we sometimes call it the \emph{pseudomap classifier}.  

Its universal property from the Kleisli adjunction is that there is a universal pseudomap $q_{\ca} \colon \ca \rightsquigarrow Q\ca$ through which each pseudomap $F \colon \ca \rightsquigarrow \cb$ factors uniquely, as depicted below.
\[\begin{tikzcd}
	{Q\ca} && \\
	\ca && \cb
	\arrow["{F'}", from=1-1, to=2-3]
	\arrow["{q_{\ca}}", squiggly, from=2-1, to=1-1]
	\arrow["F"', squiggly, from=2-1, to=2-3]
\end{tikzcd}\]

Though not required, we note that $q_{\ca} \colon \ca \rightsquigarrow Q\ca$ is described as follows: it is the identity on objects, sends identity $1$-cells to identity $1$-cells, and sends non-identities $f \colon x \to y$ to words $[f] \colon x \to y$.  It is also the identity on $2$-cells and $3$-cells.  The cocycle $q^2_{g,f}$ is given by $1_{gf} \colon [g] \circ [f] \to [gf]$.  The construction of $F'$ is described in \cite[Lemma 2.25]{GohlaB:mapp}.\end{rmk}

In Proposition~\ref{prop:coalgebra}, we identify the $Q$-coalgebras as the $1$-free Gray-categories, which we define now.

\begin{defn}
The category $\Free_1$ of $1$-free Gray-categories and $1$-free functors is defined as the pullback below.
\[\begin{tikzcd}
	{\Free_1} && \Free \\
	\GCat && \Cat
	\arrow[from=1-1, to=1-3]
	\arrow[hook, from=1-1, to=2-1]
	\arrow[hook, from=1-3, to=2-3]
	\arrow["{(-)_1}"', from=2-1, to=2-3]
\end{tikzcd}\]
That is,  $\ca$ is $1$-free if its underlying category $\ca_1$ is free whilst $F\colon \ca \to \cb$ is $1$-free if its underlying functor is free.  
\end{defn}

\begin{egs}
\begin{enumerate}
\item If the underlying category of $\ca$ is free on a directed graph, then $\ca$ is $1$-free.  This is because each category which is free on a graph is free on a reflexive graph.
\item Each cofibrant Gray-category is $1$-free.  Indeed, by \cite[Corollary 9.4]{Lack2011A-quillen}, the cofibrant Gray-categories are those whose underlying sesqui-category is free on a $2$-computad, which in particular mean that their underlying category is free on a graph.  Therefore this follows from the previous point.
\item Each of the generating cofibrations and trivial cofibrations in $\GCat$ is a $1$-free morphism between $1$-free Gray-categories.  Indeed, all $1$-cells in the domain or codomain of the generating cofibrations, described in \Cref{sect:Graymodel}, are identities or atomic, so being a $1$-free morphism follows automatically.
The generating trivial cofibrations are the maps $l_n \colon D^3_i \to E^3_{i+1}$ for $i=0,1,2$ of \Cref{sect:Graymodel}.  The underlying functors of $l_1$ and $l_2$ are both the identity on the generic parallel pair of $1$-cells $\{0 \rightrightarrows 1\}$ and so $1$-free.  The underlying functor of $l_0 \colon D^3_0 \to E^3_{1}$ simply selects the object $0$ of the free category on the graph $0 \leftrightarrows 1$ and so is also $1$-free.

\end{enumerate}
\end{egs}

\begin{prop}\label{prop:coalgebra}
There is an isomorphism of categories $\Free \cong \Coalg{Q}$ commuting with the forgetful functors to $\GCat$. 
\end{prop}

By definition of $\Free_1$ and Proposition \ref{free} the left rectangle below is the composite of two pullback squares and so a pullback itself.
\begin{proof}
\[\begin{tikzcd}
	{\Free_1} & \Free & {\Coalg{FU}} && {\Coalg{Q}} & {\Coalg{FU}} \\
	\GCat & \Cat & \Cat && \GCat & \Cat
	\arrow[from=1-1, to=1-2]
	\arrow[hook, from=1-1, to=2-1]
	\arrow["\cong", from=1-2, to=1-3]
	\arrow[hook, from=1-2, to=2-2]
	\arrow["{U^{FU}}", from=1-3, to=2-3]
	\arrow[from=1-5, to=1-6]
	\arrow["{U^Q}"', from=1-5, to=2-5]
	\arrow["{U^{FU}}", from=1-6, to=2-6]
	\arrow["{(-)_1}"', from=2-1, to=2-2]
	\arrow[equals, from=2-2, to=2-3]
	\arrow["{(-)_1}"', from=2-5, to=2-6]
\end{tikzcd}\]
Therefore it suffices to prove that the square above right is a pullback.  This is a routine formal consequence of the construction of $Q$ via lifting $FU$ through the fibration $(-)_1$.  To see this, let us consider what it means to equip $\ca$ with the structure of a $Q$-coalgebra.  Since $p_{\ca}$ is a cartesian lifting of $\epsilon_{\ca_1}$ along $(-)_1$, the diagram below 
\[\begin{tikzcd}
	& {Q\ca} &&&& {FU\ca_1} \\
	\ca & \ca &&& {\ca_1} & {\ca_1}
	\arrow[""{name=0, anchor=center, inner sep=0}, "{p_{\ca}}", from=1-2, to=2-2]
	\arrow["{\epsilon_{\ca_1}}", from=1-6, to=2-6]
	\arrow["a", dashed, from=2-1, to=1-2]
	\arrow["1"', from=2-1, to=2-2]
	\arrow[""{name=1, anchor=center, inner sep=0}, "{a_1}", dashed, from=2-5, to=1-6]
	\arrow["1"', from=2-5, to=2-6]
	\arrow["{(-)_1}", between={0.3}{0.7}, maps to, from=0, to=1]
\end{tikzcd}\]
exhibits a bijection between sections of $p_{\ca}$ and sections $a_1$ of $\epsilon_{\ca_1}$.  That $a$ is a $Q$-coalgebra if and only if $a_1$ is an $FU$-coalgebra follows easily using that $p_{\ca} \circ p_{Q\ca}$ is cartesian and that $\Delta_\ca$ is a lifting of $\epsilon_{\ca_1}$.  In particular, we obtain a bijective correspondence between $Q$-coalgebra structures on $\ca$ and $FU$-coalgebra structures on $\ca_1$.  A corresponding statement holds for morphisms of coalgebras.

\end{proof}

\subsection{2-trivial fibrations and icons}

We now turn to an important class of morphism called (normal) $2$-trivial fibrations.  These are important because each such Gray-functor admits a pseudomap section, which will prove crucial in Section \ref{subsect:replacing}.  

\begin{defn}
A Gray-functor $F \colon \ca \to \cb$ is called a $2$-trivial fibration if it is surjective on objects, full on $1$-cells and fully faithful both on $2$-cells and $3$-cells. (In other words, a trivial fibration which is also faithful on $2$-cells.)

A normal cleavage for a $2$-trivial fibration is a \emph{lifting function} $\phi$ giving:
\begin{itemize}
\item for each $x \in \cb$, an object $\phi x \in \ca$ such that $G( \phi x) = x$;
\item for each $x,y \in \ca$ and $1$-cell $f \colon Gx \to Gy \in \cb$, a $1$-cell $\phi(x,f,y) \colon x \to y$ such that $G \phi(x,f,y) = f$ and satisfying the normality condition $\phi(x,1_{Gx},x) = 1_x$.
\end{itemize}

A normal $2$-trivial fibration  $(F,\phi) \colon \ca \to \cb$ is a $2$-trivial fibration equipped with a normal cleavage.  We sometimes abuse notation by writing $\phi(f) \colon x \to y$ instead of $\phi(x,f,y)$ when the context is clear.  

\end{defn}

\begin{rmk}{We mention that normal $2$-trivial fibrations are the right class of an algebraic weak factorisation system (awfs) on $\GCat$ whose cofibrant replacement comonad is $Q$.  This perspective is further explored in the proof of \cite[Proposition 3.6]{BourkeLobbia}.  Note that in \cite{BourkeLobbia}, normal $2$-trivial fibrations were called algebraic surjections. }\end{rmk}

\begin{eg}
Each $2$-trivial fibration admits a normal cleavage, though it is not canonical.
\end{eg}

\begin{eg}\label{counitnormal}
The counit $p_{\ca} \colon Q\ca \to \ca$ is a normal $2$-trivial fibration when equipped with the lifting function $\phi x = x$ on objects and, on maps $f$, $\phi(x,f,y) = [f]$ if $f$ is non-identity and $\phi(x,f,y) = 1_x$ if it is the identity.  
\end{eg}


\begin{defn}
Let $\eta \colon F \to G \in \psd_p(\ca,\cb)$ be a pseudo-transformation. It is said to be an icon if the 1-cell components $\eta_a \colon Fa \to Ga$ are identities.
\end{defn}

\begin{lemma}\label{section}
Let $F \colon \ca \to \cb$ be a $2$-trivial fibration.
\begin{enumerate}[leftmargin=*, label=(\arabic*)]
\item Then $F$ admits a section in $\GCat_p$.  More precisely, each normal cleavage $\phi$ on $F$ induces a canonical pseudomap $G  \colon \cb \rightsquigarrow \ca$ in $\GCat_p$ such that $F G  = 1$.
\item Suppose that $F$ is furthermore the identity on objects.  Given pseudomaps $G,H \colon \cc \rightsquigarrow \ca$ such that $F G = F H$, there exists a unique invertible icon $\lambda \colon G \to H$ such that $F \lambda = id$.  
\end{enumerate}
\end{lemma}
\begin{proof}
\begin{enumerate}[leftmargin=*, label=(\arabic*)]
\item 
Since each $2$-trivial fibration admits a normal cleavage, it suffices to prove the second statement.  This is the content of \cite[Lemma 3.4]{BourkeLobbia} but we include the proof for the reader's convenience.  

On objects, we set $G (x) = \phi (x)$ and, at $f \colon x \to y$, we define $G  f = \phi(f) \colon G x \to G y$ so that $F(G x) = x$ and $F(G f) = f$.  Since $F$ is fully faithful on $2$-cells and $3$-cells, there exists a unique way in which to complete $G $ to a morphism of $3$-graphs such that $FG  = 1$.  The remaining piece of data for a pseudomap is given by the invertible $2$-cell depicted below left

\[\begin{tikzcd}[ampersand replacement=\&]
	{G x} \&\& Gy \&\& x \&\& y \\
	\\
	\&\& Gz \&\&\&\& z
	\arrow["Gf", from=1-1, to=1-3]
	\arrow[""{name=0, anchor=center, inner sep=0}, "{G(gf)}"', from=1-1, to=3-3]
	\arrow[""{name=1, anchor=center, inner sep=0}, "Gg", from=1-3, to=3-3]
	\arrow["f", from=1-5, to=1-7]
	\arrow[""{name=2, anchor=center, inner sep=0}, "gf"', from=1-5, to=3-7]
	\arrow["g", from=1-7, to=3-7]
	\arrow["{G^2_{g,f}}"'{pos=0.6}, between={0}{0.7}, Rightarrow, nfold, from=1-3, to=0]
	\arrow[between={0.3}{0.6}, maps to, from=1, to=2]
	\arrow[between={0}{0.7}, equals, nfold, from=1-7, to=2]
\end{tikzcd}\]

Using that $F$ is fully faithful on $2$-cells, this is defined to be the unique $2$-cell (necessarily invertible) sent by $F$ to the identity on $g \circ f$.  

With regards the equations for $G$ to be a pseudomap, note that by the normality condition for $\phi$, $G $ preserves identity $1$-cells.  Now all of the remaining equations concern equalities between parallel $2$-cells or $3$-cells in $\ca$.  Since $F$ is faithful on $2$-cells and $3$-cells, it suffices to show that each of these equations hold after application of $F$, whereupon they are sent to the corresponding equation for $F G $ to be a pseudomap.  And these hold since $FG $ is the identity on $\cb$. 
\item
The argument is similar to the previous part.  Since $F$ is the identity-on-objects and $FG = FH$, we must have $Ga=Ha$ for all $x \in \ca$, and so we set the components $\lambda_x$ to be the identities, as required.  

To extend $\lambda$ to an icon, it remains to give adjoint equivalence $2$-cells $\lambda_f \colon Gf \cong Hf$ plus invertible $3$-cells of the form $\lambda^2_{g,f}$ and $\lambda_{\phi}$ satisfying certain equalities of $2$-cells and $3$-cells.  These are uniquely determined by the fact that $F$ is fully faithful on $2$-cells and $3$-cells and that $F G = F  H$.  Indeed, given $f \colon x \to y$, the parallel $1$-cells $Gf$ and $Hf$ satisfy $FGf = FHf$.  Then as $F$ is a 2-trivial fibration, there exists a unique invertible $2$-cell $\lambda_f \colon Gf \cong Hf$ such that $F\lambda_f$ is the identity: this isomorphism the adjoint equivalence 2-cell for the pseudo-transformation $\lambda$.  The construction of $\lambda^2_{g,f}$ and $\lambda_{\phi}$ is similar.  That the equations for a pseudo-transformation are satisfied follows from the fact that  $F$ is faithful on $2$-cells and $3$-cells and that the identity $F G=F H$ is a pseudo-transformation.

For invertibility of $\lambda$, observe that a symmetric argument products a unique icon $\theta \colon H \to G$ such that $F\theta = id$.  We claim that $\theta \circ \lambda = id_G$ in $\psd_p(\cc,\ca)$.  Certainly, both sides have the same (identity) $1$-cell components, so it remains to check they have the same $2$-cell and $3$-cell components.  Since $F$ is faithful on $2$-cells and $3$-cells, this is equally to show that $F(\theta \circ \lambda) = F(id_G)$.  But as $F$ is strict, so is $\psd_p(\cc,F)$, so $F(\theta \circ \lambda) = F\theta \circ F\lambda = id_{FG} \circ id_{FG} =  id_{FG} = F(id_G)$ as required.  Similarly $\lambda \circ \theta = id_F$.
\end{enumerate}
\end{proof}

%
%

\subsection{Replacing pseudomaps by isomorphic strict maps}\label{subsect:replacing} 
The following is our main result about replacing pseudomaps by strict ones.

\begin{theorem}\label{thm:replacement}
Let $\ca$ be $1$-free. 
\begin{enumerate}[leftmargin=*, label=(\arabic*)]
\item Given a pseudomap $F \colon \ca \rightsquigarrow \cb$ there is a strict map $K(F) \colon \ca \to \cb$ and invertible icon $\rho_F \colon K(F) \cong F$.  
\item If $F$ is strict, then $K(F)=F$ and $\rho_F=id$.
\item The assignment is natural in strict maps $\cb \to \cc$ and in $1$-free morphisms $\cd \to \ca$.
\end{enumerate}
\end{theorem}
\begin{proof}
\begin{enumerate}[leftmargin=*, label=(\arabic*)]
\item Since $\ca$ is $1$-free, by Proposition \ref{prop:coalgebra} it has a coalgebra structure $a \colon \ca \to Q\ca$.  Now the pseudomap $F\colon \ca \rightsquigarrow \cb$ factors as $F' \circ q_{\ca} \colon \ca \rightsquigarrow Q\ca \to \cb$ for the unique strict map $F^{\prime}$ and we set $K(F) = F^{\prime} \circ a$.  

\[\begin{tikzcd}
	{Q\ca} &&&& \\
	\ca && \cb && {}
	\arrow["{F'}", curve={height=-6pt}, from=1-1, to=2-3]
	\arrow[""{name=0, anchor=center, inner sep=0}, "a", shift left=4, from=2-1, to=1-1]
	\arrow[""{name=1, anchor=center, inner sep=0}, "{q_{\ca}}"', shift right=5, squiggly, from=2-1, to=1-1]
	\arrow["F"', squiggly, from=2-1, to=2-3]
	\arrow["{\lambda_\ca}", shift right, between={0.2}{0.8}, Rightarrow, from=0, to=1]
\end{tikzcd}\]

Now $p_{\ca} \circ q_{\ca} = 1_{\ca} = p_{\ca} \circ a$.  Since $p_{\ca}$ is the identity on objects, by Lemma \ref{section}, there is a unique invertible icon $\lambda_{\ca} \colon a \cong q_\ca$ for which $p_{\ca}\circ  \lambda_{\ca}$ is the identity.  Then $\rho_F := F^{\prime} \circ \lambda_{\ca}$ is the invertible icon.

\item If $F$ is strict, then $F' = F \circ p_{\ca}$ since both are strict maps which yield $F$ when precomposed by $q_{\ca}$.  Therefore  $K(F) = F \circ p_{\ca} \circ a = F$ and $\rho_F = F \circ p_{\ca} \circ \lambda_{\ca} = id_F$.

\item The naturality situation is encapsulated in the following diagram.

\[\begin{tikzcd}
	{Q\cd} && {Q\ca} &&&& \\
	\cd && \ca && \cb && \cc
	\arrow["QG", from=1-1, to=1-3]
	\arrow["{F'}", curve={height=-6pt}, from=1-3, to=2-5]
	\arrow[""{name=0, anchor=center, inner sep=0}, "d", shift left=4, from=2-1, to=1-1]
	\arrow[""{name=1, anchor=center, inner sep=0}, "{q_{\cd}}"', shift right=4, squiggly, from=2-1, to=1-1]
	\arrow["G"', from=2-1, to=2-3]
	\arrow[""{name=2, anchor=center, inner sep=0}, "a", shift left=4, from=2-3, to=1-3]
	\arrow[""{name=3, anchor=center, inner sep=0}, "{q_{\ca}}"', shift right=5, squiggly, from=2-3, to=1-3]
	\arrow["F"', squiggly, from=2-3, to=2-5]
	\arrow["H"', from=2-5, to=2-7]
	\arrow["{\lambda_\cd}", shift right, between={0.2}{0.8}, Rightarrow, from=0, to=1]
	\arrow["{\lambda_\ca}", shift right, between={0.2}{0.8}, Rightarrow, from=2, to=3]
\end{tikzcd}\]

Let us first consider $H \colon \cb \to \cc$ strict.  By the universal property of $q_\ca$, we have that $(H \circ F)' = H \circ F'$.  The claim follows immediately. 

Secondly consider $G \colon \cd \to \ca$ a $1$-free morphism.  Let us show that the left square of the above diagram serially commutes, which is precisely what we need.

By \Cref{prop:coalgebra}, $G$ is a morphism of the associated coalgebras, so that $QG \circ d = a \circ G$ whilst by naturality of the unit components $q_{\ca}$, we also have that $q_{\ca} \circ G = QG \circ q_{\cd}$.  Next we show that $\lambda_\ca \circ G = QG \circ \lambda_{\cd}$.  By Lemma \ref{section} Part 2, it suffices to show that both yield the identity when postcomposed by $p_{\ca}$.  The lhs does so by definition of $\lambda_\ca$ whilst $p_{\ca} \circ QG \circ \lambda_{\cd} = G \circ p_{\cd} \circ \lambda_{\cd}$ which equals the identity by definition of $\lambda_\cd$.  Thus the left part of the above diagram serially commutes, as claimed.

Then to obtain naturality in $G$, now observe that by the universal property of $q_{\cd}\colon  \cd \rightsquigarrow Q\cd$, the above commutativity implies that $(F \circ G)'= F' \circ QG$, but then $K(F \circ G) := (F \circ G)' \circ d = F' \circ QG \circ d = F' \circ a \circ G = K(F) \circ G$.  Likewise $\rho_{F \circ G} = \rho_F \circ G$ by an analogous argument.
\end{enumerate}
\end{proof}

\begin{prop}\label{prop:adjointEquiv}
Let $\ca$ be $1$-free.  Then the full inclusion $J \colon \psd_s(\ca,\cb) \to \psd_p(\ca,\cb)$ forms part of a Gray-enriched adjoint equivalence $(\rho, J \dashv K,1)$ with counit $\rho$ and identity unit.  The adjoint equivalences are natural in strict maps $\cb \to \cc$ and in $1$-free morphisms $\cd \to \ca$.
\end{prop}
\begin{proof} $J$ is essentially surjective on objects by \Cref{thm:replacement}.  Since it is enriched fully faithful, it is an enriched equivalence.  The structure of the adjoint equivalence is as follows.  The preceding proposition provides a section $K$ on objects to the inclusion $J$ and isomorphisms $\rho_F \colon J K (F) \cong F$.  By a standard enriched categorical argument since $J$ is fully faithful, there then exists a unique extension of $K$ to a Gray-functor for which the $\rho_F$ form the components of a Gray-enriched natural isomorphism.  By \Cref{thm:replacement} we have $\rho_J = id$ whilst $K \rho = id$ follows from $K$'s construction as a Gray-functor, so giving the counit and identity unit of an adjoint equivalence.  The claimed naturality property amount to precisely those stated in the third part of \Cref{thm:replacement}.
 \end{proof}
 
 
 
 \black
 
 \subsection{Replacing pseudo-transformations by isomorphic semi-strict ones}\label{subsect:replacing2} 
 
 We now use the same techniques as above to replace pseudo-transformations by isomorphic semi-strict ones.  The key to achieving this will be the following proposition which allows us see semi-strict pseudo-transformations as strict Gray-functors.

  \begin{prop}\label{semistrictasstrict}
 Let $F,G \colon \ca \rightrightarrows \cb$ be strict maps.  A pseudo-transformation  $\alpha \colon F \to G$ is semi-strict just when the associated pseudomap $\overline{\alpha} \colon \ca \rightsquigarrow \psd_s(D^3_1,\cb)$ is a strict Gray-functor.  
 \end{prop}
 \begin{proof}
Like any pseudomap, $\overline{\alpha} \colon \ca \rightsquigarrow \psd_s(D^3_1,\cb)$ will be strict just when the cocycle $2$-cells $\overline{\alpha}^2_{f_2,f_1} \colon \overline{\alpha}_{f_2}\overline{\alpha}_{f_1} \to \overline{\alpha}_{f_2f_1} \in \psd_s(D^3_1,\cb)$ are identities.   And as for any $2$-cell in $\psd_s(D^3_1,\cb)$, this will be the case just when (1) the components of $\overline{\alpha}^2_{f_2,f_1}$ are identities \emph{and moreover} (2) its source $ \overline{\alpha}_{f_2}\overline{\alpha}_{f_1}$ and target $ \overline{\alpha}_{f_2f_1}$ have the same associated adjoint equivalences.
 Now following \Cref{explicitbiclosed} the components of ${\overline{\alpha}}^2_{f_2,f_1}$ are the $2$-cells $F^2_{f_2,f_1}$, $G^2_{f_2,f_1}$ and $3$-cell $(\alpha^\circ)^2_{f_2,f_1}$.  Since $F$ and $G$ are assumed to be strict, the first two are identities so that (1) will hold just when the $3$-cell $(\alpha^o)^2_{f_2,f_1}$ is an identity. 

 Now given (1), by the formulae in \Cref{explicitbiclosed} the adjoint equivalences in (2)  are $(1 \circ (\aexi^{\alpha}_{f_1})^{\rev}) \cdot ((\aexi^{\alpha}_{f_2})^{\rev} \circ 1)$ and $(\aexi^{\alpha}_{f_2f_1})^{\rev}$ respectively.  Since taking reverses is contravariant functorial and preserved by $1$-cell whiskering (or indeed by any $2$-functor), the former is $((\aexi^{\alpha}_{f_2} \circ 1) \cdot(1 \circ (\aexi^{\alpha}_{f_1}) )^{\rev}$.  Now since $(-)^{\rev}$ is involutive, (2) will therefore hold just when $\aexi^{\alpha}_{f_2f_1} = (\aexi^{\alpha}_{f_2} \circ 1) \cdot(1 \circ (\aexi^{\alpha}_{f_1})$.  Therefore by \Cref{semistrictequivalence}, (1) and (2) hold iff and only if $\alpha$ is semi-strict.
 \end{proof}%

\begin{rmk}
The main issue with semi-strict pseudo-transformations is that they are not closed under composition in $\psd_p(\ca,\cb)$ in general, and so not the $1$-cells of a Gray-category.  However, given strict maps $F$ and $G$, we can form the form the full sub-$2$-category $\psd_s^{ss}(\ca,\cb)(F,G)$ of $\psd_s(\ca,\cb)(F,G)$ with objects the semi-strict pseudo-transformations from $F$ to $G$.
\end{rmk}

\black
\begin{theorem}\label{thm:replacement2}
If $\ca$ is $1$-free, then the inclusion $\psd_s^{ss}(\ca,\cb)(F,G) \hookrightarrow \psd_s(\ca,\cb)(F,G)$ is a $2$-equivalence of $2$-categories.
\end{theorem}
\begin{proof}
 Since it is a full inclusion, it remains to check essential surjectivity on objects.  By \Cref{prop:p-maps-into-PA}, a  pseudo-transformation $\alpha \colon F \to G$ bijectively corresponds to a pseudomap $\overline{\alpha}$ below left
\[\begin{tikzcd}[ampersand replacement=\&]
	\&\& {\psd_s(D^3_1,\cb)} \&\&\&\& {\psd_s(D^3_1,\cb)} \\
	\ca \&\& \cb \&\& \ca \&\& \cb
	\arrow["s"', from=1-3, to=2-3]
	\arrow["t", shift left=3, from=1-3, to=2-3]
	\arrow["s"', from=1-7, to=2-7]
	\arrow["t", shift left=3, from=1-7, to=2-7]
	\arrow["{\overline{\alpha}}", squiggly, from=2-1, to=1-3]
	\arrow["F", shift left, from=2-1, to=2-3]
	\arrow["G"', shift right=2, from=2-1, to=2-3]
	\arrow[""{name=0, anchor=center, inner sep=0}, "{\overline{\alpha}}", squiggly, from=2-5, to=1-7]
	\arrow[""{name=1, anchor=center, inner sep=0}, "{\overline{\beta} = K(\overline{\alpha})}", shift left=3, curve={height=-12pt}, from=2-5, to=1-7]
	\arrow["F", shift left, from=2-5, to=2-7]
	\arrow["G"', shift right=2, from=2-5, to=2-7]
	\arrow["\rho"{pos=0.4}, between={0.2}{0.6}, Rightarrow, nfold, from=1, to=0]
\end{tikzcd}\]

making the diagram commute.  Since $\ca$ is $1$-free, by \Cref{thm:replacement}, we obtain the strict map $K(\overline{\alpha})$ depicted above right and invertible icon $\rho \colon K(\overline{\alpha}) \cong \overline{\alpha}$ for which $s \circ \rho = id_F$ and $t \circ \rho = id_G$.  Again using \Cref{prop:p-maps-into-PA}, $K(\overline{\alpha}) = \overline{\beta}$ for a unique pseudo-transformation $\beta$, which is semi-strict by \Cref{semistrictasstrict}.

Now consider the inverse invertible icon $\gamma = \rho^{-1}$.  As $\psd_s(\ca,s)$ and $\psd_s(\ca,t)$ are Gray-functors, it also satisfies $s \circ \gamma = id_F$ and $t \circ \gamma = id_G$.  Using $\gamma$, we will construct an invertible modification $\alpha \cong \beta$.  As an icon,  $\gamma_x \colon \overline{\beta}_x \to \overline{\alpha_x} \in \psd_s(D^3_1,\cb)$ is the identity, so we have the identity square as depicted below.

\[\begin{tikzcd}[ampersand replacement=\&]
	Fx \& Fx \\
	Gx \& Gx
	\arrow[""{name=0, anchor=center, inner sep=0}, equals, nfold, from=1-1, to=1-2]
	\arrow["{\beta_x}"', from=1-1, to=2-1]
	\arrow["{\alpha_x}", from=1-2, to=2-2]
	\arrow[""{name=1, anchor=center, inner sep=0}, equals, nfold, from=2-1, to=2-2]
	\arrow["{\gamma_x=1}"', shift left, between={0.3}{0.7}, equals, nfold, from=0, to=1]
\end{tikzcd}\]

 Note that here we are using the elementary description of $\psd_s(D^3_1,\cb)$ from \Cref{sect:pathobject}.  Furthermore, at $f \colon x \to y$, consider the component $\gamma_f$ below left.

\[\begin{tikzcd}[ampersand replacement=\&]
	{\overline{\beta}_x} \& {\overline{\alpha}_x} \&\& Fx \& Fy \& Fy \&\& Fx \& Fx \& Fy \\
	{\overline{\beta}_y} \& {\overline{\alpha}_y} \&\& Gx \& Gy \& Gy \&\& Gx \& Gx \& Gy
	\arrow[""{name=0, anchor=center, inner sep=0}, equals, nfold, from=1-1, to=1-2]
	\arrow["{\overline{\beta}_f}"', from=1-1, to=2-1]
	\arrow["{\overline{\alpha}_f}", from=1-2, to=2-2]
	\arrow[""{name=1, anchor=center, inner sep=0}, "Ff", from=1-4, to=1-5]
	\arrow["{\beta_x}"', from=1-4, to=2-4]
	\arrow[""{name=2, anchor=center, inner sep=0}, equals, nfold, from=1-5, to=1-6]
	\arrow["{\beta_y}", from=1-5, to=2-5]
	\arrow[""{name=3, anchor=center, inner sep=0}, "{\alpha_y}", from=1-6, to=2-6]
	\arrow[""{name=4, anchor=center, inner sep=0}, equals, nfold, from=1-8, to=1-9]
	\arrow[""{name=5, anchor=center, inner sep=0}, "{\beta_x}"', from=1-8, to=2-8]
	\arrow[""{name=6, anchor=center, inner sep=0}, "Ff", from=1-9, to=1-10]
	\arrow["{\alpha_x}", from=1-9, to=2-9]
	\arrow["{\alpha_y}", from=1-10, to=2-10]
	\arrow[""{name=7, anchor=center, inner sep=0}, equals, nfold, from=2-1, to=2-2]
	\arrow[""{name=8, anchor=center, inner sep=0}, "Gf"', from=2-4, to=2-5]
	\arrow[""{name=9, anchor=center, inner sep=0}, equals, nfold, from=2-5, to=2-6]
	\arrow[""{name=10, anchor=center, inner sep=0}, equals, nfold, from=2-8, to=2-9]
	\arrow[""{name=11, anchor=center, inner sep=0}, "Gf"', from=2-9, to=2-10]
	\arrow["{\gamma_f}", between={0.2}{0.8}, Rightarrow, nfold, from=0, to=7]
	\arrow["{\beta^0_f}", between={0.2}{0.8}, Rightarrow, nfold, from=1, to=8]
	\arrow[between={0.2}{0.8}, equals, nfold, from=2, to=9]
	\arrow["{\gamma_f}", between={0.2}{0.8}, Rightarrow, scaling nfold=3, from=3, to=5]
	\arrow[between={0.2}{0.8}, equals, nfold, from=4, to=10]
	\arrow["{\alpha_f^o}", between={0.2}{0.8}, Rightarrow, nfold, from=6, to=11]
\end{tikzcd}\]

As a $2$-cell in $\psd_s(D^3_1,\cb)$, this involves a pair of $2$-cells in $\cb$ and an invertible $3$-cell.  The requirement that $s \circ \gamma$ and $t \circ \gamma$ are identities forces the $2$-cells to be identities.  As such $\gamma_f$ is specified by a single invertible $3$-cell as above right.

Altogether, the components $(\gamma_x,\gamma_f)$ capture exactly the data for a modification  $\gamma \colon \alpha^o \Rightarrow \beta^o$ of the associated \emph{oplax transformations} $\alpha^o,\beta^o \colon F \rightrightarrows G$ described in Proposition 3.1 of \cite{BourkeLobbia2}.   Let us now show that this data satisfies the three axioms for a modification of oplax transformations, described in Appendix A.2 of \cite{BourkeLobbia}.  The first of these asserts that the $3$-cell $\gamma_f$ above right is an identity if $f$ is an identity $1$-cell.  This holds since, by the axioms for a pseudo-transformation, $\gamma_f$ above left is an identity if $f$ is.

To obtain the other two axioms, observe that as a pseudo-transformation, $\gamma$ comes equipped with invertible $3$-cells 

\[\begin{tikzcd}[ampersand replacement=\&]
	{\overline{\alpha}_f\gamma_x} \&\& {\gamma_y \overline{\beta}_f} \\
	{\overline{\alpha}_g\gamma_x} \&\& {\gamma_y \overline{\beta}_g}
	\arrow[""{name=0, anchor=center, inner sep=0}, "{\gamma_f}", from=1-1, to=1-3]
	\arrow["{\overline{\alpha}_\phi\circ1}"', from=1-1, to=2-1]
	\arrow["{1\circ \overline{\beta}_\phi}", from=1-3, to=2-3]
	\arrow[""{name=1, anchor=center, inner sep=0}, "{\gamma_{g}}"', from=2-1, to=2-3]
	\arrow["{\gamma_\phi =id}", between={0.3}{0.7}, Rightarrow, nfold, from=0, to=1]
\end{tikzcd}
\hspace{1cm}
\begin{tikzcd}[ampersand replacement=\&]
	{\overline{\alpha}_{f_2}\overline{\alpha}_{f_1}\gamma_x} \& {\overline{\alpha}_{f_2}\gamma_y \overline{\beta}_{f_1}} \& {\gamma_z \overline{\beta}_{f_2}\overline{\beta}_{f_1}} \\
	{\overline{\alpha}_{f_2f_1}\gamma_x} \&\& {\gamma_z \overline{\beta}_{f_2f_1}}
	\arrow["{1\circ\gamma_{f_1}}", from=1-1, to=1-2]
	\arrow["{\overline{\alpha}^2_{f_2,f_1}\circ1}"', from=1-1, to=2-1]
	\arrow["{\gamma_{f_2}\circ1}", from=1-2, to=1-3]
	\arrow["{1\circ \overline{\beta}^2_{f_2,f_1}}", from=1-3, to=2-3]
	\arrow[""{name=0, anchor=center, inner sep=0}, "{\gamma_{f_2f_1}}"', from=2-1, to=2-3]
	\arrow["{\gamma^2_{f_2,f_1} = id}", between={0.3}{0.7}, Rightarrow, nfold, from=1-2, to=0]
\end{tikzcd}\]
in $\psd_s(D^3_1,\cb)$ for $\phi \colon f \Rightarrow g$ a $2$-cell in $\ca$ and $(f_1 \colon x \to y,f_2 \colon y \to z)$ a composable pair of $1$-cells in $\ca$.  Now since $s \circ \gamma$ and $t \circ \gamma$ are identities, and since postcomposition by $s$ and $t$ reflect identity $3$-cells, it follows that the above depicted $3$-cells are identities.  Thus they give a pair of equations between parallel $2$-cells in $\psd_s(D^3_1,\cb)$.  These two equations --- expanded out using the whiskering and composition operations in   $\psd_s(D^3_1,\cb)$ \cite[Appendix B2.1/3.1]{BourkeLobbia} --- are exactly the two remaining equations for a modification $\gamma \colon \alpha^o \to \beta^o$.  By Proposition 3.4 of \cite{BourkeLobbia2} this bijectively corresponds to a modification $\gamma^l \colon \alpha \to \beta$ where $\gamma^l_x = \gamma_x$ and $\gamma^l_f$ is obtained as the mate of the invertible $3$-cell $\gamma_f$ through an adjoint equivalence.  As such, $\gamma^l$ has invertible components and is, therefore, itself invertible by \cite[Corollary B.3]{BourkeLobbia2}.
\end{proof}\black

\section{The characterisation of equivalences in $\GCat^4$}\label{sect:char}

In this section, we characterise the equivalence $n$-cells in the Gray-categories $\psd_p(\ca,\cb)$ as the pointwise equivalences.  Building on this, we characterise the equivalences $F \colon \ca \to \cb$ between cofibrant Gray-categories in the left-sharp $4$-category $\GCat^4$ as the triequivalences.

The only hard part is to show that given a pseudo-transformation $\alpha \colon F \to G$ whose $1$-cell components $\alpha_x \colon Fx \to Gx$ are equivalences (that is, biequivalences), we can assemble their equivalence inverses $\beta_x \colon Gx \to Fx$ into a pseudo-transformation $\beta \colon G \to F$, equivalence-inverse to $\alpha$.  An elementary proof using pasting diagrams, whilst possible, is extremely complicated.  Fortunately, we manage to give a model-categorical proof avoiding any diagram chases.


\subsection{The easy direction}
One direction is, of course, straightforward.

\begin{prop}\label{prop:char-equiv-in-internal-hom-trivial}
For $n=1,2,3$, each equivalence $n$-cell $\Theta$ in $\psd_p(\ca,\cb)$ is a pointwise equivalence: that is, for each $x \in \ca$, its component $\Theta_x$ is an equivalence $n$-cell in $\cb$.
\end{prop}
\begin{proof}
Restriction along the Gray-functor $x \colon 1 \to \ca$ induces the Gray-functor $ev_x \colon \psd_p(\ca,\cb) \to \psd_p(1,\cb) \cong \cb$ which sends $\Theta$ to $\Theta_x$.  Since, by \Cref{pseudomapEquivs}, Gray-functors preserve equivalence $n$-cells, this proves the claim.
\end{proof}

The central result of this section, Theorem \ref{thm:char-equiv-in-internal-hom}, establishes the converse statement.  A consequence of Proposition \ref{prop:char-equiv-in-internal-hom-trivial} is the following.

\begin{prop}\label{weak-equivalences}
Each equivalence $1$-cell $F \colon \ca \to \cb$ in the left-sharp $4$-category $\GCat^4$ is a weak equivalence of Gray-categories.
\end{prop}
\begin{proof}
By definition, $F$ is an equivalence in $\GCat^4$ just when there exists a strict map $G \colon \cb \to \ca$ and equivalences $GF \to 1_{\ca} \in \psd_s(\ca,\ca)$ and $FG \to 1_\cb \in \psd_s(\cb,\cb)$ in the hom Gray-categories.  By the preceding result such equivalences are, in particular, pseudo-transformations which are pointwise equivalences.  Therefore, by \Cref{prop:p-maps-into-PA}, they induce pseudomaps as depicted below making the two diagrams serially commute.

\[\begin{tikzcd}
	&& {P\ca} &&&& {P\cb} \\
	\ca && \ca && \cb && \cb
	\arrow["s"', from=1-3, to=2-3]
	\arrow["t", shift left=3, from=1-3, to=2-3]
	\arrow["s"', from=1-7, to=2-7]
	\arrow["t", shift left=3, from=1-7, to=2-7]
	\arrow["\exists", squiggly, from=2-1, to=1-3]
	\arrow["{1_{\ca}}"'{pos=0.7}, shift right=2, from=2-1, to=2-3]
	\arrow["GF"{pos=0.7}, shift left, from=2-1, to=2-3]
	\arrow["\exists", squiggly, from=2-5, to=1-7]
	\arrow["{1_{\cb}}"'{pos=0.7}, shift right=2, from=2-5, to=2-7]
	\arrow["FG"{pos=0.7}, shift left, from=2-5, to=2-7]
\end{tikzcd}\]

By the universal property of the pseudomap classifiers $Q\ca$ and $Q\cb$, this is equally to say that there exist strict lifts as below making the two diagrams serially commute.

\[\begin{tikzcd}
	{Q\ca} && {P\ca} && {Q\cb} && {P\cb} \\
	\ca && \ca && \cb && \cb
	\arrow["{\exists  H_1}", from=1-1, to=1-3]
	\arrow["{p_{\ca}}"', from=1-1, to=2-1]
	\arrow["s"', from=1-3, to=2-3]
	\arrow["t", shift left=3, from=1-3, to=2-3]
	\arrow["{\exists H_2}", from=1-5, to=1-7]
	\arrow["{p_{\cb}}"', from=1-5, to=2-5]
	\arrow["s"', from=1-7, to=2-7]
	\arrow["t", shift left=3, from=1-7, to=2-7]
	\arrow["{1_{\ca}}"'{pos=0.7}, shift right=2, from=2-1, to=2-3]
	\arrow["GF"{pos=0.7}, shift left, from=2-1, to=2-3]
	\arrow["{1_{\cb}}"'{pos=0.7}, shift right=2, from=2-5, to=2-7]
	\arrow["FG"{pos=0.7}, shift left, from=2-5, to=2-7]
\end{tikzcd}\]

By \Cref{counitnormal} and \Cref{pathprojection} the vertical maps are weak equivalences.  Since the horizontal identities are weak equivalences, it therefore follows by $2$-out-of-$3$ that $H_1$ and $H_2$ are weak equivalences and hence that $FG$ and $GF$ are weak equivalences by $2$-out-of-$3$ again.  Therefore, by $2$-out-of-$6$, $F$ and $G$ are both weak equivalences.

\end{proof}

\subsection{An auxiliary construction}
 In \Cref{sect:hard}, we will make use of Gray-categories of the form $\psd_p(\ca,\cb)$ where $\ca$ is the (co)domain of a generating (trivial) cofibration.  In most cases, such $\ca$ are uniquely determined by $2$-categories, in the sense of the following definition, and in that case the Gray-categories  $\psd_p(\ca,\cb)$ can be described in purely $2$-categorical terms.  This makes dealing with such cases rather easy.
%
%
\begin{defn}
We define a functor $D \colon \twocat \to \GCat$ as follows.
\begin{itemize}
\item Given a $2$-category $\cc$, $D\cc$ is the Gray-category with two objects $0,1$ and hom 2-categories defined by $D\cc(0,1) = \cc, D\cc(1,0) = \emptyset, D\cc(0,0)=1, D\cc(1,1)=1$.  
\item Given a $2$-functor $F \colon \cb \to \cc$, $DF \colon D\cb \to D\cc$ is the identity-on-objects Gray-functor whose hom $2$-functor $D\cb(0,1) \to D\cc(0,1)$ is $F$ itself.
\end{itemize}
\end{defn}

 The following lemma describes the Gray-category $\psd_p(D\cc,\ca)$ in $2$-categorical terms, using $2$-functors, pseudonatural transformations and modifications.
  
 \begin{lemma}\label{lem:aux}
 The Gray-category $\psd_p(D\cc,\ca)$ can be described as follows:
 \begin{enumerate}
 \item The $0$-cells $F$ are specified by objects $F0,F1 \in \ca$ together with a $2$-functor $F \colon \cc \to \ca(F0,F1)$;
 \item The $1$-cells $\alpha \colon F \to G$ are specified by $0$-cells $\alpha_i \in \ca(Fi,Gi)$ for $i=0,1$ together with a pseudonatural transformation
\[\begin{tikzcd}[ampersand replacement=\&]
	\& {} \\
	\cc \& {\ca(G0,G1)} \\
	{\ca(F0,F1)} \& {\ca(F0,G1)}
	\arrow["G", from=2-1, to=2-2]
	\arrow[""{name=0, anchor=center, inner sep=0}, "F"', from=2-1, to=3-1]
	\arrow[""{name=1, anchor=center, inner sep=0}, "{(\alpha_0)^\star}", from=2-2, to=3-2]
	\arrow["{(\alpha_1)_\star}"', from=3-1, to=3-2]
	\arrow["\alpha"', between={0.3}{0.7}, Rightarrow, nfold, from=1, to=0]
\end{tikzcd}\]
 whose components $\alpha_x$ are equipped with the structure of adjoint equivalences;
  \item The $2$-cells $\Gamma \colon \alpha \to \beta$ are specified by $1$-cells $\Gamma_i \colon \alpha_i \to \beta_i \in \ca(Fi,Gi)$ for $i=0,1$ together with an invertible modification
\[\begin{tikzcd}
	& {} & \\
	{(\alpha_0)^\star G} && {(\alpha_1)_\star F} \\
	{(\beta_0)^\star G} && {(\beta_1)_\star F}
	\arrow["\alpha", from=2-1, to=2-3]
	\arrow[""{name=0, anchor=center, inner sep=0}, "{(\Gamma_0)^\star G}"', from=2-1, to=3-1]
	\arrow[""{name=1, anchor=center, inner sep=0}, "{(\Gamma_1)_\star F}", from=2-3, to=3-3]
	\arrow["\beta"', from=3-1, to=3-3]
	\arrow["\Gamma"', between={0.3}{0.7}, Rightarrow, from=1, to=0]
\end{tikzcd}\]
 \item The $3$-cells $\Theta \colon \Gamma \to \Delta$ are specified by $2$-cells $\Theta_i \colon \Gamma_i \to \Delta_i \in \ca(Fi,Gi)$ for $i=0,1$ such that we have an equality of modifications
\[\begin{tikzcd}
	& {} &&&&&&& \\
	{(\alpha_0)^\star G} && {(\alpha_1)_\star F} &&&& {(\alpha_0)^\star G} && {(\alpha_1)_\star F} \\
	{(\beta_0)^\star G} && {(\beta_1)_\star F} &&&& {(\beta_0)^\star G} && {(\beta_1)_\star F}
	\arrow["\alpha", from=2-1, to=2-3]
	\arrow[""{name=0, anchor=center, inner sep=0}, "{(\Gamma_0)^\star G}", curve={height=-18pt}, from=2-1, to=3-1]
	\arrow[""{name=1, anchor=center, inner sep=0}, "{(\Delta_0)^\star G}"', curve={height=18pt}, from=2-1, to=3-1]
	\arrow[""{name=2, anchor=center, inner sep=0}, "{(\Gamma_1)_\star F}", curve={height=-18pt}, from=2-3, to=3-3]
	\arrow["\alpha", from=2-7, to=2-9]
	\arrow[""{name=3, anchor=center, inner sep=0}, "{(\Delta_0)^\star G}"', curve={height=18pt}, from=2-7, to=3-7]
	\arrow[""{name=4, anchor=center, inner sep=0}, "{(\Gamma_1)_\star F}", curve={height=-18pt}, from=2-9, to=3-9]
	\arrow[""{name=5, anchor=center, inner sep=0}, "{(\Delta_1)_\star G}"', curve={height=24pt}, from=2-9, to=3-9]
	\arrow["\beta"', from=3-1, to=3-3]
	\arrow["\beta"', from=3-7, to=3-9]
	\arrow["{(\Theta_0)^\star G}"', shift left, between={0.2}{0.8}, Rightarrow, from=0, to=1]
	\arrow["\Gamma"'{pos=0.4}, between={0.2}{0.6}, Rightarrow, from=2, to=0]
	\arrow[between={0.4}{0.6}, equals, from=2, to=3]
	\arrow["\Delta"'{pos=0.6}, between={0.4}{0.8}, Rightarrow, from=5, to=3]
	\arrow["{(\Delta_1)_\star G}"', shift left, between={0.2}{0.6}, Rightarrow, from=4, to=5]
\end{tikzcd}\] 
 \end{enumerate}
 
Consider a $2$-functor $H \colon \cd \to \cc$.  Under the above correspondence, the induced Gray-functor $(DH)^{\star} \colon \psd_p(D\cd,\ca) \to \psd_p(D\cc,\ca)$ sends $(F,F0,F1)$ to $(FH,F0,F1)$, sends $(\alpha,\alpha_0,\alpha_1)$ to $(\alpha \circ H,\alpha_0,\alpha_1)$, with a similar action on $2$-cells and $3$-cells.
 \end{lemma}
\begin{proof}
This is simply a matter of comparing the definitions.
\end{proof}

\begin{rmk}
Note that for Gray-categories of the form $D\cc$ we have $\psd_s(D\cc,\ca)= \psd_p(D\cc,\ca)$.  Indeed, as $D\cc$ has no non-trivial composite $1$-cells, each pseudomap with domain $D\cc$ is strict.  We will freely use this fact in later sections.
\end{rmk}
 
 \begin{lemma}\label{lem:bo}
Let $H$ be a trivial cofibration of $2$-categories.  Then $(DH)^{\star} \colon \psd_p(D\cd,\ca) \to \psd_p(D\cc,\ca)$ is a $2$-trivial fibration of Gray-categories
 \end{lemma}
 \begin{proof}
 Since all objects are fibrant in $\twocat$, $H$ is a split mono.  Therefore $DH$ is a split mono too, whence $(DH)^{\star}$ is a split epi.  In particular, it is surjective on objects.
 Now consider $F,G \colon D\cd \rightrightarrows \ca$ and set $\cb = \ca(F0,G1)$.  As explained in \Cref{sect:2model}, the Gray-tensor product equips $\twocat$ with the structure of a monoidal model category.  Therefore since $H$ is a trivial cofibration and $\cb$, like any $2$-category, is fibrant, it follows that the $2$-functor $H^{\star} \colon \psd^2(\cd, \cb) \to \psd^2(\cc, \cb)$ is a trivial fibration.
 
 Now consider $P=(\alpha_0)^\star G$ and $Q=(\alpha_1)_\star F$.  Then, following \Cref{lem:aux}, to say that $(DH)^{\star}$ is full on $1$-cells is to say that given a pseudonatural transformation $\beta \colon PH \to QH$ equipped with the structure of an adjoint equivalence in $\psd^2(\cc, \cb)$, there exists a pseudonatural transformation  $\alpha \colon P \to Q$ equipped with the structure of an adjoint equivalence in $\psd^2(\cc, \ca)$ such that $\alpha \circ H = \beta$ and likewise $H^\star$ preserves the adjoint equivalences.  This holds since $H^{\star}$, like any trivial fibration, is full on $1$-cells and also lifts adjoint equivalences (since $j^2_1 \colon D^2_1 \to E^2_1$ is a cofibration).
 
%
%
%

 To say that $(DH)^{\star}$ is fully faithful on $2$-cells amounts to showing that certain invertible modifications with domain $\cd$ lift uniquely along $H^{\star} \colon \psd^2(\cd, \cb) \to \psd^2(\cc, \cb)$ whilst the corresponding condition for $3$-cells amounts to showing that certain equations between invertible modifications hold just when they hold under restriction along $H$.  According, these conditions hold if $(H)^{\star}$ is fully faithful on invertible $2$-cells which, like any trivial fibration of 2-categories, it is.
 \end{proof}
 
 \begin{cor}\label{twofibrations}
Consider the generating trivial cofibrations $l_n \colon D^3_n \to E^3_{n+1}$ for $n=1,2$.  The Gray-functor $(l_n)^{\star} \colon \psd_s(E^3_{n+1},\ca) \to \psd_s(D^3_n,\ca)$ is a $2$-trivial fibration. \end{cor}
\begin{proof}
These are $DH$ for the two generating trivial cofibrations in $\twocat$.  Therefore the claim holds by the preceding lemma.
\end{proof}

\subsection{The hard direction and the full characterisation}\label{sect:hard}

We now turn to the hard part.  Following Section \ref{sect:equiv-gray}, consider the free adjoint equivalence $1$-cell $E^3_1$ and the morphism $j^3_1\colon D^3_1 \to E^3_1 \in \GCat$ selecting the equivalence $1$-cell.  Then the induced restriction along the diagonal below

\begin{equation}\label{K}
\begin{tikzcd}
	{\psd_s(E^3_1,\ca)} & {P\ca} \\
	& {\psd_s(D^3_1,\ca)}
	\arrow["K", from=1-1, to=1-2]
	\arrow["{(j^3_1)^{\star}}"', from=1-1, to=2-2]
	\arrow[hook, from=1-2, to=2-2]
\end{tikzcd}
\end{equation}

factors through $P\ca \hookrightarrow Ps(D^3_1,\ca)$ as depicted.  We will prove in Proposition \ref{prop:kn} that the induced $K \colon \psd_s(E^3_1,\ca) \to P\ca$ is a $2$-trivial fibration.  

\begin{rmk} Why will this be useful?  A pointwise equivalence  $\eta \colon F \to G \in \psd_p(\cb,\ca)$ amounts to a certain pseudomap $\eta\colon \cb \rightsquigarrow P\ca$.  Now since $K$, as a $2$-trivial fibration, admits a section, we can form the composite $\cb \rightsquigarrow P\ca \rightsquigarrow \psd_s(E^3_1,\ca) \hookrightarrow \psd_p(E^3_1,\ca) $, whose adjoint transpose $E^3_1 \rightsquigarrow \psd_p(\cb,\ca)$  exhibits $\eta$ as an equivalence $1$-cell in $\psd_p(\cb,\ca)$
\end{rmk}

Firstly, we need a lemma.

\begin{lemma}\label{semistrict2}
Let $\cc$ be a $2$-category.  A Gray-functor $F \colon D\cc \to \psd_s(\ca,\cb)$ takes values in semi-strict transformations if and only if its transpose $\overline{F} \colon \ca \rightsquigarrow \psd_s(D\cc,\cb)$ is a Gray-functor.
\end{lemma}
\begin{proof}
Each object $f$ of $\cc$ specifies a morphism $f \colon 0 \to 1$ of $D\cc$ and so a Gray-functor $f \colon D^3_1 \to D\cc$.  Then to say that $F$ takes values in semi-strict transformations is to to say that each composite $F \circ f \colon D^3_1 \to \psd_s(\ca,\cb)$ specifies a semi-strict transformation.  By \Cref{semistrictasstrict}, this is equally to say that the transpose $\overline{F \circ f} \colon \ca \to  \psd_s(D^3_1,\cb)$ is strict and by naturality of the biclosedness bijections, we have that $\overline{F \circ f} = \psd_s(f,\cb) \circ \overline{F} \colon \ca \rightsquigarrow \psd_s(D\cc,\cb) \to \psd_s(D^3_1,\cb)$.  

\black Now to say that $\overline{F}$ is strict is to say that the pseudo-modifications $\overline{F}^2_{f_1,f_2}$ are identities.  Therefore it suffices to show that  the Gray-functors $\psd_s(f,\cb) \colon \psd_s(D\cc,\cb) \to \psd_s(D^3_1,\cb)$, which evaluate at $f$, jointly reflect identity $2$-cells.  Consider a pseudo-modification $\Gamma \colon \alpha \to \beta$ and suppose that each $\psd_s(f,\cb) (\Gamma)$ is the identity.  This is to say that its domains and codomains agree --- that is $(\alpha_0, \alpha_f^o \dashv \alpha_f,\alpha_1) = (\beta_0, \beta_f^o \dashv \beta_f,\beta_1)$ --- and that the components $\Gamma_0$, $\Gamma_1$ and $\Gamma_f$ are identities for all $f$.  Since these are all the components of $\Gamma$, it remains to show that $\alpha = \beta$.  Here, the only remaining thing to show is that $\alpha_{\phi} = \beta_{\phi}$ at a $1$-cell $\phi$ of $\cc$.  But this is precisely the modification condition A.3(iii) \cite{BourkeLobbia} for $\Gamma$ at $\phi$ given that the components of $\Gamma$ are identities.\black
%
\end{proof}

\begin{prop}\label{prop:kn}
$K \colon \psd_s(E^3_1,\ca) \to P\ca$ is a 2-trivial fibration.
\end{prop}
\begin{proof}
Since the proof is non-trivial, we break it into steps.
\begin{enumerate}[leftmargin=*, label=(\roman*)]
\item \emph{Firstly, observe that $K$ is surjective-on-objects (s.o.).}  Indeed, by \cite[Proposition 3.1]{Lack2011A-quillen}, each equivalence $1$-cell in a Gray-category can be promoted to an adjoint biequivalence.
\item \emph{Next, we prove that $K$ is a weak equivalence.}  Indeed, we have a commutative triangle
\[\begin{tikzcd}[ampersand replacement=\&]
	{\psd_s(E^3_1,\ca)} \&\& {P\ca} \\
	\& \ca
	\arrow["K", from=1-1, to=1-3]
	\arrow["{ev_0}"', from=1-1, to=2-2]
	\arrow["s", from=1-3, to=2-2]
\end{tikzcd}\]
in which $s$ is a weak equivalence, since $P\ca$ is a path object for the model structure on Gray-categories.  Therefore, to show that $K$ is a weak equivalence, it suffices to show that $ev_0$ is a weak equivalence.  Now this is obtained by applying the loose $4$-functor $\psd_s(-,\ca) \colon (\GCat^4)^{op} \to \GCat^4$  to the generating trivial cofibration $l_0 \colon 1= D^3_0 \to E^3_1$.  Since, by \Cref{looseequivalence}, loose $4$-functors preserve equivalence $1$-cells, it suffices to show $l_0$ is an equivalence $1$-cell in $\GCat^4$.  

Since $l_0$ is a (generating) trivial cofibration, its retraction $! \colon E^3_1 \to 1$ is a weak equivalence, by $2$-out-of-$3$.  In fact, since all objects in $\GCat$ are fibrant, $!$ is therefore a trivial fibration.

Since $! \circ l_0$ is the identity, to prove that $l_0$ is an equivalence, it remains to construct an equivalence $1$-cell $\Delta_0 :=  l_0\circ ! \simeq id_{E_{1}} \in \psd_s(E^3_1,E^3_1)$.    Now as $E^3_1$ is cofibrant, it follows from \Cref{mc2} below that $\psd_s(E^3_1,E^3_1) \to \psd_s(E^3_1,1)=1$ is a trivial fibration, hence $\psd_s(E^3_1,E^3_1)$ is contractible --- in particular, any two objects therein are equivalent, so that $\Delta_0 \simeq id$, as required.

\item \emph{It remains to prove that the $2$-functor $\psd_s(E^3_1,\ca)(F,G) \to \psd_s(D^3_1,\ca)(F\circ j^3_1,G \circ j^3_1)$ is a s.o. $2$-equivalence.}  As the hom $2$-functor action of $K$, it is a biequivalence and so fully faithful on $2$-cells automatically.  For faithfulness on $1$-cells, consider $\Lambda,\Gamma \colon \alpha \rightrightarrows \beta \in \psd_s(E^3_1,\ca)(F,G)$ and suppose that $ \Lambda \circ j^3_1 = \Gamma \circ j^3_1$.  Then by fully faithfulness on $2$-cells, there exists a unique invertible $3$-cell $\Theta \colon \Lambda \cong \Gamma$ such that $\Theta \circ j^3_1 = id$.   Since $j^3_1$ is identity-on-objects and $\Theta$ is specified by its $3$-cell components $\Theta$ for $x \in E^3_1$, it follows that each $\Theta$ is an identity, and therefore that $\Lambda = \Gamma$, as required.  
\item \emph{Therefore it remains to establish that $\psd_s(E^3_1,\ca)(F,G) \to \psd_s(D^3_1,\ca)(F \circ j^3_1,G \circ j^3_1)$ is s.o. and full on $1$-cells.}   Firstly, let $\psd^{ss}_s(E^3_1,\ca)(F,G) \hookrightarrow \psd_s(E^3_1,\ca)(F,G)$ be the full sub-$2$-category on the semi-strict transformations.  Since $E^3_1$ is cofibrant, by \Cref{thm:replacement2}, the inclusion is a $2$-equivalence.    Indeed, we have a commutative square 

\begin{equation}\label{semistrict}
\begin{tikzcd}[ampersand replacement=\&]
	{\psd^{ss}_s(E^3_1,\ca)(F,G)} \& {\psd_s(E^3_1,\ca)(F,G)} \\
	{\psd^{ss}_s(D_1,\ca)(F \circ j^3_1,G \circ j^3_1)} \& {\psd_s(D^3_1,\ca)(F \circ j^3_1,G \circ j^3_1)}
	\arrow[hook, from=1-1, to=1-2]
	\arrow["{(j^3_1)^{\star}}"', from=1-1, to=2-1]
	\arrow["{(j^3_1)^{\star}}", from=1-2, to=2-2]
	\arrow[equals, nfold, from=2-1, to=2-2]
\end{tikzcd}
\end{equation}
where the top is a $2$-equivalence and the bottom an identity, since $D^3_1$ contains no composable non-identity $1$-cells.  Note that since the right leg is a biequivalence, so too is \emph{the left leg}.

\item \emph{Let us prove that $({j^3_1})^\star \colon \psd_s(E^3_1,\ca)(F,G) \to \psd_s(D^3_1,\ca)(F \circ j^3_1,G \circ j^3_1)$ is s.o.} Consider a $0$-cell $\alpha \in \psd_s(D^3_1,\ca)(F \circ j^3_1,G \circ j^3_1)$.  Since the left leg in \eqref{semistrict} is a biequivalence, there exists $\beta \in \psd^{ss}_s(E^3_1,\ca)(F,G)$ and an equivalence $\Lambda \colon  \beta \circ j^3_1 \simeq \alpha \in \psd_s(D^3_1,\ca)(F \circ j^3_1,G \circ j^3_1)$.  Therefore, it suffices to show that given such an equivalence, i.e. one with $\beta$ semistrict, there exists an equivalence $\Gamma \colon \beta \simeq \alpha' \in \psd_s(E^3_1,\ca)(F, G)$ such that $ \Gamma  \circ j^3_1 = \Lambda$.  

We will prove a slightly stronger claim: namely, that given an \emph{adjoint equivalence} $2$-cell $\Lambda \colon  \beta \circ j^3_1 \simeq \alpha \in \psd_s(D^3_1,\ca)$, there exists an adjoint equivalence $2$-cell $\Gamma \colon \beta \simeq \alpha' \in \psd_s(E^3_1,\ca)$ such that  $\Gamma  \circ j^3_1 = \Lambda$.  Note that this implies the desired result since each adjoint equivalence $2$-cell can be made into an adjoint equivalence $2$-cell (see \Cref{sect:equiv-gray}).  Now the stronger claim says that given a commutative square as on the left below, in which the upper horizontal take values in semi-strict transformations, then it admits a diagonal filler as depicted.

\[\begin{tikzcd}[ampersand replacement=\&]
	{D^3_1} \& {\psd_s(E^3_1,\ca)} \&\& {D^3_1} \& {\psd_s(E^3_1,\ca)} \& {\psd_p(E^3_1,\ca)} \\
	{E^3_2} \& {\psd_s(D^3_1,\ca)} \&\& {E^3_2} \& {\psd_s(D^3_1,\ca)} \& {\psd_p(D^3_1,\ca)}
	\arrow["\beta", from=1-1, to=1-2]
	\arrow["{l_1}"', from=1-1, to=2-1]
	\arrow["{({j^3_1}) ^{\star}}", from=1-2, to=2-2]
	\arrow["\beta", from=1-4, to=1-5]
	\arrow["{l_1}"', from=1-4, to=2-4]
	\arrow["J", hook, from=1-5, to=1-6]
	\arrow["{({j^3_1}) ^{\star}}", from=1-6, to=2-6]
	\arrow["\exists", from=2-1, to=1-2]
	\arrow["\Lambda"', from=2-1, to=2-2]
	\arrow["\exists", from=2-4, to=1-6]
	\arrow["\Lambda"', from=2-4, to=2-5]
	\arrow["{J=Id}"', equals, nfold, from=2-5, to=2-6]
\end{tikzcd}\]

Now since $l_1$ is identity-on-objects and $J \colon \psd_s(E^3_1,\ca) \to \psd_p(E^3_1,\ca)$ is enriched fully-faithful, they are orthogonal.  Hence to give a filler in the square on the left above is equally to give one in the square on the right.  Note that any pseudomorphism which is a diagonal filler is necessarily strict since $E^3_2$ contains no composable non-identity $1$-cells.  Therefore, by biclosedness, such a diagonal filler corresponds to a diagonal filler for the transposed square below.

\[\begin{tikzcd}[ampersand replacement=\&]
	{D^3_1} \& {\psd_p(E^3_2,\ca)} \\
	{E^3_1} \& {\psd_p(D^3_1,\ca)}
	\arrow["{\overline{\Lambda}}", from=1-1, to=1-2]
	\arrow["{j^3_1}"', from=1-1, to=2-1]
	\arrow["{(l_1) ^{\star}}", from=1-2, to=2-2]
	\arrow["\exists", squiggly, from=2-1, to=1-2]
	\arrow["{\overline{J \circ \beta}}"', from=2-1, to=2-2]
\end{tikzcd}\]

The important thing here to note is that both horizontal transposes are strict.  The top one since all pseudomaps with domain $D^3_1$ are strict and the bottom one by Lemma \ref{semistrict2} since we assumed that $\beta$ takes values in semi-strict transformations.  Therefore it suffices to solve the above the lifting problem.  Now the left leg is a cofibration and the right leg a trivial fibration (by Proposition \ref{twofibrations} noting that the $s/p$ subscript makes no difference since pseudomaps and strict ones coincide in these cases).  Therefore such a filler exists (indeed, a strict one).

\item \emph{Finally we prove that $({j^3_1})^\star \colon \psd_s(E^3_1,\ca)(F,G) \to \psd_s(D^3_1,\ca)(F \circ j^3_1,G \circ j^3_1)$ is full on $1$-cells.}  Since the top horizontal in \eqref{semistrict} is an equivalence of underlying $1$-categories, it suffices to show that $({j^3_1})^\star$ is full on pseudo-modifications between semi-strict transformations.   (Indeed, this is a special case of the easy fact that given functors $U \colon A \to B$ and $V \colon B \to C$ with $U$ an equivalence, then $V$ is full if it is full on morphisms between objects in the image of $U$).  So consider $\alpha, \beta \in \psd^{ss}_s(E^3_1,\ca)(F,G)$ and $\Lambda \colon \alpha \circ j^3_1 \to \beta \circ j^3_1 \in \psd_s(D^3_1,\ca)(F \circ j^3_1,G \circ j^3_1)$.  As $(j^3_1)^{\star}$ is essentially full, there exists $\Gamma \colon \alpha \to \beta$ and an invertible $\Theta \colon \Gamma \circ j^3_1 \cong \Lambda$.  Therefore, given such a $\Theta$, it suffices to lift it to an invertible $\Theta' \colon \Gamma \cong \Lambda'$ such that $\Theta' \circ j^3_1 = \Theta$.  This is to say that each diagram whose top horizontal takes values in semi-strict transformations

\[\begin{tikzcd}[ampersand replacement=\&]
	{D^3_2} \& {\psd_s(E^3_1,\ca)} \\
	{E^3_3} \& {\psd_s(D^3_1,\ca)}
	\arrow["\Gamma", from=1-1, to=1-2]
	\arrow["{l_2}"', from=1-1, to=2-1]
	\arrow["{({j^3_1}) ^{\star}}", from=1-2, to=2-2]
	\arrow["\exists", from=2-1, to=1-2]
	\arrow["\Theta"', from=2-1, to=2-2]
\end{tikzcd}\]
admits a diagonal filler.  The proof is identical in form to the preceding case, but now using that $(l_2)^\star \colon \psd_s(E^3_3,\ca) \to \psd_s(D^3_2,\ca)$ is a trivial fibration, as established in Proposition \ref{twofibrations}.
\end{enumerate}

\end{proof}

\begin{theorem}\label{thm:char-equiv-in-internal-hom}
For $n=1,2,3$, the equivalence $n$-cells in $\psd_p(\ca,\cb)$ are the pointwise equivalences.
\end{theorem}

\begin{proof}
The routine direction was established in \Cref{prop:char-equiv-in-internal-hom-trivial}.  Conversely, the relatively straightforward cases $n=2,3$ are established in \cite[Proposition B.1 and Corollary B.3]{BourkeLobbia2} using pasting diagrams.   Therefore it remains to prove that if a $1$-cell $\alpha \colon F \to G \in \psd_p(\ca,\cb)$ is a pointwise equivalence, then it is an equivalence.  Now $\alpha$ is specified by a strict morphism $\alpha \colon D^3_1 \to \psd_p(\ca,\cb)$.  Since each equivalence $1$-cell in a Gray-category can be made into an adjoint equivalence $1$-cell (\cite[Proposition 3.1]{Lack2011A-quillen}) it follows that $\alpha$ is an equivalence just when there exists an extension to a strict map as below left.

\[\begin{tikzcd}[ampersand replacement=\&]
	{D^3_1} \& {\psd_p(\ca,\cb)} \& {D^3_1} \& {\psd_p(\ca,\cb)} \\
	{E^3_1} \&\& {E^3_1}
	\arrow["\alpha", from=1-1, to=1-2]
	\arrow["{j^3_1}"', from=1-1, to=2-1]
	\arrow["\alpha", from=1-3, to=1-4]
	\arrow["{j^3_1}"', from=1-3, to=2-3]
	\arrow["\exists"', from=2-1, to=1-2]
	\arrow["\exists"', squiggly, from=2-3, to=1-4]
\end{tikzcd}\]
%
%
On underlying categories, the Gray-functor $j^3_n \colon D^3_n \to E^3_n$ selects one of the two atomic $1$-cells of the free category on the graph $0 \leftrightarrows 1$.  As such it is a $1$-free morphism.  Therefore by \Cref{thm:replacement}, an extension as above left exists if and only if there exists a pseudomap making the diagram above right commute.  By biclosedness, this is equally to say that there exists a pseudomap $S$ making the diagram below left commute.

\[\begin{tikzcd}[ampersand replacement=\&]
	\& {\psd_p(E^3_1,\cb)} \&\&\& {P_1(\cb)} \\
	\ca \& {\psd_p(D^3_1,\cb)} \&\& \ca \& {\psd_p(D^3_1,\cb)}
	\arrow["{(j^3_1)^\star}", from=1-2, to=2-2]
	\arrow[hook, from=1-5, to=2-5]
	\arrow["{\exists S}", squiggly, from=2-1, to=1-2]
	\arrow["{\overline{\alpha}}"', squiggly, from=2-1, to=2-2]
	\arrow["{\exists T}", squiggly, from=2-4, to=1-5]
	\arrow["{\overline{\alpha}}"', squiggly, from=2-4, to=2-5]
\end{tikzcd}\]


 By \Cref{prop:p-maps-into-PA} on the other hand, $\Gamma_x$ is an equivalence for each $x \in \ca$ just when there exists a $T$ making the diagram above right commute.  

Therefore we must show that given such a $T$ as in the right diagram, we obtain an $S$ as in the left diagram.  Now by \Cref{prop:kn}, $K$ is a $2$-trivial fibration, and so by \Cref{section} it has a section $V$ in $\GCat_p$.  Therefore the diagram below left commutes.  
\[\begin{tikzcd}[ampersand replacement=\&]
	{\psd_s(E^3_1,\cb)} \& {P\cb} \&\&\& {\psd_s(E^3_1,\cb)} \& {\psd_p(E^3_1,\cb)} \\
	\ca \& { \psd_p(D^3_1,\cb)} \&\& \ca \& {P\cb} \& { \psd_p(D^3_1,\cb)}
	\arrow["K", from=1-1, to=1-2]
	\arrow[hook, from=1-2, to=2-2]
	\arrow[hook, from=1-5, to=1-6]
	\arrow["K"', from=1-5, to=2-5]
	\arrow["{(j^3_1)^{\star}}"', from=1-6, to=2-6]
	\arrow["{V \circ T}", squiggly, from=2-1, to=1-1]
	\arrow["T", squiggly, from=2-1, to=1-2]
	\arrow["{\overline{\alpha}}"', squiggly, from=2-1, to=2-2]
	\arrow["{V\circ T}", squiggly, from=2-4, to=1-5]
	\arrow["T", squiggly, from=2-4, to=2-5]
	\arrow[hook, from=2-5, to=2-6]
\end{tikzcd}\]

Now since the square in the diagram above right commutes, setting $S$ to be the composite diagonal provides the sought-for lifting.
\end{proof}

\subsection{Equivalences in $\GCat^4$ and further consequences}
With the above in place, it is straightforward to prove the following characterisation of equivalences in the left-sharp $4$-category $\GCat^4$.
\begin{theorem}\label{thm:char-equiv-in-GCat_s}
In the left-sharp $4$-category $\GCat^4$:
\begin{enumerate}
\item  For $n>1$, the $n$-cell equivalences are the pointwise equivalences;
\item Each $1$-cell equivalence $F \colon \ca \to \cb$ is a triequivalence and, if $\ca$ and $\cb$ are cofibrant, then $F$ is a $1$-cell equivalence if and only if it is a triequivalence.
\end{enumerate}
\end{theorem}
\begin{proof}
Since $\GCat^4(\ca,\cb) = \psd_s(\ca,\cb) \hookrightarrow \psd_p(\ca,\cb)$, the first claim is a special case of \Cref{thm:char-equiv-in-internal-hom}.   

It remains to prove (2).  The first part of it was proved in Proposition \ref{weak-equivalences}.  For the second part, suppose that $\ca$ and $\cb$ are cofibrant.  Since all objects are fibrant, then $F \colon \ca \to \cb$ will be a weak equivalence if and only if it is a homotopy equivalence: that is, if and only if there exists a strict map $G \colon \cb \to \ca$ and strict morphisms $\ca \to P\ca$ and $\cb \to P\cb$ making the diagrams below commute.

\[\begin{tikzcd}
	&& {P\ca} &&&& {P\cb} \\
	\ca && \ca && \cb && \cb
	\arrow["s"', from=1-3, to=2-3]
	\arrow["t", shift left=3, from=1-3, to=2-3]
	\arrow["s"', from=1-7, to=2-7]
	\arrow["t", shift left=3, from=1-7, to=2-7]
	\arrow["\exists", from=2-1, to=1-3]
	\arrow["{1_{\ca}}"'{pos=0.7}, shift right=2, from=2-1, to=2-3]
	\arrow["GF"{pos=0.7}, shift left, from=2-1, to=2-3]
	\arrow["\exists", from=2-5, to=1-7]
	\arrow["{1_{\cb}}"'{pos=0.7}, shift right=2, from=2-5, to=2-7]
	\arrow["FG"{pos=0.7}, shift left, from=2-5, to=2-7]
\end{tikzcd}\]

These define pointwise equivalences $GF \to 1_\ca \in \psd_p(\ca,\ca)$ and $FG \to 1_\cb \in \psd_p(\cb,\cb)$.  But since each pointwise equivalence $2$-cell is, by the above, an equivalence, this completes the proof.
\end{proof}

The next result shows that the path object functor commutes with the internal hom $\psd_p(\ca,\cb)$, and will be important in the following section.  Our slight variation $P\ca$ of Lack's path object was designed so that we obtain a natural isomorphism in the following, rather than just an equivalence.

\begin{prop}\label{prop:P-and-psd}
Let $\ca,\cb\in\GCat_p$. There is a natural isomorphism $\phi_{\ca,\cb} \colon\psd_p(\ca,P\cb)\rightsquigarrow P(\psd_p(\ca,\cb)) \in \GCat_p$ making the triangles below commutative. 
\[\begin{tikzcd}[ampersand replacement=\&]
	{\psd_p(\ca,P\cb)} \&\& {P(\psd_p(\ca,\cb))} \\
	\& {\psd_p(\ca,\cb)}
	\arrow["{\phi_{\ca,\cb}}", squiggly, from=1-1, to=1-3]
	\arrow["{t_\ast}", shift left, from=1-1, to=2-2]
	\arrow["{s_\ast}"', shift right, from=1-1, to=2-2]
	\arrow["s"', shift right, from=1-3, to=2-2]
	\arrow["t", shift left, from=1-3, to=2-2]
\end{tikzcd}\]
\end{prop}

\begin{proof}
Since $\GCatmp^l$ is a symmetric biclosed multicategory, we have the natural isomorphism as depicted on the lower horizontal below

\[\begin{tikzcd}[ampersand replacement=\&]
	{\psd_p(\ca,P\cb)} \&\& {P(\psd_p(\ca,\cb))} \\
	{\psd_p(\,\ca,\psd_p(D^3_1,\cb)\,)} \&\& {\psd_p(\,D^3_1,\psd_p(\ca,\cb)\,)} \\
	\& {\psd(\ca,\cb)}
	\arrow["{\phi_{\ca,\cb}}", squiggly, from=1-1, to=1-3]
	\arrow[hook, from=1-1, to=2-1]
	\arrow[hook, from=1-3, to=2-3]
	\arrow["{\lambda_{\ca,D_1^3,\cb}}", squiggly, from=2-1, to=2-3]
	\arrow["{t_\ast}", shift left, from=2-1, to=3-2]
	\arrow["{s_\ast}"', shift right, from=2-1, to=3-2]
	\arrow["s"', shift right, from=2-3, to=3-2]
	\arrow["t", shift left, from=2-3, to=3-2]
\end{tikzcd}\]

and, by naturality in the two morphisms $1=D_0^3 \rightrightarrows D_1^3$, it commutes with the source and target mappings in the lower part of the diagram.  

Now $\psd_p(\ca,P\cb)$ is the full sub-Gray-category of $\psd_p(\,\ca,\psd_p(D^3_1,\cb)\,)$ consisting of the pseudo-transformations which are pointwise equivalence whilst $P(\psd_p(\ca,\cb))$ is the full sub-Gray-category of $\psd_p(\,D^3_1,\psd_p(\ca,\cb)\,)$ consisting of the equivalences.  Since by \Cref{thm:char-equiv-in-internal-hom} the pointwise equivalences and equivalences coincide, the isomorphism restricts as in the upper rectangle.
%
%
\end{proof}

\section{The homotopy category $\Ho(\GCat)$ is symmetric monoidal closed}\label{sect:homotopycat}

In this section we investigate the interaction between the left-sharp skew monoidal structure $(\GCat,\otimes_\sharp,1)$ and the Quillen model structure on $\GCat$.  Firstly, we show that the left-sharp skew structure induces a closed skew monoidal structure on the homotopy category $\Ho(\GCat)$.  Building on this, in \Cref{thm:smchomotopy} we show that the  homotopy category is actually symmetric monoidal closed.



\subsection{The left derived skew monoidal structure}

In Section 5 of \cite{Bourke2016Skew}, the first-named author introduced a weakening of the notion of a monoidal model structure, appropriate for closed skew monoidal categories.  In particular, it ensures that the associated homotopy category itself inherits a closed skew monoidal structure.  The key condition, Axiom MC\footnote{The equivalent conditions in Axiom MC weaken those for a left Quillen bifunctor of two variables, as appears in the notion of a monoidal model structure.  The reason for considering this weakening is that in the skew setting, we typically only have closedness on one side, and so co-continuity of the tensor product only in one variable.  This makes checking the full Quillen bifunctor condition difficult since it cannot be reduced to the usual technique of verifying it for generating (trivial) cofibrations in each variable.}, is as follows.  

\begin{defn}\label{mc}
Let $\cc$ be a Quillen model category and $(\cc,\otimes,I)$ be a closed left skew monoidal category.  Then $\cc$ satisfies Axiom MC if the unit $I$ is cofibrant and satisfies either of the equivalent conditions
\begin{enumerate}
 \item For cofibrant $X$ the functor $- \otimes X$ is left Quillen and for each cofibrant $Y$ the functor $Y \otimes -$ preserves weak equivalences between cofibrant objects.
\item For cofibrant $X$ the functor $[X,-]$ is right Quillen and for fibrant $Y$ the functor $[-,Y]$ preserves weak equivalences between cofibrant objects.
\end{enumerate}
\end{defn}

The first condition ensures that $\otimes \colon \cc^2 \to \cc$ preserves cofibrant objects and weak equivalences between them.  Therefore the skew monoidal structure restricts to the full subcategory $\cc_c$ of cofibrant objects in $\cc$.  Since the restricted tensor product $\otimes \colon{ \cc_c}^2 \to \cc_c$ preserves \emph{all} weak equivalences, it induces a skew monoidal structure $(\Ho(\cc_c),\otimes,I)$ on the homotopy category $\Ho(\cc_c)$, obtained by inverting the weak equivalences in $\cc_c$.  Since $\Ho(\cc) \simeq \Ho(\cc_c)$, this gives a skew monoidal structure on the homotopy category called the \emph{left derived skew monoidal structure} in \cite{Bourke2016Skew}.  Furthermore, the condition in Axiom MC concerning $[X,-]$ being right Quillen ensures that the left-derived skew monoidal structure is closed.  

Let us now verify Axiom MC for the left-sharp skew monoidal structure. 

\begin{prop}\label{mc1}
For each $\ca$, the functor $\psd_s(-,\ca) \colon \GCat^{op}_s \to \GCat$ preserve weak equivalences between cofibrant objects for each $\ca$.
\end{prop}
\begin{proof}
Let $F \colon \cb \to \cc$ be a weak equivalence between cofibrant objects.  By Proposition \ref{thm:char-equiv-in-GCat_s} it is an equivalence in $\GCat^4$.  Therefore, by Corollary \ref{cor:RepEquiv}, it is sent by the loose $4$-functor $\GCat^4(-,\ca) \colon (\GCat^4_s)^{op} \to \GCat^{4}_s:F \mapsto \psd_s(F,\ca)$ to an equivalence in $\GCat^4_s$. By \Cref{weak-equivalences}, each equivalence in $\GCat^4_s$ is a weak equivalence, completing the proof.
\end{proof}

It remains to show that if $\ca$ is cofibrant, then $\psd_s(\ca,-)$ is right Quillen: that is, preserves fibrations and trivial fibrations.  The latter is not difficult.

\begin{prop}\label{mc2}
$\psd_s(\ca,-)$ preserves trivial fibrations if $\ca$ is cofibrant.
\end{prop}
\begin{proof}
Let $F \colon \cb \to \cc$ be a trivial fibration and consider $\psd_s(\ca,F) \colon \psd_s(\ca,\cb) \to \psd_s(\ca,\cc)$.  Surjectivity on objects follows from the lifting property of the cofibration $\emptyset \to \ca$ against $F$.

Given a pseudo-transformation $\alpha \colon FX \to FY$, we need to lift it to a pseudo-transformation $\beta \colon X \to Y$ such that $F\beta = \alpha$.  Given $a \in \ca$, we have the component $\alpha_a$ on the right below

\[\begin{tikzcd}
	Xa & Ya && FXa & FYa
	\arrow["{\exists\beta_a}", from=1-1, to=1-2]
	\arrow["F", between={0.3}{0.7}, maps to, from=1-2, to=1-4]
	\arrow["{\alpha_a}", from=1-4, to=1-5]
\end{tikzcd}\]
and so use the fullness of $F$ on $1$-cells to define $\beta_a$ as any lifting of it along $F$, as depicted.  Similarly, given a non-identity $f \colon a \to b \in \ca$, we use fullness of $F$ to define $\beta_a$ as any lifting as depicted below. 
\[\begin{tikzcd}
	Xa & Ya && FXa & FYa \\
	Xb & Yb && FXb & FYb
	\arrow[""{name=0, anchor=center, inner sep=0}, "{\beta_a}", from=1-1, to=1-2]
	\arrow["Xf"', from=1-1, to=2-1]
	\arrow[""{name=1, anchor=center, inner sep=0}, "Yf", from=1-2, to=2-2]
	\arrow[""{name=2, anchor=center, inner sep=0}, "{\alpha_a}", from=1-4, to=1-5]
	\arrow[""{name=3, anchor=center, inner sep=0}, "FXf"', from=1-4, to=2-4]
	\arrow["FYf", from=1-5, to=2-5]
	\arrow[""{name=4, anchor=center, inner sep=0}, "{\beta_b}"', from=2-1, to=2-2]
	\arrow[""{name=5, anchor=center, inner sep=0}, "{\alpha_b}"', from=2-4, to=2-5]
	\arrow["{\exists \beta_f}", shift right, between={0.2}{0.8}, Rightarrow, from=0, to=4]
	\arrow["F", between={0.4}{0.6}, maps to, from=1, to=3]
	\arrow["{\alpha_f}", between={0.2}{0.8}, Rightarrow, from=2, to=5]
\end{tikzcd}\]
Similarly, we define $\beta_f^o$ as a lifting of $\alpha_f^o$ and, by fully faithfulness of $3$-cells, this data then extends uniquely to an adjoint equivalence $\beta_f^o \dashv \beta_f$ lifting the adjoint equivalence $\alpha_f^o \dashv \alpha_f$.

In order to enforce the necessary degeneracy condition for $\beta_f^o \dashv \beta_f$, in the case of an identity $1$-cell $1_a \colon a \to a$ we define $\beta_{1_a} = 1_{\beta_a}$ together with the structure of the identity adjoint equivalence.  Again this is preserved by $F$ since $\alpha$ satisfies this condition.

The remainder of the structure for $\beta$ to be a pseudo-transformation involves invertible $3$-cells $\beta^2_{f_2,f_1}$ and $\beta_\phi$.  Using that $F$ is fully faithful on $3$-cells, these are defined as the unique liftings of $\alpha^2_{f_2,f_1}$ and $\alpha_\phi$ along $F$.  Apart from the degeneracy condition, which we have already handled by construction, the remaining conditiions for $\beta$ to be a pseudo-transformation involve equalities of parallel $3$-cells (see \cite[Appendix A.2]{BourkeLobbia}).  These follow from the fact that $F$ is faithful on $3$-cells and the corresponding conditions for $\alpha$.

The proof that $F$ is full on $2$-cells and fully faithful on $3$-cells is similar, but easier.\end{proof}

In order to prove that $\psd_s(\ca,-)$ preserves fibrations, we will need a couple of intermediate results.

\begin{lemma}\label{internalgraph}
Let $\cb$ be cofibrant.  There exist internal graph maps
\[\begin{tikzcd}
	{\psd_s(\cb,P\cc)} & {P(\psd_s(\cb,\cc))} && {P(\psd_s(\cb,\cc))} & {\psd_s(\cb,P\cc)} \\
	{\psd_s(\cb,\cc)} & {\psd_s(\cb,\cc)} && {\psd_s(\cb,\cc)} & {\psd_s(\cb,\cc)}
	\arrow["U", squiggly, from=1-1, to=1-2]
	\arrow["{t_\star}", shift left=3, from=1-1, to=2-1]
	\arrow["{s_\star}"', from=1-1, to=2-1]
	\arrow["s"', from=1-2, to=2-2]
	\arrow["t", shift left=3, from=1-2, to=2-2]
	\arrow["V", squiggly, from=1-4, to=1-5]
	\arrow["t", shift left=3, from=1-4, to=2-4]
	\arrow["s"', from=1-4, to=2-4]
	\arrow["{s_{\star}}"', from=1-5, to=2-5]
	\arrow["{t_\star}", shift left=3, from=1-5, to=2-5]
	\arrow[equals, from=2-1, to=2-2]
	\arrow[equals, from=2-4, to=2-5]
\end{tikzcd}\]
in $\GCat_p$.
\end{lemma}
\begin{proof}
These are
\[\begin{tikzcd}[ampersand replacement=\&]
	{\psd_s(\cb,P\cc)} \& {\psd_p(\cb,P\cc)} \& {P(\psd_p(\cb,\cc))} \& {P(\psd_s(\cb,\cc))} \\
	{\psd_s(\cb,\cc)} \& {\psd_p(\cb,\cc)} \& {\psd_p(\cb,\cc)} \& {\psd_s(\cb,\cc)}
	\arrow["{J_{\cb,P\cc}}", from=1-1, to=1-2]
	\arrow["{t_\star}", shift left=3, from=1-1, to=2-1]
	\arrow["{s_\star}"', from=1-1, to=2-1]
	\arrow["{\phi_{\cb,\cc}}", squiggly, from=1-2, to=1-3]
	\arrow["{s_\star}"', from=1-2, to=2-2]
	\arrow["{t_\star}", shift left=3, from=1-2, to=2-2]
	\arrow["{PK_{\cb,\cc}}", from=1-3, to=1-4]
	\arrow["s"', from=1-3, to=2-3]
	\arrow["t", shift left=3, from=1-3, to=2-3]
	\arrow["s"', from=1-4, to=2-4]
	\arrow["t", shift left=3, from=1-4, to=2-4]
	\arrow["{J_{\cb,\cc}}"', from=2-1, to=2-2]
	\arrow[equals, from=2-2, to=2-3]
	\arrow["{K_{\cb,\cc}}"', from=2-3, to=2-4]
\end{tikzcd}\]
and
\[\begin{tikzcd}
	{P(\psd_s(\cb,\cc))} & {P(\psd_p(\cb,\cc))} & {\psd_p(\cb,P\cc)} & {\psd_s(\cb,P\cc)} \\
	{\psd_s(\cb,\cc)} & {\psd_p(\cb,\cc)} & {\psd_p(\cb,\cc)} & {\psd_s(\cb,\cc)}
	\arrow["{PJ_{\cb,\cc}}", from=1-1, to=1-2]
	\arrow["t", shift left=3, from=1-1, to=2-1]
	\arrow["s"', from=1-1, to=2-1]
	\arrow["{\phi^{-1}_{\cb,\cc}}", squiggly, from=1-2, to=1-3]
	\arrow["s"', from=1-2, to=2-2]
	\arrow["t", shift left=3, from=1-2, to=2-2]
	\arrow["{K_{\cb,P\cc}}", from=1-3, to=1-4]
	\arrow["{s_\star}"', from=1-3, to=2-3]
	\arrow["{t_\star}", shift left=3, from=1-3, to=2-3]
	\arrow["{s_\star}"', from=1-4, to=2-4]
	\arrow["{t_\star}", shift left=3, from=1-4, to=2-4]
	\arrow["{J_{\cb,\cc}}"', from=2-1, to=2-2]
	\arrow[equals, from=2-2, to=2-3]
	\arrow["{K_{\cb,\cc}}"', from=2-3, to=2-4]
\end{tikzcd}\]
In both diagrams, the central graph map is the isomorphism $\phi$ of \Cref{prop:P-and-psd} or its inverse.  In the other graph maps we use the Gray-enriched adjoint equivalence $J_{\cb,\cc} \dashv K_{\cb,\cc}$ of \Cref{prop:adjointEquiv} and its naturality in postcomposition by strict maps, as established there.  We also use that $s,t \colon P \rightrightarrows 1$ are natural as established in \Cref{lem:functorial}.  Finally, note that both composite lower horizontals are the identity since the adjoint equivalence is a retract equivalence.
\end{proof}

%
%
%
 \begin{defn}\label{def:homotopic-multimaps}
  Let $s,t \colon \ce \rightrightarrows \cb$ be a graph in $\GCat$.
\begin{itemize}
\item A \emph{loose homotopy} $H \colon F \sim_l G$ (rel $\ce$) is a pseudomap $H\colon \ca \rightsquigarrow \ce$ such that the diagram
\[\begin{tikzcd}[ampersand replacement=\&]
	\&\& \ce \\
	\ca \&\& \cb
	\arrow["s"', shift right=2, from=1-3, to=2-3]
	\arrow["t", shift left=2, from=1-3, to=2-3]
	\arrow["H", curve={height=-6pt}, squiggly, from=2-1, to=1-3]
	\arrow["F", shift left, squiggly, from=2-1, to=2-3]
	\arrow["G"', shift right, squiggly, from=2-1, to=2-3]
\end{tikzcd}\]
commutes.  If one such exists, we say that $F$ and $G$ are \emph{loosely homotopic}, written $F\sim_l G (rel \ce)$. 

\item If $F$ and $G$ are strict, a \emph{homotopy} $H \colon F \sim G$ (rel $\ce$) is a strict map $H$ as above making the diagram commute.  If one such exists, we say that $F$ and $G$ are homotopic rel $\ce$, written $F \sim G (rel \ce)$.
\end{itemize}
\end{defn}

 \begin{lemma}\label{relative}
 Let $\ca$ be $1$-free and consider strict maps $F,G \colon \ca \rightrightarrows \cb$.  Then $F \sim G (rel \ce) \iff F \sim_l G (rel \ce)$.
 \end{lemma}
 \begin{proof}
For the non-trivial direction, consider $H \colon \ca \rightsquigarrow \ce$ such that $s \circ H=F$ and $t \circ H=G$.  We take $K(H) \colon \ca \to \ce$ as per \Cref{thm:replacement} and by the same result, we then have $s \circ K(H) = K(s \circ H) = K(F) = F$ by naturality and since $s$ and $F$ are strict.  Similarly $t \circ K(H) = G$.
 \end{proof}

\begin{cor}\label{adjointhomotopy}
Let $\ca$ and $\cb$ be cofibrant.  Then $F,G \colon \ca \leftsharp \cb \rightrightarrows \cc$ are homotopic (rel $P\cc$) if and only if their transposes $\overline{F},\overline{G} \colon \ca \rightrightarrows \psd_s(\cb,\cc)$ are homotopic (rel $P(\psd_s(\cb,\cc))$).
\end{cor}
\begin{proof}
By adjointness, $F \sim G$ (rel $P\cc$) iff $\overline{F} \sim \overline{G}$ (rel $\psd_s(\cb,P\cc)$).  Since $\ca$ is cofibrant, by Lemma \ref{relative}, this is equally to say that $\overline{F} \sim_l \overline{G}$ (rel $\psd_s(\cb,P\cc))$.  Now since $\cb$ is cofibrant, by \Cref{internalgraph}, this is equivalently to say that  $\overline{F} \sim_l \overline{G}$ (rel $P(\psd_s(\cb,\cc))$).  Finally, again using that $\ca$ is cofibrant and Lemma \ref{relative}, this is to say that $\overline{F} \sim \overline{G}$ (rel $P(\psd_s(\cb,\cc))$).
\end{proof}

\begin{prop}\label{mc3}
$\psd_s(\ca,-)$ preserves fibrations whenever $\ca$ is cofibrant.
\end{prop}
\begin{proof}
Firstly observe that the right adjoint $\psd_s(\ca,-)$ preserves fibrations just when its left adjoint $- \leftsharp \ca$ preserves trivial cofibrations.  Now, as a left adjoint, it will preserve trivial cofibrations just when it sends the generating trivial cofibrations to trivial cofibrations.  Since the generating trivial cofibrations, described in Section \ref{sect:Graymodel} have cofibrant domain and codomain, it suffices therefore to show that $- \leftsharp \ca$ sends trivial cofibrations between cofibrant objects to trivial cofibrations.  Now since the right adjoint $\psd_s(\ca,-)$ preserves trivial fibrations, certainly the left adjoint $- \leftsharp \ca$ preserves cofibrations.  Therefore, it suffices to prove $- \leftsharp \ca$ sends trivial cofibrations between cofibrant objects to weak equivalences.  This is what we will prove.

Let $F \colon \cb \to \cc$ be a trivial cofibration between cofibrant objects.  Certainly $F \leftsharp \ca \colon \cb \leftsharp \ca \to \cc \leftsharp \ca$ has cofibrant domain and codomain since $- \leftsharp \ca$ preserves cofibrations (and hence cofibrant objects).  Since all objects are fibrant, it will be a weak equivalence if and only is a homotopy equivalence.  This is equally to say that for all $\cd$, the induced function on the top horizontal below is a bijection 


\[\begin{tikzcd}
	{\GCat(\cc \leftsharp \ca,\cd)/\sim} &&&& {\GCat(\cb \leftsharp \ca,\cd)/\sim} \\
	{\GCat(\cc,\psd_s(\ca,\cd))/\sim} &&&& {\GCat(\cb, \psd_s(\ca,\cd))/\sim}
	\arrow["{\GCat(F\leftsharp \ca,\cd)/\sim}", from=1-1, to=1-5]
	\arrow["\cong"', from=1-1, to=2-1]
	\arrow["\cong", from=1-5, to=2-5]
	\arrow["{\GCat(F,\psd_s(\ca,\cd))/\sim}"', from=2-1, to=2-5]
\end{tikzcd}\]
where homotopy is taken relative to $P\cd$.  By \Cref{adjointhomotopy}, the adjoint transpose functions induce bijections on the verticals.  Therefore the top horizontal will be a bijection just when the lower one is, now for homotopy rel $P(\psd_s(\ca,\cd))$.  But by standard model categorical results --- see \cite[Proposition 1.2.5]{Hovey} --- this is true since $F$ is a trivial cofibration between cofibrant objects.
\end{proof}

\begin{theorem}\label{axiommc}
The left-sharp skew monoidal structure satisfies Axiom MC.
\end{theorem}
\begin{proof}
The unit $1$ is cofibrant.  \Cref{mc2} and \Cref{mc3} together establish that $\psd_s(\ca,-)$ is right Quillen for cofibrant $\ca$.  Combined with \Cref{mc1}, this verifies Axiom MC.
\end{proof}

We now give a direct description of the left derived skew monoidal structure on $\Ho(\GCat)$.  For our purposes, the convenient description of $\Ho(\GCat)$, up to equivalence, is the following one.

\begin{defn/propn}\label{explicit1}
The homotopy category $\Ho(\GCat)$ can be constructed as the category whose objects are the cofibrant Gray-categories and whose morphisms $\ca \to \cb$ are homotopy classes of Gray-functors, relative to $P\cb$.
\end{defn/propn}
\begin{proof}
It is a basic result of model categories that the homotopy category $\Ho(\cc)$ of a model category $\cc$ can be constructed by taking the bifibrant objects and homotopy classes of morphisms between them, relative to any path object (see e.g. \cite[Theorem 1.2.10]{Hovey}).  Since all objects in $\GCat$ are fibrant and $P\cb$ is a path object, this establishes the claim.
\end{proof}

In what follows, we shall work with the above description of $\Ho(\GCat)$.  Let us now describe the derived skew monoidal structure on $\Ho(\GCat)$ in these terms.  (This is essentially just an instance of \cite[Theorem 5.11]{Bourke2016Skew} but since we are using the equivalent, though notationally different, description of the homotopy category using cofibrant objects, we find it simpler to give a direct proof.)  

\begin{theorem}\label{explicit2}
The closed skew monoidal structure $(\GCat,\leftsharp,1)$ induces a closed skew monoidal structure $(\Ho(\GCat),\leftsharp,1)$ on the homotopy category.  Explicitly:
\begin{itemize}
\item the tensor product $\leftsharp \colon \Ho(\GCat)^2 \to \Ho(\GCat)$ acts as $\leftsharp$ on objects and on homotopy classes of maps by $[f] \leftsharp [g] = [f \leftsharp g]$;
\item the associators and unitors for $(\Ho(\GCat),\leftsharp,1)$ are the homotopy classes of the associators and unitors for $(\GCat,\leftsharp,1)$;
\item the internal hom in $\Ho(\GCat)$ is the cofibrant replacement $C(\psd_s(\cb,\cc))$ of the internal hom $\psd_s(\cb,\cc)$.
\end{itemize}
\end{theorem}
\begin{proof}
By Axiom MC, $\leftsharp$ preserves cofibrant objects and weak equivalences between them. Therefore it restricts to a skew monoidal structure $(\GCat_c,\leftsharp,1)$ on the full subcategory $\GCat_c$ of cofibrant objects.  Since the restricted tensor product $\leftsharp \colon{ \GCat_c}^2 \to \GCat_c$ preserves \emph{all} weak equivalences, by the universal property of the localization, it induces a skew monoidal structure $(Ho(\GCat_c),\leftsharp,1)$ on the homotopy category $\Ho(\GCat_c)$, obtained by inverting the weak equivalences in $\GCat_c$.  Now since all objects are fibrant in $\GCat$, $\Ho\GCat_c)$ is exactly $\Ho(\GCat)$ as described in \Cref{explicit1}, and the description of the induced skew monoidal structure as claimed.

With regards the internal hom, by \Cref{adjointhomotopy}, given $\ca,\cb$ cofibrant, the closedness bijections $\phi_{\ca,\cb,\cc} \colon \GCat(\ca \leftsharp \cb,\cc) \to \GCat(\ca, \psd_s(\cb,\cc))$ induce natural bijections $$\phi \colon \GCat(\ca \leftsharp \cb,\cc) / \sim \to \GCat(\ca, \psd_s(\cb,\cc)) / \sim$$ on homotopy classes of maps rel $P$.  Since $\psd_s(\cb,\cc)$ is not necessarily cofibrant, we take its cofibrant replacement $T \colon C(\psd_s(\cb,\cc)) \to \psd_s(\cb,\cc)$.  Now since $\ca$ is cofibrant and $T$ a weak equivalence between fibrant objects, by \cite[Proposition 1.2.5]{Hovey}, the natural transformation $$T_{\star} \colon \GCat(\ca, C(\psd_s(\cb,\cc)) / \sim \longrightarrow \GCat(\ca, \psd_s(\cb,\cc))/ \sim$$ is a natural bijection.  The composite natural bijection $$(T_{\star})^{-1} \circ \phi_{\ca,\cb,\cc} \colon \GCat(\ca \leftsharp \cb,\cc) / \sim \longrightarrow \GCat(\ca, C(\psd_s(\cb,\cc))) / \sim$$ verifies the claim.
\end{proof}

In \Cref{thm:smchomotopy} we will prove that the derived skew monoidal structure on $\Ho(\GCat)$ is in fact symmetric monoidal closed.

\subsection{The homotopy multicategory}\label{sect:homotopymulticategory}

The left-sharp skew monoidal structure $(\GCat,\leftsharp,1)$ is obtained as representing the multicategory $\GCatmpshl$.  In order to understand the derived skew monoidal structure, it will be easiest to first understand the associated homotopy multicategory, which is what we do in the present section.  

\begin{defn}\label{def:homotopic-multimaps}
Let $F,G \in \GCatmp^l_n(\overline{\ca}; \cb)$. 
\begin{itemize}
\item A \emph{loose homotopy} $H \colon F \sim_l G$ is a loose multimap $H\colon\overline{\ca}\rightsquigarrow P\cb$ in $\GCatmp^l$ such that the diagram
\[\begin{tikzcd}
	&& {P\cb} \\
	{\overline{\ca}} && \cb
	\arrow["s"', shift right=2, from=1-3, to=2-3]
	\arrow["t", shift left=2, from=1-3, to=2-3]
	\arrow["H", curve={height=-6pt}, squiggly, from=2-1, to=1-3]
	\arrow["F", shift left, squiggly, from=2-1, to=2-3]
	\arrow["G"', shift right, squiggly, from=2-1, to=2-3]
\end{tikzcd}\]
commutes.  If one such exists, we say that $F$ and $G$ are \emph{loosely homotopic}, written $F\sim_l G$. 

\item If $F$ and $G$ are left-sharp, a \emph{left-sharp homotopy} $H \colon F \sim_\sharp G$ is a left-sharp multimap $H$ as above making the diagram commute.  If one such exists, we say that $F$ and $G$ are left-sharp homotopic, written $F \sim_\sharp G$.
\end{itemize}
\end{defn}

\begin{lemma}\label{lem:biclosed-resp-hom-class}
The left and right closedness bijections $\GCatmp^l_{n+1}(\overline{\ca},\cb;\cc)\cong\GCatmp^l_{n}(\overline{\ca};\psd_p(\cb,\cc))$ and $\GCatmp^l_{n+1}(\ca,\overline{\cb};\cc)\cong\GCatmp^l_{n}(\overline{\cb};\psd_p(\ca,\cc))$ restrict to loose homotopy classes, as do the symmetry bijections.
\end{lemma}

\begin{proof}
We start by showing that left closedness respect homotopy classes.
By definition, two multimaps $F,G \in \GCatmp^l_{n+1}(\overline{\ca},\cb;\cc)$ are loosely homotopic if and only if there exists a diagonal filler $H_1$ as in the diagram below left.
\[\begin{tikzcd}[ampersand replacement=\&]
	\&\& {P\cc} \&\&\&\& {\psd_p(\cb,P\cc)} \\
	{\overline{\ca},\cb} \&\& \cc \&\& {\overline{\ca}} \&\& {\psd_p(\cb,\cc)}
	\arrow["s"', shift right=2, from=1-3, to=2-3]
	\arrow["t", shift left=2, from=1-3, to=2-3]
	\arrow["{s_\star}"', shift right=2, from=1-7, to=2-7]
	\arrow["{t_\star}", shift left=2, from=1-7, to=2-7]
	\arrow["{\exists H_1}", squiggly, from=2-1, to=1-3]
	\arrow["F", shift left, squiggly, from=2-1, to=2-3]
	\arrow["G"', shift right, squiggly, from=2-1, to=2-3]
	\arrow["{\exists H_2}", squiggly, from=2-5, to=1-7]
	\arrow["{\overline{F}}"{pos=0.7}, shift left, squiggly, from=2-5, to=2-7]
	\arrow["{\overline{G}}"'{pos=0.7}, shift right, squiggly, from=2-5, to=2-7]
\end{tikzcd}\]
By naturality of the left closedness bijections, such corresponds to a diagonal filler $H_2$ for the diagram above right.  Now pasting with the graph isomorphism of \Cref{prop:P-and-psd}
\[\begin{tikzcd}
	{\psd_p(\cb,P\cc)} & {P(\psd_p(\cb,\cc))} &&& {P(\psd_p(\cb,\cc))} \\
	{\psd_p(\cb,\cc)} & {\psd_p(\cb,\cc)} && {\overline{\ca}} & {\psd_p(\cb,\cc)}
	\arrow["{\phi_{\cb,\cc}}", squiggly, from=1-1, to=1-2]
	\arrow["{s_\star}"', shift right=2, from=1-1, to=2-1]
	\arrow["{t_\star}", shift left=2, from=1-1, to=2-1]
	\arrow["s"', from=1-2, to=2-2]
	\arrow["t", shift left=3, from=1-2, to=2-2]
	\arrow["s"', from=1-5, to=2-5]
	\arrow["t", shift left=3, from=1-5, to=2-5]
	\arrow[equals, from=2-1, to=2-2]
	\arrow["{\exists H_3}", squiggly, from=2-4, to=1-5]
	\arrow["{\overline{F}}"{pos=0.7}, squiggly, from=2-4, to=2-5]
	\arrow["{\overline{G}}"'{pos=0.7}, shift right=3, squiggly, from=2-4, to=2-5]
\end{tikzcd}\]
this induces a bijection between fillers $H_2$ and fillers $H_3$ for the diagram above right.  Thus $F\sim_l G$ if and only if $\overline{F}\sim_l \overline{G}$.  The case of the right closedness bijections is entirely analogous.  That the symmetry bijections restrict to loose homotopy classes follows trivially from their naturality (in the codomain).
\end{proof}

\begin{lemma}\label{prop:equiv-rel}
The loose homotopy relation $\sim_l$ is an equivalence relation on each set $\GCatmp^l_{n}(\overline{\ca};\cb)$.
\end{lemma}
\begin{proof}
Using that the closedness bijections respect loose homotopy classes, and that each $n$-ary map corresponds to a nullary one by iterated application of the closedness bijections, it suffices to verify the claim when $n=0$.  Now nullary maps $\diamond \rightsquigarrow \cb$ are just objects of $\cb$ and, since the objects of $P\cb$ are $1$-cell equivalences, two such objects $a,b \in \cb$ are loosely homotopic as nullary maps just when there exists a $1$-cell equivalence $a \to b$.  Since $1$-cell equivalences are composable, contain the identities and each has an equivalence inverse, this yields an equivalence relation.
\end{proof}

\begin{lemma}\label{thm:multicat-Homotopy}
Substitutions in $\GCatmp^l$ respects the loose homotopy relation.  
\end{lemma}
\begin{proof}
Consider the composite multimap $G\circ_i F$. We must show that if $F\sim_l F'$ and $G\sim_l G'$, then $G\circ_i F\sim_l G'\circ_i F'$.  Since $\sim_l$ is an equivalence relation (see \Cref{prop:equiv-rel}), it suffices to prove $G\circ_i F\sim_l G'\circ_i F$ and $G\circ_i F\sim_l G\circ_i F'$. 

For the former:  if $H \colon G \sim_l G'$ then the composite multimap $H \circ_i F$ is the desired loose homotopy  $G\circ_i F \sim_l G\circ_i F'$.

We prove that $G\circ_i F\sim_l G\circ_i F'$ by induction on the arity of $G$.  If $G$ is unary, we use that $P$ is functorial --- see \Cref{lem:functorial} --- to take the composite loose homotopy depicted below. 
\[\begin{tikzcd}[ampersand replacement=\&]
	\&\& {P\cb} \&\& {P\cc} \\
	{\overline{\ca}} \&\& \cb \&\& \cc
	\arrow["PG", squiggly, from=1-3, to=1-5]
	\arrow["s"', from=1-3, to=2-3]
	\arrow["t", shift left=3, from=1-3, to=2-3]
	\arrow["s"', from=1-5, to=2-5]
	\arrow["t", shift left=3, from=1-5, to=2-5]
	\arrow["H", shift left=2, squiggly, from=2-1, to=1-3]
	\arrow["F"'{pos=0.7}, shift right, squiggly, from=2-1, to=2-3]
	\arrow["{F'}"{pos=0.7}, shift left, squiggly, from=2-1, to=2-3]
	\arrow["G"', squiggly, from=2-3, to=2-5]
\end{tikzcd}\]

For the inductive step, suppose that $G \colon (\ca_1,\ldots, \ca_{n+1}) \to \cb$ has arity $n+1$.  Suppose that $i = 1$.  Then by left closedness, we consider the adjoint map $\overline{G} \colon (\ca_1,\ldots, \ca_{n}) \to \psd_p(\ca_{n+1},\cb)$.  By induction, we have $\overline{G} \circ_1 F \sim_l \overline{G} \circ_1 F'$.  But since by \Cref{lem:biclosed-resp-hom-class} the right closedness bijections respect loose homotopy, we then have $\overline{\overline{G} \circ_1 F} \sim_l \overline{\overline{G} \circ_1 F'}$.  By naturality of the closedness bijections, this is the desired loose homotopy $G \circ_1 F \sim G \circ_1 F'$.  If $i  > 1$, we can argue in an analogous fashion but using right closedness.

\end{proof}

Combining the three preceding results we immediately obtain:

\begin{theorem}\label{thm:closedhomotopy}
There is a symmetric closed multicategory $\Ho(\GCatmp^l)$ whose objects are Gray-categories and with $\Ho(\GCatmp^l)(\overline{\ca};\cb) := \GCatmp^l_{n}(\overline{\ca};\cb) / \sim_l$.  The substitution operation of $\Ho(\GCatmp^l)$ is that of $\GCatmp^l$, applied to loose homotopy classes: that is, $[F] \circ_i [G] := [F \circ_i G]$.  
\end{theorem}


We now turn our attention to the full sub-multicategory $$\Ho(\GCatmp^l_c) \hookrightarrow \Ho(\GCatmp^l)$$ of $\Ho(\GCatmp^l)$ containing the cofibrant objects.  It inherits its symmetry from $\Ho(\GCatmp^l)$ and we would like to show it is also closed.  For this we shall need the following lemma.

\begin{lemma}\label{lem:input-cofib-bij}
Let $\overline{\ca}$ be pointwise cofibrant and $F\colon\cb\to\cc$ be a weak equivalence in $\GCat$.  Then the induced function 
$$[F] \circ-\colon \Ho(\GCatmp^l)_n(\overline{\ca};\cb)\to \Ho(\GCatmp^l)_n(\overline{\ca};\cc)$$
is a bijection. 
\end{lemma}
\begin{proof}
For $n=0$, $\Ho(\GCatmp^l)(\diamond;\cb) = \pi_0(\cb)$, the set of equivalence classes of objects of $\cb$, modulo $1$-cell equivalence in $\cb$.  Therefore we must show that the induced function $\pi_0(F)  \colon \pi_0(\cb) \to \pi_0(\cc)$ is a bijection.  But as $F$ is a weak equivalence of Gray-categories (a triequivalence) this is routine.  

For $n \geq 1$, we argue by induction.  By left closedness of $\Ho(\GCatmp)^l$ as above, we have the commutative diagram below.
\[\begin{tikzcd}[ampersand replacement=\&]
	{Ho(\GCatmp^l)_n(\overline{\ca},\cx;\cb)} \&\&\& {Ho(\GCatmp^l)_n(\overline{\ca},\cx;\cc)} \\
	{Ho(\GCatmp^l)_{n-1}(\,\overline{\ca};\psd_p(\cx,\cb)\,)} \&\&\& {Ho(\GCatmp^l)_{n-1}(\,\overline{\ca};\psd_p(\cx,\cc)\,)}
	\arrow["{[F]\circ-}", from=1-1, to=1-4]
	\arrow["\cong"{marking, allow upside down}, draw=none, from=1-1, to=2-1]
	\arrow["\cong"{marking, allow upside down}, draw=none, from=1-4, to=2-4]
	\arrow["{[\psd_p(\cx,F)]\circ-}"', from=2-1, to=2-4]
\end{tikzcd}\]
Therefore, by induction, it suffices to prove that, for $\cx$ cofibrant and $F$ a weak equivalence, then $\psd_p(\cx,F)\colon\psd_p(\cx,\cb)\to\psd_p(\cx,\cc)$ is a weak equivalence.  
To prove this, consider the commuting square 

\[\begin{tikzcd}
	{\psd_s(\cx,\cb)} && {\psd_s(\cx,\cc)} \\
	{\psd_p(\cx,\cb)} && {\psd_p(\cx,\cc)}
	\arrow["{\psd_s(\cx,F)}", from=1-1, to=1-3]
	\arrow[hook, from=1-1, to=2-1]
	\arrow[hook, from=1-3, to=2-3]
	\arrow["{\psd_p(\cx,F)}"', from=2-1, to=2-3]
\end{tikzcd}\]

The vertical morphisms are Gray-enriched equivalences by \Cref{prop:adjointEquiv} and, in particular, weak equivalences of Gray-categories.   By 2-out-of-3, it therefore suffices to prove that the top horizontal is a weak equivalence.  Now, by \Cref{axiommc}, $\psd_s(\cx,-)$ is right Quillen.  Therefore by Ken Brown's lemma, \cite[ Lemma 1.1.12]{Hovey}, it preserves weak equivalences between fibrant objects.  Since all objects are fibrant, $\psd_s(\cx,F)$ is a weak equivalence. 
\end{proof}

\begin{prop}\label{thm:hom-cofib-loose-closed}
The symmetric multicategory $\Ho(\GCatmp_c^l)$ is biclosed. %
\end{prop}
\begin{proof}
Since it is symmetric, it suffices to show it is left closed.  The internal hom is given by the cofibrant replacement $C(\psd_s(\cb,\cc))$.  We consider the composite weak equivalence $J \circ T \colon C(\psd_s(\cb,\cc))\to\psd_s(\cb,\cc)\hookrightarrow \psd_p(\cb,\cc)$ whose first component is the weak equivalence $T$ for the cofibrant replacement and whose second component, the inclusion, is a Gray-enriched equivalence by \Cref{prop:adjointEquiv} since $\cb$ is cofibrant.  Now let  $e^{l} \colon \psd_p(\cb,\cc),\cb \rightsquigarrow \cc$ be the representing multimap for the closed structure on $\GCatmp^l$.  Then, we define the binary map which will provide the closedness bijection as the loose homotopy class $[e^l\circ_1 (J \circ T)] \colon C(\psd_s(\cb,\cc)), \cb \rightsquigarrow \cc$.
It remains to show that the composite function
\[\begin{tikzcd}[ampersand replacement=\&]
	{Ho(\GCatmp^l)_n(\,\overline{A};C(\psd_s(\cb,\cc)))} \& {Ho(\GCatmp^l)_n(\,\overline{A};\psd_p(\cb,\cc)\,)} \& {Ho(\GCatmp^l)_n(\overline{A},\cb;\cc)}
	\arrow["{[(J \circ T)] \circ-}", from=1-1, to=1-2]
	\arrow["{[e^l]\circ_1-}", from=1-2, to=1-3]
\end{tikzcd}\]
 is a bijection.  \Cref{lem:input-cofib-bij} above ensures that $[(J \circ T)] \circ -$ is a bijection whilst the second map is a bijection since $\Ho(\GCatmp^l)$ is closed (\Cref{thm:closedhomotopy}). 
\end{proof}

The following lemma shows that, given the restriction on objects in $\Ho(\GCatmp_c^l)$ to cofibrant ones, it suffices to consider left-sharp multimaps (up to left-sharp homotopy).

\begin{lemma}\label{sharpvsloose}
Let $\overline{\ca}$ be pointwise cofibrant.
\begin{enumerate}
\item Two left-sharp multimaps $F,G \colon \overline{\ca} \rightrightarrows \cb$ in $\GCatmp^l$ are loosely homotopic if and only if they are left-sharp homotopic. 
\item Each loose multimap  $F \colon \overline{\ca} \rightsquigarrow \cb$ in $\GCatmp^l$ is loosely homotopic to a left-sharp one. 
\end{enumerate}
\end{lemma}
\begin{proof}

We prove (1) by induction.   For nullary multimaps it is trivial since left-sharp and loose nullary multimaps coincide.   

For $n=1$, consider $F,G \colon \ca \rightrightarrows \cb$.  Since $\ca$ is cofibrant, and hence $1$-free, it is a special case of \Cref{relative} where $\ce = P\cb$.
%
%

 For $n\geq 2$, consider left-sharp $n$-ary multimaps $F,G\colon\overline{\ca},\cb  \rightrightarrows \cc$ and suppose that $F\sim_l G$.  By \Cref{lem:biclosed-resp-hom-class} the two adjoint maps $\overline{F}, \overline{G}\colon \overline{\ca}\rightrightarrows\psd_p(\cb,\cc)$ are loosely homotopic.  Therefore, by induction, there is a left-sharp homotopy $H \colon \overline{F} \sim_\sharp \overline{G}$. Now we define $H^\ast:=J_{\cb,P\cc} \circ K_{\cb, P\cc} \circ H$ as depicted on the top row below
\[\begin{tikzcd}[ampersand replacement=\&]
	\&\& {\psd_p(\cb,P\cc)} \& {\psd_s(\cb,P\cc)} \& {\psd_p(\cb,P\cc)} \\
	{\overline{\ca}} \&\& {\psd_p(\cb,\cc)} \& {\psd_s(\cb,\cc)} \& {\psd_p(\cb,\cc)}
	\arrow["{K_{\cb,P\cc}}", from=1-3, to=1-4]
	\arrow["{s_\ast}"', shift right=2, from=1-3, to=2-3]
	\arrow["{t_\ast}", shift left=2, from=1-3, to=2-3]
	\arrow["{J_{\cb,P\cc}}", hook, from=1-4, to=1-5]
	\arrow["{t_\ast}", shift left=2, from=1-4, to=2-4]
	\arrow["{s_\ast}"', shift right=2, from=1-4, to=2-4]
	\arrow["{t_\ast}", shift left=2, from=1-5, to=2-5]
	\arrow["{s_\ast}"', shift right=2, from=1-5, to=2-5]
	\arrow["H", curve={height=-12pt}, from=2-1, to=1-3]
	\arrow["{\overline{G}}"', shift right=2, from=2-1, to=2-3]
	\arrow["{\overline{F}}", shift left=2, from=2-1, to=2-3]
	\arrow["{K_{\cb,\cc}}"', from=2-3, to=2-4]
	\arrow["{J_{\cb,\cc}}"', hook, from=2-4, to=2-5]
\end{tikzcd}\]
where we use that $\cb$ is cofibrant and \Cref{prop:adjointEquiv} to form the enriched adjoint equivalences on the horizontals.  The diagram is serially commutative because, as proven in \Cref{prop:adjointEquiv}, the adjoint equivalences are natural with respect to postcomposition by strict maps.  Furthermore, as $F$ and $G$ are left-sharp, by \Cref{sharploose} both $\overline{F}$ and $\overline{G}$ factor through $J_{\cb,\cc} \colon \psd_s(\cb,\cc) \hookrightarrow \psd_s(\cb,\cc)_p$.  Therefore since $ K_{\cb, \cc} \circ  J_{\cb,\cc} = Id$, the two composites on the bottom row are just $\overline{F}$ and $\overline{G}$.  Furthermore, since $H^\ast$ factors through $J_{\cb,P\cc}$ via a left-sharp multimap, by \Cref{sharploose} again, its adjoint $\overline{H^\ast}$ is a left-sharp homotopy between $F$ and $G$. 

Turning to (2), we again prove it by induction on $n$.  The case $n=0$ is trivial as before.  For $n=1$, consider a pseudomap $F\colon\ca \rightsquigarrow \cb$ with $\ca$ cofibrant. Then, by \Cref{thm:replacement}, there exists a strict map $K(F)$ and an invertible icon $\rho_F\colon K(F) \cong F$.  By \Cref{prop:p-maps-into-PA}, this provides the required loose homotopy as depicted below.
\[\begin{tikzcd}[ampersand replacement=\&]
	\&\& {P\cb} \\
	\ca \&\& \cb
	\arrow[shift right=2, from=1-3, to=2-3]
	\arrow[shift left=2, from=1-3, to=2-3]
	\arrow["{\rho_F}", squiggly, from=2-1, to=1-3]
	\arrow["F"'{pos=0.6}, shift right, squiggly, from=2-1, to=2-3]
	\arrow["{K(F)}"{pos=0.6}, shift left, from=2-1, to=2-3]
\end{tikzcd}\]
 For $n\geq 2$, consider a loose $n$-ary multimap $F\colon\overline{\ca},\cb \rightsquigarrow \cc$ and its adjoint transpose $\overline{F} \colon\overline{\ca} \rightsquigarrow \psd_p(\cb,\cc)$  Since $\overline{\ca}$ is pointwise cofibrant, by induction there exists a left-sharp multimap $G\colon\overline{A}\to\psd(\cb,\cc)$ and loose homotopy $G \sim \overline{F}$.  Since $\cb$ is cofibrant, by \Cref{prop:adjointEquiv}  the inclusion $J\colon \psd_s(\cb,\cc) \hookrightarrow \psd_p(\cb,\cc)$ forms part of an enriched equivalence $J \dashv K$ with counit $\rho \colon J \circ K \cong Id$.  As before, this specifies a loose homotopy $J \circ K \sim_l Id$.  Since by \Cref{thm:multicat-Homotopy} substitution respects loose homotopy, we obtain a composite loose homotopy

\[\begin{tikzcd}[ampersand replacement=\&]
	{\overline{\ca}} \&\& {\psd_p(\cb,\cc)} \&\& {\psd_p(\cb,\cc)}
	\arrow[""{name=0, anchor=center, inner sep=0}, "G", curve={height=-18pt}, from=1-1, to=1-3]
	\arrow[""{name=1, anchor=center, inner sep=0}, "{\overline{F}}"', curve={height=18pt}, squiggly, from=1-1, to=1-3]
	\arrow[""{name=2, anchor=center, inner sep=0}, "Id"', curve={height=24pt}, from=1-3, to=1-5]
	\arrow[""{name=3, anchor=center, inner sep=0}, "{J \circ K}", curve={height=-24pt}, from=1-3, to=1-5]
	\arrow["{\sim_l}", shift right=2, draw=none, from=0, to=1]
	\arrow["{\sim_l}", shift right=3, draw=none, from=3, to=2]
\end{tikzcd}\]
Now by \Cref{{lem:biclosed-resp-hom-class}} the closedness bijections respect loose homotopy.   Therefore it remains to show that the adjoint multimap to $J \circ K \circ G \colon \colon\overline{\ca} \to \psd_p(\cb,\cc)$ is left-sharp, which holds again by \Cref{sharploose} since $J \circ K \circ G$ factors through $J$ via the left-sharp multimap $K \circ G$.

\end{proof}

The multicategory of prime interest to us is the following one.  That it is well-defined as a multicategory follows from the fact that $\GCatmpshl \hookrightarrow \GCatmp^l$ is a sub-multicategory and Part 1 of \Cref{sharpvsloose}.

\begin{defn}\label{leftsharphomulticat}
$\Ho(\GCatmpshl) \hookrightarrow \Ho(\GCatmp_c^l)$ is the sub-multicategory of $\Ho(\GCatmp_c^l)$ with the same objects and whose multimaps are left-sharp multimaps (modulo left-sharp homotopy).
\end{defn} 

By Part 2 of \Cref{sharpvsloose} we immediately obtain:

\begin{prop}
The inclusion $\Ho(\GCatmpshl) \hookrightarrow \Ho(\GCatmp_c^l)$ is an isomorphism of multicategories.
\end{prop}

\begin{rmk}\label{strictification}
We expect the above result to be useful in strictifying loose $4$-categories to left-sharp ones.  Indeed, if $\ca$ is a loose $4$-category whose hom-objects are cofibrant Gray-categories, then it says that the loose composition bimap $\circ \colon \ca(y,z),\ca(x,y) \rightsquigarrow \ca(x,z)$ is loosely homotopic to a left-sharp one $\ca(y,z),\ca(x,y) \to \ca(x,z)$, as required in a left-sharp $4$-category.
\end{rmk}

Finally, we can give the main theorem of the section. 

\begin{theorem}\label{thm:cofibrantmulticat}
The multicategory $\Ho(\GCatmpshl)$ has the following properties.
\begin{enumerate}
\item Its category of unary maps is $\Ho(\GCat)$;
\item It is a symmetric, closed and representable multicategory.
\end{enumerate}
\end{theorem}
\begin{proof}
By construction, its unary category has cofibrant Gray-categories as objects whilst morphisms are strict morphisms of Gray-categories (i.e. left-sharp unary maps) modulo the left-sharp homotopy relation.  This is precisely $\Ho(\GCat)$ as defined in \Cref{explicit1}. 

Let us turn to (2).  Since $\Ho(\GCatmp_c^l)$ is symmetric closed, the symmetric closed structure restricts along the isomorphism $\Ho(\GCatmpshl) \hookrightarrow \Ho(\GCatmp_c^l)$ of multicategories.  With regards representabillity, it suffices by \cite[Proposition 4.7]{BourkeLack:skew} to prove that the multicategory has a nullary map classifier and that the internal hom functor on unary maps has a left adjoint.  Clearly $1$ is a nullary map classifier for $\Ho(\GCatmp_c^l)$ whilst the internal hom functor on the category of unary maps is $C(\psd_s(\cb,-)) \colon \Ho(\GCat) \to \Ho(\GCat)$.  This has left adjoint $- \leftsharp \cb \colon \Ho(\GCat) \to \Ho(\GCat)$ by \Cref{explicit2}.
\end{proof}

\subsection{The homotopy category $\Ho(\GCat)$ is symmetric monoidal closed}
We are finally in a position to show that the derived skew monoidal structure $(\Ho(\GCat),\leftsharp,1)$ of \Cref{explicit2} is symmetric monoidal closed.  Our argument is as follows.
\begin{enumerate}
\item The construction of the derived skew monoidal structure $(\Ho(\GCat),\leftsharp,1)$ of \Cref{explicit2} had ultimately two steps.  We started, in \cite{BourkeLobbiaCorrected}, with the multicategory $\GCatmpshl$ and constructed $(\GCat,\otimes_\sharp,1)$ as the representing skew monoidal category.  In \Cref{explicit2} we showed that it induces a skew monoidal structure $(\Ho(\GCat),\leftsharp,1)$ on the homotopy category.
\item On the other hand, in \Cref{thm:cofibrantmulticat}, we constructed a multicategory $\Ho(\GCatmpshl)$ with underlying category the homotopy category $\Ho(\GCat)$.  Since we have proved that it is representable and symmetric, it therefore induces a symmetric monoidal category structure on $\Ho(\GCat)$ (see \Cref{representable} for more details.)
\item Accordingly, if we can show that the above two constructions coincide: that is, if we can show that $(\Ho(\GCat),\leftsharp,1)$ is the monoidal structure representing the representable multicategory $\Ho(\GCatmpshl)$, we will have proven the following theorem.
\end{enumerate}

%

\begin{theorem}\label{thm:smchomotopy}
The left derived skew monoidal structure $(\Ho(\GCat),\leftsharp,1)$ is symmetric monoidal closed.
\end{theorem}

\begin{proof}
By the above, it suffices show that $(\Ho(\GCat),\leftsharp,1)$ is the monoidal structure associated to the representable multicategory $\Ho(\GCatmpshl)$.  In the proof, we will freely use that substitution in $\GCatmpshl$ respects the left-sharp homotopy relation, which we proved in showing that $\Ho(\GCatmpshl)$ is a well-defined multicategory as per \Cref{leftsharphomulticat}.
\begin{enumerate}[leftmargin=*]
\item Firstly, we must show that the tensor product $\ca\leftsharp \cb$ for $(\Ho(\GCat),\leftsharp,1)$ classifies binary maps in the multicategory $\Ho(\GCatmpshl)$.  To this end, let $\ca$ and $\cb$ be cofibrant and consider the left-sharp binary map classifier $\theta_{\ca,\cb}\colon\ca,\cb\to\ca\leftsharp \cb$ defining $\ca\leftsharp \cb$.  It has the universal property that the induced function $$\theta_{\ca,\cb} \circ_1 - \colon \GCatmpshl_1(\ca\leftsharp \cb,\cc) \to \GCatmpshl_2(\ca,\cb;\cc)$$ is a bijection.  We must show that its left-sharp homotopy class $[\theta_{\ca,\cb}] \colon \ca,\cb\to\ca\leftsharp \cb \in \Ho(\GCatmpshl)$ has the corresponding universal property: namely, that the induced function 
$$[\theta_{\ca,\cb}] \circ_1 - \colon \Ho(\GCatmpshl)_1(\ca\leftsharp \cb,\cc) \to \Ho(\GCatmpshl)_2(\ca,\cb;\cc)$$
is also a bijection.  

Certainly it is surjective since $\theta_{\ca,\cb} \circ_1 - \colon \GCatmpshl(\ca\leftsharp \cb,\cc) \to\GCatmpshl(\ca,\cb;\cc)$ is so and since substitution respects left-sharp homotopy classes.  For injectivity, consider $F,G \colon \ca,\cb \rightrightarrows \cc$ and a left-sharp homotopy $H$ as depicted below.

\[\begin{tikzcd}[ampersand replacement=\&]
	\&\&\&\& {P\cc} \\
	{\ca,\cb} \&\& {\ca \leftsharp \cb} \&\& \cc
	\arrow["s"', from=1-5, to=2-5]
	\arrow["t", shift left=3, from=1-5, to=2-5]
	\arrow["H", curve={height=-12pt}, from=2-1, to=1-5]
	\arrow["{\theta_{\ca,\cb}}"', from=2-1, to=2-3]
	\arrow["{\exists H'}", dashed, from=2-3, to=1-5]
	\arrow["F", shift left, from=2-3, to=2-5]
	\arrow["G"', shift right, from=2-3, to=2-5]
\end{tikzcd}\]

By the universal property of the left-sharp binary map classifier $\theta_{\ca,\cb}$, there exists a unique Gray-functor $H'$ as depicted such that $H' \circ_1 \theta_{\ca,\cb} = H$.  By uniqueness, it makes the whole diagram serially commute.  In particular, $H'$ is a left-sharp homotopy from $F$ to $G$, as required. 

Therefore $[\theta_{\ca,\cb}] \colon \ca,\cb\to\ca\leftsharp \cb$ is the binary map classifier for the multicategory $\Ho(\GCatmpshl)$, so that $\ca\leftsharp \cb$ is the representing tensor product.

\item Given morphisms $F \colon \ca \to \ca'$ and $G \colon \cb \to \cb'$ in $\GCat$, their tensor product $F \leftsharp G \colon \ca\leftsharp \cb \to \ca' \leftsharp \cb'$ is the unique morphism such that $(F \leftsharp G) \circ_1 \theta_{\ca,\cb}
 = \theta_{\ca',\cb'}(F,G)$.  For cofibrant objects, consider the corresponding morphisms $[F] \colon \ca \to \ca'$ and $[G] \colon \cb \to \cb'$ in $\Ho(\GCat)$.  Since substitution in $\GCatmpshl$ respects left-sharp homotopy, we then have that $[F \leftsharp G] \circ_1 [\theta_{\ca,\cb}]
 = [\theta_{\ca',\cb'}(F,G)] = [\theta_{\ca',\cb'}]([F],[G])$.  By the universal property of the binary map classifier $[\theta_{\ca,\cb}]  \in \Ho(\GCatmpshl)$, $[F \leftsharp G]$ is therefore the tensor product in the monoidal structure representing $\Ho(\GCatmpshl)$.   In summary, the tensor product functor for the representing monoidal structure coincides with that of the derived skew monoidal structure.
%
\item Turning to associators, the associator $\alpha_{\ca,\cb,\cc}\colon (\ca\otimes_\sharp\cb)\otimes_\sharp\cc\to\ca\otimes_\sharp(\cb\otimes_\sharp\cc)$ for the left-sharp skew structure is, by definition, the unique morphism such that 
$$\alpha_{\ca,\cb,\cc}\circ_1(\theta_{\ca\otimes_\sharp\cb,\cc}\circ_1\theta_{\ca,\cb})=\theta_{\ca,\cb\otimes_\sharp\cc}\circ_2 \theta_{\cb,\cc}.$$
Then for $\ca,\cb,\cc$ cofibrant, the associator for the derived skew structure is the left-sharp homotopy class $[\alpha_{\ca,\cb,\cc}]$.  Now since substitution respects left-sharp homotopy, we then have  
$$[\alpha_{\ca,\cb,\cc}]\circ_1([\theta_{\ca\otimes_\sharp\cb,\cc}]\circ_1[\theta_{\ca,\cb}])=[\theta_{\ca,\cb\otimes_\sharp\cc}] \circ_2 [\theta_{\cb,\cc}]$$
which proves that $[\alpha_{\ca,\cb,\cc}]$ is the associator for the representing monoidal structure, since it is the unique morphism satisfying this equation.
\item Similar calculations apply to the units: the core fact is that the left-sharp homotopy class $[u] \colon \diamond \to 1$ of the nullary map classifier is the nullary map classifier in $\Ho(\GCatmpshl)$.  The remainder follows, essentially as above, using the definition of the unitors for $\GCat,\leftsharp,1)$ in terms of the multicategory $\GCatmpshl$ and that substitution respects left-sharp homotopy.
\end{enumerate}
\end{proof}

We have proven, in particular, that the left derived skew monoidal structure $(\Ho(\GCat),\leftsharp,1)$ is monoidal.  In the language of \cite{Bourke2016Skew}, this says that the left-sharp skew monoidal structure $(\GCat,\leftsharp,1)$ is \emph{homotopy monoidal}.  In particular, since the associator for $(\Ho(\GCat),\leftsharp,1)$ is simply obtained by applying the localization functor $\GCat_c \to \Ho(\GCat_c)$ to the associator for the skew monoidal structure on the category of cofibrant objects, this proves the following result, which was promised in Section 7.1 of  \cite{BourkeLobbia}.  

\begin{cor}
Let $\ca,\cb$ and $\cc$ be cofibrant Gray-categories.  Then the associator $\alpha \colon (\ca \leftsharp \cb) \leftsharp \cc \to \ca \leftsharp (\cb \leftsharp \cc)$ for the left-sharp skew monoidal structure $(\GCat,\leftsharp,1)$ is a triequivalence or, equivalently by \Cref{thm:char-equiv-in-GCat_s}, an equivalence in the left-sharp $4$-category $\GCat^4$.
\end{cor}

\appendix 

\section{Axioms for a moderate 4-category}
\label{app:axioms-4-cats}

\black The data for a moderate $4$-category $\ca$ in \Cref{unpacked} specifies the data of a binary map $\circ_{x,y,z} \colon \ca(y,z),\ca(x,y) \rightsquigarrow \ca(x,z) \in \GCatmp^l$ (strict in each variable) and a nullary map $1_x \colon \diamond\rightsquigarrow \ca(x,x) \in \GCatmp^l$.  The axioms for a moderate $4$-category state exactly that this data gives a category enriched in $\GCatmp^l$.  As such, there are three groups of axioms: those asserting that the data forms a binary map plus the associativity equations and the unit equations for an enriched category.  We treat these in turn.

\subsection{The binary map axioms for $\circ_{x,y,z} \colon \ca(y,z),\ca(x,y) \rightsquigarrow \ca(x,z)$}\label{bimapaxioms}
The axioms for $\circ_{x,y,z}$ to be a binary map in $\GCatmp^l$ are the axioms for a \emph{lax binary map} in \cite[Appendix D.1]{BourkeLobbia}, with the two axioms \Cref{ax:bin_D_1a:B} and \Cref{ax:bin_D_A:1b} enhanced to require that the adjoint equivalences $\aexi^A_B$ are identities if $A$ or $B$ is (as required of a \emph{pseudo binary map}).  Our notation for the axioms matches that in \cite[Appendix D.1]{BourkeLobbia} but note that some of the axioms there simplify significantly because we are assuming that  $\circ_{x,y,z}$ is strict in each variable.	\black
	
\begin{equation}\label{ax:4-cat_A:Theta}
	\tag{$A:\Theta$}
\begin{tikzcd}[ampersand replacement=\&]
	{a'B.Ab} \& {Ab'.aB} \\
	\\
	{a'B'.Ab} \& {Ab'.aB'}
	\arrow["{A_B}", from=1-1, to=1-2]
	\arrow[""{name=0, anchor=center, inner sep=0}, "{a'\beta.1}"{description}, curve={height=-24pt}, from=1-1, to=3-1]
	\arrow[""{name=1, anchor=center, inner sep=0}, "{a'\beta'.1}"', curve={height=24pt}, from=1-1, to=3-1]
	\arrow[""{name=2, anchor=center, inner sep=0}, "{1.a\beta}", curve={height=-18pt}, from=1-2, to=3-2]
	\arrow["{A_{B'}}"', from=3-1, to=3-2]
	\arrow["{a'\Theta.1}"'{yshift=0.1cm}, between={0.4}{0.6}, Rightarrow, from=0, to=1]
	\arrow["{A_\beta}"', between={0.4}{0.6}, Rightarrow, from=2, to=0]
\end{tikzcd}
\hspace{0.25cm}
=
\hspace{0.25cm}
\begin{tikzcd}[ampersand replacement=\&]
	{a'B.Ab} \& {Ab'.aB} \\
	\\
	{a'B'.Ab} \& {Ab'.aB'}
	\arrow["{A_B}", from=1-1, to=1-2]
	\arrow[""{name=0, anchor=center, inner sep=0}, "{a'\beta'.1}"', curve={height=24pt}, from=1-1, to=3-1]
	\arrow[""{name=1, anchor=center, inner sep=0}, "{1.a\beta}", curve={height=-24pt}, from=1-2, to=3-2]
	\arrow[""{name=2, anchor=center, inner sep=0}, "{1.a\beta'}"{description}, curve={height=24pt}, from=1-2, to=3-2]
	\arrow["{A_{B'}}"', from=3-1, to=3-2]
	\arrow["{A_{\beta'}}"', between={0.4}{0.6}, Rightarrow, from=2, to=0]
	\arrow["{1.a\Theta}"'{yshift=0.1cm}, between={0.4}{0.6}, Rightarrow, from=1, to=2]
\end{tikzcd}
\end{equation}

\begin{equation}\label{ax:4-cat_Lambda:B}
	\tag{$\Lambda:B$}
\begin{tikzcd}[ampersand replacement=\&]
	{a'B.Ab} \&\& {a'B.A'b} \&\& {a'B.Ab} \&\& {a'B.A'b} \\
	\&\&\& {=} \\
	{Ab'.aB} \&\& {A'b'.aB} \&\& {Ab'.aB} \&\& {A'b'.aB}
	\arrow[""{name=0, anchor=center, inner sep=0}, "{1.\alpha b}", curve={height=-18pt}, from=1-1, to=1-3]
	\arrow["{A_B}"', from=1-1, to=3-1]
	\arrow["{A'_B}", from=1-3, to=3-3]
	\arrow[""{name=1, anchor=center, inner sep=0}, "{1.\alpha b}", curve={height=-18pt}, from=1-5, to=1-7]
	\arrow[""{name=2, anchor=center, inner sep=0}, "{1.\alpha'b}"{description}, curve={height=18pt}, from=1-5, to=1-7]
	\arrow["{A_B}"', from=1-5, to=3-5]
	\arrow["{A'_B}", from=1-7, to=3-7]
	\arrow[""{name=3, anchor=center, inner sep=0}, "{\alpha b'.1}", curve={height=-18pt}, from=3-1, to=3-3]
	\arrow[""{name=4, anchor=center, inner sep=0}, "{\alpha' b'.1}"', curve={height=18pt}, from=3-1, to=3-3]
	\arrow[""{name=5, anchor=center, inner sep=0}, "{\alpha' b'.1}"', curve={height=18pt}, from=3-5, to=3-7]
	\arrow["{\alpha_B}", between={0.3}{0.7}, Rightarrow, nfold, from=0, to=3]
	\arrow["{1.\Lambda b}", between={0.3}{0.7}, Rightarrow, nfold, from=1, to=2]
	\arrow["{\alpha'_B}", between={0.3}{0.7}, Rightarrow, nfold, from=2, to=5]
	\arrow["{\Lambda b'.1}", between={0.3}{0.7}, Rightarrow, nfold, from=3, to=4]
\end{tikzcd}
\end{equation}

\begin{equation}\label{ax:4-cat_A:beta_1,beta_2}
	\tag{$A:\beta_1,\beta_2$}
\begin{tikzcd}[ampersand replacement=\&]
	{a'B.Ab} \&\& {Ab'.aB} \\
	{a'B'.Ab} \&\& {Ab'.aB'} \\
	{a'B''.Ab} \&\& {Ab'.aB''}
	\arrow[""{name=0, anchor=center, inner sep=0}, "{A_B}", from=1-1, to=1-3]
	\arrow["{a'\beta_1.1}"', from=1-1, to=2-1]
	\arrow["{1.a\beta_1}", from=1-3, to=2-3]
	\arrow[""{name=1, anchor=center, inner sep=0}, "{A_{B'}}"{description}, from=2-1, to=2-3]
	\arrow["{a'\beta_2.1}"', from=2-1, to=3-1]
	\arrow["{1.a\beta_2}", from=2-3, to=3-3]
	\arrow[""{name=2, anchor=center, inner sep=0}, "{A_{B''}}"', from=3-1, to=3-3]
	\arrow["{A_{\beta_1}}", between={0.3}{0.7}, Rightarrow, nfold, from=0, to=1]
	\arrow["{A_{\beta_2}}", between={0.3}{0.7}, Rightarrow, nfold, from=1, to=2]
\end{tikzcd}
\hspace{0.25cm}
=
\hspace{0.25cm}
\begin{tikzcd}[ampersand replacement=\&]
	{a'B.Ab} \&\& {Ab'.aB} \\
	{a'B'.Ab} \&\& {Ab'.aB'} \\
	{a'B''.Ab} \&\& {Ab'.aB''}
	\arrow[""{name=0, anchor=center, inner sep=0}, "{A_B}", from=1-1, to=1-3]
	\arrow["{a'\beta_1.1}"', from=1-1, to=2-1]
	\arrow["{1.a\beta_1}", from=1-3, to=2-3]
	\arrow["{a'\beta_2.1}"', from=2-1, to=3-1]
	\arrow["{1.a\beta_2}", from=2-3, to=3-3]
	\arrow[""{name=1, anchor=center, inner sep=0}, "{A_{B''}}"', from=3-1, to=3-3]
	\arrow["{A_{\beta_2 \star \beta_1}}", between={0.3}{0.7}, Rightarrow, nfold, from=0, to=1]
\end{tikzcd}
\end{equation}

\begin{equation}\label{ax:4-cat_alpha_1,alpha_2:B}
	\tag{$\alpha_1,\alpha_2:B$}
\adjustbox{width=\textwidth}{
\begin{tikzcd}[ampersand replacement=\&,row sep=2.25em]
	{a'B.Ab} \& {a'B.A'b} \& {a'B.A''b} \\
	{Ab'.aB} \& {A'b'.aB} \& {A''b'.aB}
	\arrow[""{name=0, anchor=center, inner sep=0}, "{1.\alpha_1 b}", from=1-1, to=1-2]
	\arrow["{A_B}"', from=1-1, to=2-1]
	\arrow[""{name=1, anchor=center, inner sep=0}, "{1.\alpha_2 b}", from=1-2, to=1-3]
	\arrow["{A'_B}"{description}, from=1-2, to=2-2]
	\arrow["{A''_B}", from=1-3, to=2-3]
	\arrow[""{name=2, anchor=center, inner sep=0}, "{\alpha_1 b'.1}"', from=2-1, to=2-2]
	\arrow[""{name=3, anchor=center, inner sep=0}, "{\alpha_2 b'.1}"', from=2-2, to=2-3]
	\arrow["{({\alpha_1})_B}"', shift left=2, between={0.3}{0.7}, Rightarrow, from=0, to=2]
	\arrow["{({\alpha_2})_B}", between={0.3}{0.7}, Rightarrow, from=1, to=3]
\end{tikzcd}
\hspace{0.25cm}
=
\hspace{0.25cm}
\begin{tikzcd}[ampersand replacement=\&,row sep=2.25em]
	{a'B.Ab} \& {a'B.A'b} \& {a'B.A''b} \\
	{Ab'.aB} \& {A'b'.aB} \& {A''b'.aB}
	\arrow["{1.\alpha_1 b}", from=1-1, to=1-2]
	\arrow["{A_B}"', from=1-1, to=2-1]
	\arrow["{1.\alpha_2 b}", from=1-2, to=1-3]
	\arrow["{({\alpha_2\star \alpha_1})_B}", shift right=5, between={0.2}{0.8}, Rightarrow, from=1-2, to=2-2]
	\arrow["{A''_B}", from=1-3, to=2-3]
	\arrow["{\alpha_1 b'.1}"', from=2-1, to=2-2]
	\arrow["{\alpha_2 b'.1}"', from=2-2, to=2-3]
\end{tikzcd}
}
\end{equation}

\begin{equation}\label{ax:4-cat_A_1,A_2:beta}
	\tag{$A_1,A_2:\beta$}
\begin{tikzcd}[ampersand replacement=\&]
	{a''B.A_2b.A_1b} \&\& {A_2b'.A_1b'.aB} \& \\
	\& {A_2b'.a'B.A_1b} \&\& {=} \\
	{a''B'.A_2b.A_1b} \&\& {A_2b'.A_1b'.aB'} \\
	\& {A_2b'.a'B'.A_1b} \\
	{a''B.A_2b.A_1b} \&\& {A_2b'.A_1b'.aB} \\
	\\
	{a''B'.A_2b.A_1b} \&\& {A_2b'.A_1b'.aB'} \\
	\& {A_2b'.a'B'.A_1b}
	\arrow[""{name=0, anchor=center, inner sep=0}, "{(A_2.A_1)_B}", from=1-1, to=1-3]
	\arrow[""{name=1, anchor=center, inner sep=0}, "{(A_2)_B.1}"', from=1-1, to=2-2]
	\arrow["{a''\beta.1}"', from=1-1, to=3-1]
	\arrow["{1.a\beta}", from=1-3, to=3-3]
	\arrow[""{name=2, anchor=center, inner sep=0}, "{1.(A_1)_{B}}"', from=2-2, to=1-3]
	\arrow["{1.a'\beta.1}"{description}, from=2-2, to=4-2]
	\arrow[""{name=3, anchor=center, inner sep=0}, "{(A_2)_{B'}.1}"', from=3-1, to=4-2]
	\arrow[""{name=4, anchor=center, inner sep=0}, "{1.(A_1)_{B'}}"', from=4-2, to=3-3]
	\arrow[""{name=5, anchor=center, inner sep=0}, "{(A_2.A_1)_B}", from=5-1, to=5-3]
	\arrow["{a''\beta.1}"', from=5-1, to=7-1]
	\arrow["{1.a\beta}", from=5-3, to=7-3]
	\arrow[""{name=6, anchor=center, inner sep=0}, "{(A_2.A_1)_{B'}}"{description}, from=7-1, to=7-3]
	\arrow["{(A_2)_B'.1}"', from=7-1, to=8-2]
	\arrow["{1.(A_1)_{B'}}"', from=8-2, to=7-3]
	\arrow["{(A_2,A_1)_B}", shift right=2, between={0.3}{0.9}, Rightarrow, nfold, from=0, to=2-2]
	\arrow["{(A_2)_\beta.1}"', between={0.3}{0.7}, Rightarrow, nfold, from=1, to=3]
	\arrow["{1.(A_1)_{\beta}}", between={0.3}{0.7}, Rightarrow, nfold, from=2, to=4]
	\arrow["{(A_2.A_1)_\beta}"{description}, between={0.2}{0.8}, Rightarrow, nfold, from=5, to=6]
	\arrow["{(A_2,A_1)_{B'}}", between={0.2}{1}, Rightarrow, nfold, from=6, to=8-2]
\end{tikzcd}
\end{equation}

\begin{equation}\label{ax:4-cat_alpha:B_1,B_2}
	\tag{$\alpha:B_1,B_2$}
	\begin{tikzcd}[ampersand replacement=\&,column sep=scriptsize]
		{a'B_2.a'B_1.Ab} \&\& {a'B_2.a'B_1.A'b} \&\& \\
		\& {a'B_2.Ab'.aB_1} \&\& {a'B_2.Ab'.aB_1} \& {=} \\
		{Ab''.aB_2.aB_1} \&\& {A'b''.aB_2.aB_1} \\
		{a'B_2.a'B_1.Ab} \&\& {a'B_2.a'B_1.A'b} \\
		\&\&\& {a'B_2.Ab'.aB_1} \\
		{Ab''.aB_2.aB_1} \&\& {A'b''.aB_2.aB_1}
		\arrow[""{name=0, anchor=center, inner sep=0}, "{1.\alpha b}", from=1-1, to=1-3]
		\arrow["{1.A_{B_1}}"{description}, from=1-1, to=2-2]
		\arrow[""{name=1, anchor=center, inner sep=0}, "{A_{B_2.B_1}}"', from=1-1, to=3-1]
		\arrow["{1.A'_{B_1}}", from=1-3, to=2-4]
		\arrow[""{name=2, anchor=center, inner sep=0}, "{1.\alpha b'.1}"{description}, from=2-2, to=2-4]
		\arrow["{A_{B_2}.1}"{description}, from=2-2, to=3-1]
		\arrow["{A'_{B_2}.1}", from=2-4, to=3-3]
		\arrow[""{name=3, anchor=center, inner sep=0}, "{\alpha b''.1}"', from=3-1, to=3-3]
		\arrow[""{name=4, anchor=center, inner sep=0}, "{1.\alpha b}", from=4-1, to=4-3]
		\arrow["{A_{B_2.B_1}}"', from=4-1, to=6-1]
		\arrow["{1.A'_{B_1}}", from=4-3, to=5-4]
		\arrow[""{name=5, anchor=center, inner sep=0}, "{A'_{B_2.B_1}}"', from=4-3, to=6-3]
		\arrow["{A'_{B_2}.1}", from=5-4, to=6-3]
		\arrow[""{name=6, anchor=center, inner sep=0}, "{\alpha b''.1}"', from=6-1, to=6-3]
		\arrow["{1.\alpha_{B_1}}"{pos=0.6}, between={0.4}{0.6}, Rightarrow, from=0, to=2]
		\arrow["{A_{(B_2,B_1)}}", between={0.3}{0.7}, Rightarrow, from=2-2, to=1]
		\arrow["{\alpha_{B_2}.1}"{pos=0.4}, between={0.4}{0.6}, Rightarrow, from=2, to=3]
		\arrow["{\alpha_{B_2.B_1}}", between={0.3}{0.7}, Rightarrow, from=4, to=6]
		\arrow["{A'_{(B_2,B_1)}}", between={0.2}{0.6}, Rightarrow, from=5-4, to=5]
	\end{tikzcd}
\end{equation}

\begin{equation}\label{ax:4-cat_alpha:beta}
	\tag{$\alpha:\beta$}
	\begin{tikzcd}[ampersand replacement=\&,row sep=scriptsize]
		{a'B.Ab} \& {a'B.A'b} \& \\
		\& {Ab'.aB} \& {A'b'.aB} \\
		{a'B'.Ab} \\
		\& {Ab'.aB'} \& {A'b'.aB'}
		\arrow[""{name=0, anchor=center, inner sep=0}, "{1.\alpha b}", from=1-1, to=1-2]
		\arrow[""{name=1, anchor=center, inner sep=0}, "{A_B}"', from=1-1, to=2-2]
		\arrow["{a'\beta.1}"', from=1-1, to=3-1]
		\arrow["{A'_B}", from=1-2, to=2-3]
		\arrow[""{name=2, anchor=center, inner sep=0}, "{\alpha b'.1}"', from=2-2, to=2-3]
		\arrow["{1.a\beta}"{description}, from=2-2, to=4-2]
		\arrow["{1.a\beta}", from=2-3, to=4-3]
		\arrow[""{name=3, anchor=center, inner sep=0}, "{A_{B'}}"', from=3-1, to=4-2]
		\arrow[""{name=4, anchor=center, inner sep=0}, "{\alpha b'.1}"', from=4-2, to=4-3]
		\arrow["{\alpha_B}"{description}, between={0.3}{0.7}, Rightarrow, from=0, to=2]
		\arrow["{A_\beta}"', between={0.4}{0.6}, Rightarrow, from=1, to=3]
		\arrow[between={0.4}{0.6}, Rightarrow, dashed, from=2, to=4]
	\end{tikzcd}
\hspace{0.25cm}
=
\hspace{0.25cm}
\begin{tikzcd}[ampersand replacement=\&]
	{a'B.Ab} \& {a'B.A'b} \& \\
	\&\& {A'b'.aB} \\
	{a'B'.Ab} \& {a'B'.A'b} \\
	\& {Ab'.aB'} \& {A'b'.aB'}
	\arrow[""{name=0, anchor=center, inner sep=0}, "{1.\alpha b}", from=1-1, to=1-2]
	\arrow["{a'\beta.1}"', from=1-1, to=3-1]
	\arrow[""{name=1, anchor=center, inner sep=0}, "{A'_B}", from=1-2, to=2-3]
	\arrow["{a'\beta.1}"{description}, from=1-2, to=3-2]
	\arrow["{1.a\beta}", from=2-3, to=4-3]
	\arrow[""{name=2, anchor=center, inner sep=0}, "{1.\alpha b}", from=3-1, to=3-2]
	\arrow["{A_{B'}}"', from=3-1, to=4-2]
	\arrow[""{name=3, anchor=center, inner sep=0}, "{A'_{B'}}", from=3-2, to=4-3]
	\arrow[""{name=4, anchor=center, inner sep=0}, "{\alpha b'.1}"', from=4-2, to=4-3]
	\arrow[between={0.4}{0.6}, Rightarrow, nfold, dashed, from=0, to=2]
	\arrow["{A'_\beta}", between={0.4}{0.6}, Rightarrow, nfold, from=1, to=3]
	\arrow["{\alpha_{B'}}", between={0.4}{0.6}, Rightarrow, nfold, from=2, to=4]
\end{tikzcd}	
\end{equation}

\begin{equation}\label{ax:4-cat_A:B_1,B_2,B_3}
	\tag{$A:B_1,B_2,B_3$}
\begin{tikzcd}[ampersand replacement=\&,column sep=small]
	\& {a'B_3.a'B_2.Ab'.aB_1} \&\& \\
	{a'B_3.a'B_2.a'B_1.Ab} \&\& {a'B_3.Ab''.aB_2.aB_1} \\
	\&\&\& {=} \\
	\&\& {Ab'''.aB_3.aB_2.aB_1} \\
	\& {a'B_3.a'B_2.Ab'.aB_1} \\
	{a'B_3.a'B_2.a'B_1.Ab} \&\& {a'B_3.Ab''.aB_2.aB_1} \\
	\\
	\&\& {Ab'''.aB_3.aB_2.aB_1}
	\arrow["{1.A_{B_2}.1}", from=1-2, to=2-3]
	\arrow["{1.A_{B_1}}", from=2-1, to=1-2]
	\arrow[""{name=0, anchor=center, inner sep=0}, "{1.A_{B_2 . B_1}}"{description}, from=2-1, to=2-3]
	\arrow[""{name=1, anchor=center, inner sep=0}, "{A_{B_3. B_2 . B_1}}"', curve={height=12pt}, from=2-1, to=4-3]
	\arrow["{A_{B_3}.1}", from=2-3, to=4-3]
	\arrow["{1.A_{B_2}.1}", from=5-2, to=6-3]
	\arrow[""{name=2, anchor=center, inner sep=0}, "{A_{B_3 . B_2}.1}"{description, pos=0.2}, curve={height=18pt}, from=5-2, to=8-3]
	\arrow["{1.A_{B_1}}", from=6-1, to=5-2]
	\arrow[""{name=3, anchor=center, inner sep=0}, "{A_{B_3. B_2 . B_1}}"', curve={height=12pt}, from=6-1, to=8-3]
	\arrow["{A_{B_3}.1}", from=6-3, to=8-3]
	\arrow["{1.A_{(B_2 , B_1)}}", shift right=2, between={0.2}{0.6}, Rightarrow, from=1-2, to=0]
	\arrow["{A_{(B_3,B_2 . B_1)}}"{pos=0.4}, shift left=5, between={0.3}{0.7}, Rightarrow, from=0, to=1]
	\arrow["{A_{(B_3.B_2,B_1)}}"{pos=0.7}, shift right=5, between={0.4}{0.8}, Rightarrow, from=2, to=3]
	\arrow["{A_{(B_3,B_2)}.1}"{pos=0.4}, between={0.2}{0.6}, Rightarrow, from=6-3, to=2]
\end{tikzcd}
\end{equation}

\begin{equation}\label{ax:4-cat_A_1,A_2,A_3:B}
	\tag{$A_1,A_2,A_3:B$}
\begin{tikzcd}[ampersand replacement=\&]
	{a'''B.A_3 b.A_2 b.A_1 b} \&\&\& \\
	\&\&\& {=} \\
	{A_3 b'.a''B.A_2 b.A_1 b} \&\& {A_3 b'.A_2 b'.A_1 b'.aB} \\
	\& {A_3 b'.A_2 b'.a'B.A_1 b} \\
	{a'''B.A_3 b.A_2 b.A_1 b} \\
	\\
	{A_3 b'.a''B.A_2 b.A_1 b} \&\& {A_3 b'.A_2 b'.A_1 b'.aB} \\
	\& {A_3 b'.A_2 b'.a'B.A_1 b}
	\arrow["{({A_3})_{B}.1}"', from=1-1, to=3-1]
	\arrow[""{name=0, anchor=center, inner sep=0}, "{(A_3. A_2 . A_1)_{B}}", curve={height=-12pt}, from=1-1, to=3-3]
	\arrow[""{name=1, anchor=center, inner sep=0}, "{1.(A_2.A_1)_B}"{description, pos=0.7}, from=3-1, to=3-3]
	\arrow["{1.({A_2})_{B}.1}"', from=3-1, to=4-2]
	\arrow["{1.({A_1})_{B}}"', from=4-2, to=3-3]
	\arrow["{({A_3})_{B}.1}"', from=5-1, to=7-1]
	\arrow[""{name=2, anchor=center, inner sep=0}, "{(A_3. A_2 . A_1)_{B}}", curve={height=-12pt}, from=5-1, to=7-3]
	\arrow[""{name=3, anchor=center, inner sep=0}, "{(A_3.A_2)_B.1}"{description, pos=0.8}, curve={height=-12pt}, from=5-1, to=8-2]
	\arrow["{1.({A_2})_{B}.1}"', from=7-1, to=8-2]
	\arrow["{1.({A_1})_{B}}"', from=8-2, to=7-3]
	\arrow["{(A_3,A_2.A_1)_B}"', between={0.2}{0.6}, Rightarrow, from=0, to=1]
	\arrow["{1.(A_2.A_1)_B}"', shift left=2, between={0.3}{0.9}, Rightarrow, from=1, to=4-2]
	\arrow["{(A_3,A_2)_B.1}"'{pos=0.6}, between={0.2}{0.8}, Rightarrow, from=3, to=7-1]
	\arrow["{(A_3.A_2,A_1)_B}", between={0.3}{0.7}, Rightarrow, from=2, to=3]
\end{tikzcd}
\end{equation}

\begin{equation}\label{ax:4-cat_A:B_1,beta}
	\tag{$A:B_1,\beta$}
\begin{tikzcd}[ampersand replacement=\&]
	\&\& {a'B_2.Ab'.aB_1} \&\& \\
	{a'B_2.a'B_1.Ab} \&\&\&\& {Ab''.aB_2.aB_1} \\
	\\
	{a'B'_2.a'B_1.Ab} \&\&\&\& {Ab''.aB'_2.aB_1} \\
	\&\& {=} \\
	\&\& {a'B_2.Ab'.aB_1} \\
	{a'B_2.a'B_1.Ab} \&\&\&\& {Ab''.aB_2.aB_1} \\
	\&\& {a'B'_2.Ab'.aB_1} \\
	{a'B'_2.a'B_1.Ab} \&\&\&\& {Ab''.aB'_2.aB_1}
	\arrow["{A_{B_2}.1}", from=1-3, to=2-5]
	\arrow["{1.A_{B_1}}", from=2-1, to=1-3]
	\arrow[""{name=0, anchor=center, inner sep=0}, "{A_{B_2.B_1}}"{description}, from=2-1, to=2-5]
	\arrow["{a'\beta.1.1=a'(\beta.B_1).1}"', from=2-1, to=4-1]
	\arrow["{1.a(\beta.B_1)=1.a\beta.1}", from=2-5, to=4-5]
	\arrow[""{name=1, anchor=center, inner sep=0}, "{A_{B'_2.B_1}}"', from=4-1, to=4-5]
	\arrow[""{name=2, anchor=center, inner sep=0}, "{A_{B_2}.1}", from=6-3, to=7-5]
	\arrow["{a'\beta.1.1}", from=6-3, to=8-3]
	\arrow[""{name=3, anchor=center, inner sep=0}, "{1.A_{B_1}}", from=7-1, to=6-3]
	\arrow["{a'\beta.1.1=a'(\beta.B_1).1}"', from=7-1, to=9-1]
	\arrow["{1.a(\beta.B_1)=1.a\beta.1}", from=7-5, to=9-5]
	\arrow[""{name=4, anchor=center, inner sep=0}, "{A_{B'_2}.1}", from=8-3, to=9-5]
	\arrow[""{name=5, anchor=center, inner sep=0}, "{1.A_{B_1}}", from=9-1, to=8-3]
	\arrow[""{name=6, anchor=center, inner sep=0}, "{A_{B'_2.B_1}}"', from=9-1, to=9-5]
	\arrow["{A_{(B_2,B_1)}}"{pos=0.4}, between={0.1}{0.7}, Rightarrow, nfold, from=1-3, to=0]
	\arrow["{A_{\beta.B_1}}", between={0.3}{0.7}, Rightarrow, nfold, from=0, to=1]
	\arrow["{A_\beta.1}", between={0.2}{0.7}, Rightarrow, nfold, from=2, to=4]
	\arrow[between={0.2}{0.7}, Rightarrow, nfold, dashed, from=3, to=5]
	\arrow["{A_{(B'_2,B_1)}}", between={0}{0.8}, Rightarrow, nfold, from=8-3, to=6]
\end{tikzcd}%
\end{equation}

\begin{equation}\label{ax:4-cat_A:beta,B_2}
	\tag{$A:\beta,B_2$}
\begin{tikzcd}[ampersand replacement=\&]
	\&\& {a'B_2.Ab'.aB_1} \&\& \\
	{a'B_2.a'B_1.Ab} \&\&\&\& {Ab''.aB_2.aB_1} \\
	\\
	{a'B_2.a'B'_1.Ab} \&\&\&\& {Ab''.aB_2.aB'_1} \\
	\&\& {=} \\
	\&\& {a'B_2.Ab'.aB_1} \\
	{a'B_2.a'B_1.Ab} \&\&\&\& {Ab''.aB_2.aB_1} \\
	\&\& {a'B_2.Ab'.aB'_1} \\
	{a'B_2.a'B'_1.Ab} \&\&\&\& {Ab''.aB_2.aB'_1}
	\arrow["{A_{B_2}.1}", from=1-3, to=2-5]
	\arrow["{1.A_{B_1}}", from=2-1, to=1-3]
	\arrow[""{name=0, anchor=center, inner sep=0}, "{A_{B_2.B_1}}"{description}, from=2-1, to=2-5]
	\arrow["{1.a'\beta.1=a'(B_2.\beta).1}"', from=2-1, to=4-1]
	\arrow["{1.a(B_2.\beta)=1.1.a\beta}", from=2-5, to=4-5]
	\arrow[""{name=1, anchor=center, inner sep=0}, "{A_{B_2.B'_1}}"', from=4-1, to=4-5]
	\arrow[""{name=2, anchor=center, inner sep=0}, "{A_{B_2}.1}", from=6-3, to=7-5]
	\arrow["{1.1.a\beta}", from=6-3, to=8-3]
	\arrow[""{name=3, anchor=center, inner sep=0}, "{1.A_{B_1}}", from=7-1, to=6-3]
	\arrow["{1.a'\beta.1=a'(B_2.\beta).1}"', from=7-1, to=9-1]
	\arrow["{1.a(B_2.\beta)=1.1.a\beta}", from=7-5, to=9-5]
	\arrow[""{name=4, anchor=center, inner sep=0}, "{A_{B_2}.1}", from=8-3, to=9-5]
	\arrow[""{name=5, anchor=center, inner sep=0}, "{1.A_{B'_1}}", from=9-1, to=8-3]
	\arrow[""{name=6, anchor=center, inner sep=0}, "{A_{B_2.B'_1}}"', from=9-1, to=9-5]
	\arrow["{A_{(B_2,B_1)}}"{pos=0.4}, between={0.1}{0.7}, Rightarrow, nfold, from=1-3, to=0]
	\arrow["{A_{B_2.\beta}}", between={0.3}{0.7}, Rightarrow, nfold, from=0, to=1]
	\arrow[between={0.2}{0.7}, Rightarrow, nfold, dashed, from=2, to=4]
	\arrow["{1.A_\beta}", between={0.2}{0.7}, Rightarrow, nfold, from=3, to=5]
	\arrow["{A_{(B_2,B'_1)}}", between={0}{0.8}, Rightarrow, nfold, from=8-3, to=6]
\end{tikzcd}
\end{equation}

\begin{equation}\label{ax:4-cat_A_1,alpha:B}
	\tag{$A_1,\alpha:B$}
	\begin{tikzcd}
		& {a''B.A_2b.A_1b} && {a''B.A'_2b.A_1b} & \\
		{A_2b'.a'B.A_1b} &&&& {=} \\
		& {A_2b'.A_1b'.aB} && {A'_2b'.A_1b'.aB} \\
		& {a''B.A_2b.A_1b} && {a''B.A'_2b.A_1b} \\
		{A_2b'.a'B.A_1b} && {A'_2b'.a'B.A_1b} \\
		& {A_2b'.A_1b'.aB} && {A'_2b'.A_1b'.aB}
		\arrow["{1.(\alpha.A_1)b}", from=1-2, to=1-4]
		\arrow["{(A_2)_B.1}"', from=1-2, to=2-1]
		\arrow[""{name=0, anchor=center, inner sep=0}, "{(A_2.A_1)_B}", from=1-2, to=3-2]
		\arrow[""{name=1, anchor=center, inner sep=0}, "{(A'_2.A_1)_B}", from=1-4, to=3-4]
		\arrow["{1.(A_1)_B}"', from=2-1, to=3-2]
		\arrow["{(\alpha.A_1)b'.1}"', from=3-2, to=3-4]
		\arrow[""{name=2, anchor=center, inner sep=0}, "{1.(\alpha.A_1)b}", from=4-2, to=4-4]
		\arrow["{(A_2)_B.1}"', from=4-2, to=5-1]
		\arrow["{(A'_2)_B.1}"'{pos=0.7}, from=4-4, to=5-3]
		\arrow[""{name=3, anchor=center, inner sep=0}, "{(A'_2.A_1)_B}", from=4-4, to=6-4]
		\arrow[""{name=4, anchor=center, inner sep=0}, "{(\alpha b').1}"{description}, from=5-1, to=5-3]
		\arrow["{1.(A_1)_B}"', from=5-1, to=6-2]
		\arrow["{1.(A_{1})_B}"'{pos=0.1}, from=5-3, to=6-4]
		\arrow[""{name=5, anchor=center, inner sep=0}, "{(\alpha.A_1)b'.1}"', from=6-2, to=6-4]
		\arrow["{(A_2,A_1)_B}"', between={0.3}{0.8}, Rightarrow, from=0, to=2-1]
		\arrow["{(\alpha.A_1)_B}"', between={0.2}{0.6}, Rightarrow, from=1, to=0]
		\arrow["{\alpha_B.1}"'{pos=0.7}, between={0.2}{0.8}, Rightarrow, from=2, to=4]
		\arrow["{(A'_2,A_1)_B}"', shift left, between={0.2}{0.8}, Rightarrow, from=3, to=5-3]
		\arrow[between={0.2}{0.8}, Rightarrow, dashed, from=4, to=5]
	\end{tikzcd}
\end{equation}

\begin{equation}\label{ax:4-cat_alpha,A_2:B}
	\tag{$\alpha,A_2:B$}
\begin{tikzcd}[ampersand replacement=\&]
	\& {a''B.A_2b.A_1b} \&\& {a''B.A_2b.A'_1b} \& \\
	{A_2b'.a'B.A_1b} \&\&\&\& {=} \\
	\& {A_2b'.A_1b'.aB} \&\& {A_2b'.A'_1b'.aB} \\
	\& {a''B.A_2b.A_1b} \&\& {a''B.A_2b.A'_1b} \\
	{A_2b'.a'B.A_1b} \&\& {A_2b'.a'B.A'_1b} \\
	\& {A_2b'.A_1b'.aB} \&\& {A'_2b'.A'_1b'.aB}
	\arrow["{1.(A_2.\alpha)b}", from=1-2, to=1-4]
	\arrow["{(A_2)_B.1}"', from=1-2, to=2-1]
	\arrow[""{name=0, anchor=center, inner sep=0}, "{(A_2.A_1)_B}", from=1-2, to=3-2]
	\arrow[""{name=1, anchor=center, inner sep=0}, "{(A_2.A'_1)_B}", from=1-4, to=3-4]
	\arrow["{1.(A_1)_B}"', from=2-1, to=3-2]
	\arrow["{(A_2.\alpha)b'.1}"', from=3-2, to=3-4]
	\arrow[""{name=2, anchor=center, inner sep=0}, "{1.(A_2.\alpha)b}", from=4-2, to=4-4]
	\arrow["{(A_2)_B.1}"', from=4-2, to=5-1]
	\arrow["{(A_2)_B.1}"'{pos=0.7}, from=4-4, to=5-3]
	\arrow[""{name=3, anchor=center, inner sep=0}, "{(A_2.A'_1)_B}", from=4-4, to=6-4]
	\arrow[""{name=4, anchor=center, inner sep=0}, "{(\alpha b').1}"{description}, from=5-1, to=5-3]
	\arrow["{1.(A_1)_B}"'{pos=0.4}, from=5-1, to=6-2]
	\arrow["{1.(A'_1)_B}"'{pos=0.2}, from=5-3, to=6-4]
	\arrow[""{name=5, anchor=center, inner sep=0}, "{(A_2.\alpha)b'.1}"', from=6-2, to=6-4]
	\arrow["{(A_2,A_1)_B}"', between={0.3}{0.8}, Rightarrow, nfold, from=0, to=2-1]
	\arrow["{(A_2.\alpha)_B}"', between={0.2}{0.6}, Rightarrow, nfold, from=1, to=0]
	\arrow[between={0.2}{0.8}, Rightarrow, nfold, dashed, from=2, to=4]
	\arrow["{(A_2,A'_1)_B}"', shift left, between={0.2}{0.8}, Rightarrow, nfold, from=3, to=5-3]
	\arrow["{1.\alpha_B}"'{pos=0.3}, between={0.2}{0.8}, Rightarrow, nfold, from=4, to=5]
\end{tikzcd}
\end{equation}

\begin{equation}\label{ax:4-cat_A_1,A_2:B_1,B_2}
	\tag{$A_1,A_2:B_1,B_2$}
	\begin{split}
	\begin{tikzcd}[ampersand replacement=\&]
		\& {a''B_2.a''B_1.A_2b.A_1b} \&\& \\
		\\
		\\
		\& {a''B_2.A_2b'.a'B_1.A_1b} \\
		{A_2b''.a'B_2.a'B_1.A_1b} \&\& {a''B_2.A_2b'.A_1b'.aB_1} \& {=} \\
		\& {A_2b''.a'B_2.A_1b'.aB_1} \\
		\\
		\\
		\& {A_2b''.A_1b''.aB_2.aB_1}
		\arrow["{1.{(A_2)}_{B_1}.1}"{description}, from=1-2, to=4-2]
		\arrow[""{name=0, anchor=center, inner sep=0}, "{{A_2}_{(B_2.B_1)}.1}"', curve={height=35pt}, from=1-2, to=5-1]
		\arrow[""{name=1, anchor=center, inner sep=0}, "{1.(A_2.A_1)_{B_1}}", curve={height=-35pt}, from=1-2, to=5-3]
		\arrow["{{(A_2)}_{B_2}.1}"{description}, from=4-2, to=5-1]
		\arrow["{1.{(A_1)}_{B_1}}"{description}, from=4-2, to=5-3]
		\arrow["{1.{(A_1)}_{B_1}}"{description}, from=5-1, to=6-2]
		\arrow[""{name=2, anchor=center, inner sep=0}, "{1.{A_1}_{(B_2.B_1)}}"', curve={height=35pt}, from=5-1, to=9-2]
		\arrow[between={0.4}{0.6}, Rightarrow, dashed, from=5-3, to=5-1]
		\arrow["{{(A_2)}_{B_2}.1}"{description}, from=5-3, to=6-2]
		\arrow[""{name=3, anchor=center, inner sep=0}, "{(A_2.A_1)_{B_2}.1}", curve={height=-35pt}, from=5-3, to=9-2]
		\arrow["{1.{(A_1)}_{B_2}.1}"{description}, from=6-2, to=9-2]
		\arrow["{1.(A_2,A_1)_{B_1}}"'{pos=0.4}, shift left=5, between={0.3}{0.5}, Rightarrow, from=1, to=4-2]
		\arrow["{{A_2}_{(B_2,B_1)}.1}"'{pos=0.6}, shift left=5, between={0.5}{0.7}, Rightarrow, from=4-2, to=0]
		\arrow["{(A_2,A_1)_{B_2}.1}"'{pos=0.6}, shift left=3, between={0.4}{0.6}, Rightarrow, from=3, to=6-2]
		\arrow["{1.{A_1}_{(B_2,B_1)}}"'{pos=0.3}, shift left=3, between={0.4}{0.6}, Rightarrow, from=6-2, to=2]
	\end{tikzcd} \\
	\begin{tikzcd}[ampersand replacement=\&]
		\& {a''B_2.a''B_1.A_2b.A_1b} \& \\
		\& {} \& {} \\
		{A_2b''.a'B_2.a'B_1.A_1b} \&\& {a''B_2.A_2b'.A_1b'.aB_1} \\
		{} \& {} \\
		\& {A_2b''.A_1b''.aB_2.aB_1}
		\arrow["{{A_2}_{(B_2.B_1)}.1}"'{pos=0.6}, curve={height=25pt}, from=1-2, to=3-1]
		\arrow["{1.(A_2.A_1)_{B_1}}"{pos=0.6}, curve={height=-25pt}, from=1-2, to=3-3]
		\arrow["{(A_2.A_1)_{(B_2.B_1)}}"{description}, from=1-2, to=5-2]
		\arrow["{{A_2.A_1}_{(B_2,B_1)}}"'{pos=0.6,yshift=0.1cm}, shift left=5, between={0.6}{0.8}, Rightarrow, from=2-3, to=2-2]
		\arrow["{1.{A_1}_{(B_2.B_1)}}"'{pos=0.4}, curve={height=25pt}, from=3-1, to=5-2]
		\arrow["{(A_2.A_1)_{B_2}.1}"{pos=0.4}, curve={height=-25pt}, from=3-3, to=5-2]
		\arrow["{{(A_2,A_1)}_{B_2.B_1}}"{pos=0.4,yshift=-0.1cm}, shift right=5, between={0.3}{0.5}, Rightarrow, from=4-2, to=4-1]
	\end{tikzcd}
\end{split}
\end{equation}

%
%
%
%
%
%
%
%
%
%
%
%
%
%
%
%
%
%
%

\subsubsection{Degeneracy axioms}\label{bimapdegeneracy}\begin{footnote}{The axioms \eqref{ax:bin_D_1A:B} and \eqref{ax:bin_D_A:1B} are in fact redundant --- for the slightly more restrictive Gray-4-categories of \cite{Fukuda}, this was noted in \cite[Lemma 2.10]{Fukuda}.  It is a special case of the fact that a function between monoids which preserves the multiplication and sends the unit to an invertible element also preserves the unit.  Since the argument depends heavily upon invertibility of the $\alpha_B$ and $A_\beta$, the axioms are not redundant in the lax setting of \Cref{laxsetting} nor would they be redundant in higher-dimensional generalisations.  As such, we prefer to include them.}\end{footnote}

\hfill \break

\begin{minipage}{0.5\textwidth}
	
	\begin{equation}\label{ax:bin_D_1a:B}
         \tag{D-$1_a:B$} \aexi^{1_a}_{B}=1^{\aeq}_{aB} 
	\end{equation}
	
	\begin{equation}\label{ax:bin_D_1A:B}
	\tag{D-$1_A:B$} {(1_A)_B}=1_{A_B}
	\end{equation}
	
	\begin{equation}\label{ax:bin_D_1a:bet}
	\tag{D-$1_a:\beta$} (1_a)_\beta=1_{a\beta} 
	\end{equation}
	
	\begin{equation}\label{ax:bin_D_A1A2:B}
		\tag{D-$A_1,A_2:B$}
		(1_{a'},A)_B=(A,1_a)_B=1_{A_B}
	\end{equation} 
	
	\begin{equation}\label{ax:bin_D_A1A2:1b}
		\tag{D-$A_1,A_2:1_b$}
		(A_2,A_1)_{1_b}=1_{1_{A_2b.A_1b}}
	\end{equation} 
\end{minipage}
\begin{minipage}{0.5\textwidth}
	
	\begin{equation}\label{ax:bin_D_A:1b}
        \tag{D-$A:1_b$} \aexi^A_{1_b}=1^{\aeq}_{Ab} 
	\end{equation}

	\begin{equation}\label{ax:bin_D_A:1B}
	\tag{D-$A:1_B$} {A_{1_B}}=1_{A_B}
	\end{equation}
	
	\begin{equation}\label{ax:bin_D_alph:1b}
		\tag{D-$\alpha:1_b$} \alpha_{1_b}=1_{\alpha b} 
	\end{equation}

	\begin{equation}\label{ax:bin_D_A:B1B2}
		\tag{D-$A:B_1,B_2$}
		A_{(1_{b'},B)}=A_{(B,1_b)}=1_{A_B}
	\end{equation}
	
	\begin{equation}\label{ax:bin_D_1a:B1B2}
		\tag{D-$1_a:B_1,B_2$}
		{(1_a)}_{(B_2,B_1)}=1_{1_{aB_2.aB_1}}
	\end{equation}
\end{minipage}

\subsection{Associativity} \label{bimapassociativity}
For objects $w,x,y,z\in\ca$ we require that the diagram
\[\begin{tikzcd}[ampersand replacement=\&]
	{\ca(y,z),\ca(x,y),\ca(w,x)} \&\& {\ca(x,z),\ca(w,x)} \\
	{\ca(y,z),\ca(w,y)} \&\& {\ca(w,z)}
	\arrow["{\circ,1}", from=1-1, to=1-3]
	\arrow["{1,\circ}"', from=1-1, to=2-1]
	\arrow["{\circ }", from=1-3, to=2-3]
	\arrow["\circ"', from=2-1, to=2-3]
\end{tikzcd}\]
commutes in $\GCatmp^l$.  As an equality of loose ternary maps, this will hold just when the composite ternary maps have the same binary map components and the same incubators.  
\begin{enumerate}
	\item[\textbf{(AssBin)}] The binary map components coincide if, for any $a,A,\alpha\in\ca(y,z)$, $b,B,\beta\in\ca(x,y)$, and $c,C,\gamma\in\ca(w,x)$, the following equations hold.
	\begin{equation}
		\tag{AssBin:1}
		\begin{split}
			(a\circ-)\circ c= a\circ(-\circ c), \quad (a\circ b)\circ-=a\circ(b\circ-), \quad 
			(-\circ b)\circ c=-\circ(b\circ c)
		\end{split}
	\end{equation}
	\vspace{0.1cm}
	\begin{equation}
		\tag{AssBin:2}
		\begin{split}
	\black	 a \circ \aexi^{B_C}=\aexi^{aB}_C, \quad \aexi^{A}_{bC}= \aexi^{Ab}_C, \quad
			 \aexi^{A}_{Bc}= \aexi^{A_B} \circ c
		\end{split}
	\end{equation}
	\vspace{0.1cm}
	\begin{equation}
		\tag{AssBin:3}
		\begin{alignedat}{2}
			a(B_{(C_2,C_1)})
			&= (aB)_{(C_2,C_1)},
			&\qquad a(\,(B_2,B_1)_C\,)
			&= (aB_2,aB_1)_C
			\\[0.1cm]
			A_{(bC_2,bC_1)}
			&= (Ab)_{(C_2,C_1)},
			&\qquad (A_2b,A_1b)_C
			&= (A_2,A_1)_{bC}
			\\[0.1cm]
			(A_{(B_2,B_1)})c
			&= A_{(B_2c,B_1c)},
			&\qquad (\,(A_2,A_1)_B\,)c
			&= (A_2,A_1)_{Bc}
		\end{alignedat}
	\end{equation}	
	\vspace{0.1cm}
\begin{equation}
	\tag{AssBin:4}
	\begin{alignedat}{2}
		(aB)_\gamma
		&= a(B_\gamma),
		&\qquad a(\beta_C)
		&= (a\beta)_C
		\\[0.1cm]
		(Ab)_\gamma
		&= A_{b\gamma},
		&\qquad (\alpha b)_C
		&= \alpha_{bC}
		\\[0.1cm]
		(A_\beta)c
		&= A_{\beta c},
		&\qquad (\alpha_B)c
		&= \alpha_{Bc}
	\end{alignedat}
\end{equation}
	
	\item[\textbf{(AssInc)}] The last associativity axioms says that the incubators of the two composites coincide. Let us notice that we will use some of the above equations in the outer part of the two diagrams below. For any 
	\begin{tikzcd}[ampersand replacement=\&]
		w \& x \& y \& z
		\arrow[""{name=0, anchor=center, inner sep=0}, "c", curve={height=-12pt}, from=1-1, to=1-2]
		\arrow[""{name=1, anchor=center, inner sep=0}, "{c'}"', curve={height=12pt}, from=1-1, to=1-2]
		\arrow[""{name=2, anchor=center, inner sep=0}, "b", curve={height=-12pt}, from=1-2, to=1-3]
		\arrow[""{name=3, anchor=center, inner sep=0}, "{b'}"', curve={height=12pt}, from=1-2, to=1-3]
		\arrow[""{name=4, anchor=center, inner sep=0}, "a", curve={height=-12pt}, from=1-3, to=1-4]
		\arrow[""{name=5, anchor=center, inner sep=0}, "{a'}"', curve={height=12pt}, from=1-3, to=1-4]
		\arrow["C", shift right, between={0.3}{0.7}, Rightarrow, from=0, to=1]
		\arrow["B", shift right, between={0.3}{0.7}, Rightarrow, from=2, to=3]
		\arrow["A", shift right, between={0.3}{0.7}, Rightarrow, from=4, to=5]
	\end{tikzcd},
	
	\begin{equation}
		\label{ax:4-cat_A_B_C}
		\tag{$A:B:C$}
		\begin{split}
		\begin{tikzcd}[ampersand replacement=\&]
			\&\& {a'C.Ab'c.aBc} \&\&\& \\
			{a'b'C.a'Bc.Abc} \&\&\&\& {Ab'c'.ab'C.aBc} \\
			\&\&\&\&\& {=} \\
			{aBc'.a'bC.Abc} \&\&\&\& {Ab'c'.aBc'.abC} \\
			\&\& {a'Bc'.Abc'.abC}
			\arrow["{(Ab')_C.1}", from=1-3, to=2-5]
			\arrow[""{name=0, anchor=center, inner sep=0}, "{(Ab'.aB)_C}"{description, pos=0.8}, shift right=2, from=1-3, to=4-5]
			\arrow["{1.(A_B)c}", from=2-1, to=1-3]
			\arrow["{(a'B)_C.1}"', from=2-1, to=4-1]
			\arrow[""{name=1, anchor=center, inner sep=0}, "{(a'B.Ab)_C}"{description, pos=0.3}, shift left, from=2-1, to=5-3]
			\arrow["{1.(aB)_C}", from=2-5, to=4-5]
			\arrow["{1.(Ab)_C}"', from=4-1, to=5-3]
			\arrow["{(A_B)c'.1}"', from=5-3, to=4-5]
			\arrow["{{(A_B)}_C}"', between={0.4}{0.6}, Rightarrow, from=0, to=1]
			\arrow["{(a'B,Ab)_C}"{pos=0.65}, shift left=3, between={0.35}{0.65}, Rightarrow, from=1, to=4-1]
			\arrow["{{(Ab',aB)_C}^{\scalebox{0.75}{-1}}}"'{pos=0.25}, shift right=2, between={0.4}{0.6}, Rightarrow, from=2-5, to=0]
		\end{tikzcd}\vspace{0.25cm}\\
		\begin{tikzcd}[ampersand replacement=\&]
			\&\& {a'C.Ab'c.aBc} \&\& \\
			{a'b'C.a'Bc.Abc} \&\&\&\& {Ab'c'.ab'C.aBc} \\
			\\
			{aBc'.a'bC.Abc} \&\&\&\& {Ab'c'.aBc'.abC} \\
			\&\& {a'Bc'.Abc'.abC}
			\arrow["{A_{(b'C)}.1=(Ab')_C.1}", from=1-3, to=2-5]
			\arrow["{1.(A_B)c=1.A_{Bc}}", from=2-1, to=1-3]
			\arrow[""{name=0, anchor=center, inner sep=0}, "{A_{(b'C.Bc)}}"', from=2-1, to=2-5]
			\arrow["{(a'B)_C.1=a'(B_C).1}"', from=2-1, to=4-1]
			\arrow["{1.a(B_C)=1.(aB)_C}", from=2-5, to=4-5]
			\arrow[""{name=1, anchor=center, inner sep=0}, "{{A_{(Bc'.bC)}}}", from=4-1, to=4-5]
			\arrow["{1.(Ab)_C=1.A_{bC}}"', from=4-1, to=5-3]
			\arrow["{A_{Bc'}.1=(A_B)c'.1}"', from=5-3, to=4-5]
			\arrow["{A_{(b'C,Bc)}}"{xshift=0.1cm}, between={0.35}{0.65}, Rightarrow, from=1-3, to=0]
			\arrow["{A_{(B_C)}}"{xshift=0.1cm}, between={0.4}{0.6}, Rightarrow, from=0, to=1]
			\arrow["{{A_{(Bc',bC)}}^{\scalebox{0.75}{-1}}}"{xshift=0.1cm}, between={0.35}{0.65}, Rightarrow, from=1, to=5-3]
		\end{tikzcd}
	\end{split}
	\end{equation}

\end{enumerate}

\subsection{Unitality}\label{bimapunit}
We require also that the binary map $\circ_{x,y,z} \colon \ca(y,z),\ca(x,y) \rightsquigarrow \ca(x,z)$ is unital, meaning that for any $x,y\in\ca$, $1_y\circ-=1_{\ca(x,y)}=-\circ1_x$.
\black

\bibliographystyle{plain}

\end{document}